\documentclass{article}
\usepackage{graphicx} 
\usepackage{amsfonts}
\usepackage{amsmath}
\usepackage{amssymb}
\usepackage{amsthm}
\usepackage{geometry}
\usepackage{color}
\usepackage{comment}
\newtheorem{theorem}{Theorem}[section]
\newtheorem{proposition}[theorem]{Proposition}
\newtheorem{lemma}[theorem]{Lemma}
\newtheorem{remark}[theorem]{Remark}
\newtheorem{definition}[theorem]{Definition}
\newtheorem{assumption}[theorem]{Assumption}
\newtheorem{corollary}[theorem]{Corollary}

\usepackage[colorlinks=true, linkcolor=blue, anchorcolor=blue, citecolor=blue]{hyperref}
\DeclareMathOperator{\E}{\mathbb{E}}
\DeclareMathOperator{\R}{\mathbb{R}}
\DeclareMathOperator{\tr}{tr}
\DeclareMathOperator{\divergence}{div}

\title{Master Equations for Mean Field Game of Controls:\\
A Unification of Weak Solution Notions}

\author{Mengzhen Li\thanks{School of Mathematics, Shandong University, Jinan, 250100, China.\\Department of Mathematics, City University of Hong Kong, Hong Kong SAR, China. 
  (mengzhli@cityu.edu.hk).}
\and Xintian Liu\thanks{School of Mathematical Sciences, Dalian University of Technology, Dalian, Liaoning, 116024, China.
  (xtliu@dlut.edu.cn).}
\and Chenchen Mou\thanks{Department of Mathematics, City University of Hong Kong, Hong Kong SAR, China.
  (chencmou@cityu.edu.hk).}
\and Zhen Wu\thanks{School of Mathematics, Shandong University, Jinan, 250100, China.\\State Key Laboratory of Cryptography and Digital Economy Security, Shandong University, Jinan, 250100, China.
  (wuzhen@sdu.edu.cn).}}

\begin{document}
\date{}
\maketitle

\begin{abstract}
    In this manuscript, we study the master equation for mean field game of controls, assuming only
that the data are Lipschitz continuous in the measure variable. Accordingly, we propose a weaker
notion of solution called weak solutions for the master equation. We establish the global well-posedness of the master equation for such solution under the Lasry-Lions monotonicity and displacement $\lambda$-monotonicity conditions, respectively. The arguments rely on the analysis of the associated Pontryagin forward-backward stochastic differential equation system, especially the stability in the measure variable. Finally, we review several existing notions of non-smooth solution from the literature, namely good solutions, weak viscosity solutions, Lipschitz solutions, monotone solutions, and examine their relationship with our weak solution. We show that under appropriate assumptions all of them are equivalent by their definitions, and thus our manuscript provides a unification of weak solution
notions for the master equations derived from mean field game of controls.
\end{abstract}


\noindent\textbf{MSCcodes. }35Q89, 49N80, 60H30, 91A16, 93E20.

\section{Introduction}
\subsection{Background}
The mean field games (MFG) theory was introduced by Lasry-Lions \cite{lasry2006jeux-i,lasry2006jeux-ii,lasry2007mean} and Huang-Caines-Malham\'{e} \cite{huang2007large, huang2006large} independently. In a classical MFG model, the agents interact through the law of the state, which is known as the mean field term. To generalize the model for applications \cite{cardaliaguet2018trade,carmona2015weak}, the mean field game of controls (MFGC) allow interactions through the joint law of the state and the control. In this manuscript, we shall study the following master equation for MFGC:
\begin{equation}
\label{introduction master}
\begin{gathered}
    \mathcal{L}V(t,x,\mu):=\partial_t V+\frac{\widehat{\sigma}^2}{2}\tr(\partial_{xx}V)+H(x,\partial_x V,\rho)+\mathcal{M}V=0,\\
    V(T,x,\mu)=G(x,\mu),\quad\text{where}\\
    \mathcal{M}V(t,x,\mu):=\tr\Big(\bar{\widetilde{\E}}\Big[\frac{\widehat{\sigma}^2}{2}\partial_{\tilde{x}\mu}V(t,x,\mu,\tilde{\xi})+\sigma_0^2\partial_{x\mu}V(t,x,\mu,\tilde{\xi})+\frac{\sigma_0^2}{2}\partial_{\mu\mu}V(t,x,\mu,\bar{\xi},\tilde{\xi})\\
    +\partial_{\mu}V(t,x,\mu,\tilde{\xi})\cdot\partial_p H\big(\tilde{\xi},\partial_x V(t,\tilde{\xi},\mu),\rho\big)\Big]\Big),\quad \rho=\big(id,\partial_p H(\cdot,\partial_x V(t,\cdot,\mu),\rho)\big)\#\mu.
\end{gathered}
\end{equation}
Here $\widehat{\sigma}^2:=\sigma^2+\sigma_0^2$, where $\sigma,\,\sigma_0$ represent the intensity of idiosyncratic noise and common noise respectively; the derivative with respect to the measure variable should be understood in the sense of Lions derivative; $\tilde{\xi},\bar{\xi}$ are independent random variables with law $\mu$, and $\bar{\widetilde{\E}}$ is the expectation on $\tilde{\xi}$ and $\bar{\xi}$. We emphasize that the master equation includes probability measure $\mu$ as an infinite-dimensional variable, and non-local term $\mathcal{M}V$, which brings more complicated structure.
\par There is abundant literature on classical solutions to master equation for classical MFG in the past years, whose interaction is only through the law of the state. We refer to \cite{bensoussan2019control,cardaliaguet2023splitting,gangbo2015existence} for the local well-posedness of the MFG master equation. To study the global well-posedness, the monotonicity conditions are required as crucial structural assumptions. Initially, the global well-posedness results for the MFG master equation were established under Lasry-Lions monotonicity conditions. In \cite{cardaliaguet2019master}, the authors studied the master equation with non-degenerate idiosyncratic noise and separable Hamiltonian under Lasry-Lions monotonicity condition. We refer to \cite{bertucci2024discrete,bertucci2019remark,carmona2018probabilistic,chassagneux2022probabilistic} for more works under Lasry-Lions monotonicity conditions. Another typical monotonicity condition is called displacement monotonicity, which is introduced in \cite{ahuja2016weak}. In \cite{gangbo2022displacement}, the authors proved the global well-posedness for the MFG master equation with non-separable Hamiltonian under displacement monotonicity condition. The displacement monotonicity can be further weakened to the displacement $\lambda$-monotonicity; see e.g. \cite{mou2022propagation}. We refer to \cite{bansil2025global,gangbo2022potential} for more works under displacement monotonicity. Relatedly, \cite{cirant2026apriori} derived estimates for $N$-player Nash systems under Lasry-Lions or displacement semimonotonicity conditions uniformly in the number of players $N$, with which the convergenve problem was discussed as applications. We mention that there are also other types of monotonicity conditions raised in recent years; see e.g. \cite{graber2023monotonicity,mou2025anti}. 
\par However, there are few studies on the master equation for MFGC. In \cite{mou2022propagation}, the authors proved that several types of monotonicities can be propagated along any classical solution to the MFGC master equation. In \cite{liu2026global}, the authors adapted both the Lasry-Lions monotonicity and displacement $\lambda$-monotonicity to MFGC setting, and proved the global well-posedness for  the classical solution to the MFGC master equation under these monotonicity conditions. In \cite{jackson2026unconditional}, the authors also studied the classical solution to the MFGC master equation but by the corresponding $N$-player Nash systems under a similar displacement $\lambda$-monotonicity condition and Lasry-Lions monotonicity condition.

\subsection{Motivations}
\par Both the monotonicity conditions and the regularities of data contribute to the global well-posedness of the MFG master equation. When monotonicity conditions are absent, the uniqueness of mean field equilibrium and the global well-posedness of the MFG master equation may fail. In \cite{cecchin2022selection,cecchin2025weak}, the authors established well-posedness results for the conservative master equation in potential mean field games. In the general case, when there exist multiple mean field equilibria, the master equation becomes ill-posed; see \cite{mou2024minimal}. On the other hand, when the data is non-smooth, one cannot expect classical solutions to the master equation; see e.g. \cite[Example 10.1]{mou2024nonsmooth}. There are a few works considering the master equation with non-smooth data for classical MFG, in which weaker notions of solution are proposed. In \cite{mou2024nonsmooth}, the authors introduced three notions named as weak solutions, good solutions, and weak viscosity solutions. \cite{bertucci2024lipschitz} introduced a notion termed Lipschitz solutions as a general concept of solution for merely Lipschitz functions. 
In addition, the monotone solution was presented in \cite{bertucci2021monotone, bertucci2023monotone} in the monotone regime. 
\par In this manuscript, we aim to weaken the regularity assumptions, propose a notion of weak solution, and consider the global well-posedness of the MFGC master equation. To be detailed, we do not require any differentiability with respect to the measure variable $\mu$, and only need the Lipschitz continuity of data in $\mu$. Moreover, we shall adapt several notions of non-smooth solution raised for classical MFG to MFGC setting, and show their equivalence with the weak solution by definitions, providing a unification of weak solution notions for MFGC master equation. This equivalence result is new also in the classical MFG setting.

\subsection{Our contributions}
\subsubsection{Non-smooth solutions}
\par The notion of weak solution, which originates from \cite{mou2024nonsmooth}, is in the spirit of applying integration by parts to the MFG system. We shall apply this method to our MFGC system, which is a system of forward-backward stochastic PDE (FBSPDE) given later in (\ref{MFGC system}). However, to adapt to data with quadratic growth, we generalize the idea in \cite{mou2024nonsmooth}, and consider the vectorial MFGC system (\ref{vectorial MFGC system}), which is a formal differentiation of the MFGC system with respect to the space variable $x$. Then the weak solution $V$ to the master equation (\ref{introduction master}) can be recovered by an expectation whenever $\partial_x V$ is determined through the vectorial MFGC system. 
\par Using the weak solution framework, we next establish the global well-posedness of master equations for MFGC with Lipschitz data with a similar method to \cite{mou2024nonsmooth}. To establish existence of weak solutions, we construct the vectorial weak solution through the decoupling field of the associated Pontryagin system of forward-backward stochastic differential equations (FBSDE). Since our data are only assumed to be Lipschitz continuous, we employ the smooth mollifiers on Wasserstein space constructed in \cite{cosso2023smooth,mou2024nonsmooth}. Then the well-posedness of the mollified master equation follows from the results for classical solutions in \cite{liu2026global}. Finally we reach the conclusion by weak convergence arguments. To prove uniqueness, we show that any vectorial weak solution must coincide with the decoupling field to the FBSDE system. We begin with choosing smooth mollifiers as test functions in the definition of weak solutions so that the involved integrals become smooth. Applying It\^{o}-Wentzell's formula, then determining the convergence for each terms, we conclude the uniqueness of weak solutions from that of the FBSDE. 
\par After investigating weak solutions, we compare this notion with several types of non-smooth solutions for the master equation found in the literature, and show they are all equivalent through their definitions. These include weak solutions, good solutions, weak viscosity solutions, Lipschitz solutions, and monotone solutions. While originally defined for the MFG setting, we adapt them to our MFGC formulation. This yields a synthesis of weak solution concepts for MFGC master equation. By the equivalence among the different solution notions, we obtain well-posedness for all solution notions mentioned above.
\par As discussed above, the notion of weak solution in this manuscript, as well as that in \cite{mou2024nonsmooth} for classical MFG, is inspired by the application of integration by parts. A major difference, however, is that we allow the data to have up to quadratic growth. Consequently, we must consider the differentiated system first and make additional analysis in the mollification argument to overcome the weaker assumptions. In fact, once the missing regularity assumptions are imposed, our notion becomes consistent, to a certain extent, with the weak solution defined in \cite{mou2024nonsmooth}. In particular, we mention that our method also works for models with degenerate idiosyncratic noise, which is not covered there. Therefore, our notion can be viewed as a quadratic-growth and possibly degenerate extension to MFGC of the weak formulation in \cite{mou2024nonsmooth}.
\par While the weak solution is defined via integration by parts applied to the BSPDE, the weak viscosity solution in \cite{mou2024nonsmooth} is defined by studying the viscosity solution to the BSPDE. Different from the previous two notions, the good solutions in \cite{mou2024nonsmooth} are defined directly through the master equation rather than the FBSPDE. This notion is based on the stability argument. Roughly speaking, when the master equation with a sequence of mollified data has a sequence of classical solution, the limit solution is defined to be the good solution. The notion of Lipschitz solution was introduced in \cite{bertucci2024lipschitz}. The idea is to consider the associated linear transport equation, whose solution can be represented by the Feynman-Kac formula. Then this solution serves as the new data for the transport equation. A Lipschitz function that acts as a fixed point of this procedure is then identified as a Lipschitz solution. All of these notions can be naturally adapted to the MFGC setting.
\par The concept of a monotone solution originated in the MFG literature within a Wasserstein-space framework under Lasry-Lions monotonicity, as studied in \cite{bertucci2021monotone,bertucci2023monotone} for the case with only idiosyncratic noise. For equations driven by common noise and without idiosyncratic noise, \cite{cardaliaguet2022monotone} introduced a Hilbert-space formulation under the same monotonicity condition, relying on Lions’ lifting method \cite{lions2007college}. Because the noise structures differ, the probability-measure spaces used in \cite{bertucci2021monotone,bertucci2023monotone} and \cite{lions2007college} also differ, resulting in distinct definitions of monotone solution. A further notion of monotone solution, based on displacement monotonicity, was investigated in \cite{meynard2025monotone}. Our aim here is to connect monotone solutions with our definition of weak solution in the MFGC framework. Accordingly, we adopt only the Hilbert-space approach of \cite{cardaliaguet2022monotone} to establish the equivalence between our notion of weak solution and monotone solutions in the presence of only common noise, under both the Lasry-Lions monotonicity and displacement $\lambda$-monotonicity conditions. In the displacement $\lambda$-monotonicity setting, we further extend the notion of a monotone solution to that of a displacement $\lambda$-monotone solution. We expect that analogous equivalence results should hold for Wasserstein-space formulations of monotone solutions, but we do not pursue this question here.

\subsubsection{Pontryagin system of FBSDE}
\par Since our construction of weak solution is closely related to the Pontryagin system of FBSDE for MFGC, it is crucial to study the global well-posedness of these FBSDEs. It is known that the uniform Lipschitz continuity of decoupling field helps us extend a local solution to a global solution. Unlike \cite{gangbo2022displacement,mou2024nonsmooth}, due to the possible degeneracy of idiosyncratic noise, we establish the Lipschitz continuity in $x$ through the structure of the underlying control problem rather than backward stochastic differential equation (BSDE) estimates. Moreover, when the data are nonconvex, the stochastic maximum principle alone does not guarantee that the solution of the Pontryagin FBSDE system corresponds to an optimal control, preventing a direct identification with the associated control problem. To overcome this difficulty, we introduce an auxiliary control problem by a standard transformation (see \cite{arnold1989mechanics,bansil2025hidden,bansil2025classical}) of the data. The displacement $\lambda$-monotonicity of the original data becomes the standard displacement monotonicity of the transformed data, which yields convexity and allows us to recover the value function representation of the Pontryagin FBSDE system. The regularity of the associated value function then provides the desired Lipschitz estimate in $x$. To establish the uniform $\mathrm{W}_1$-Lipschitz estimate in $\mu$, we benefit from the monotonicity conditions. In particular, we mention that the $\mathrm{W}_1$-Lipschitz estimate under displacement $\lambda$-monotonicity condition benefits from the representation formula established in \cite{liu2026global}.

\subsubsection{Fixed-point condition}
\par Compared with classical MFG, the master equation for MFGC has an additional fixed-point condition of the form:
\begin{equation}\label{introduction fixed-point}
    \rho=\mathcal{L}\big(X,\partial_p H(X,Y,\rho)\big),
\end{equation}
where $X,Y$ are square-integrable random variables. Under appropriate conditions, the above fixed-point equation has a unique solution $\rho^*:=\Phi\big(\mathcal{L}(X,Y)\big)$, where $\Phi:\mathcal{P}_2(\R^{2d})\rightarrow\mathcal{P}_2(\R^{2d})$. To study the master equation for MFGC, it is important to consider the solvability of (\ref{introduction fixed-point}), and the regularity of $\Phi$, which is a map from infinite-dimensional space to infinite-dimensional space. In \cite{mou2022propagation}, the authors provided several concrete forms of $\Phi$ in special cases, as well as the definition for the linear functional derivative of $\Phi$. In \cite{jackson2025quantitative,jackson2026unconditional}, the authors proved the solvability of the fixed-point equation under certain monotonicity condition, and studied the regularities of the finite-dimensional fixed-point maps in corresponding $N$-player games. 
\par From (\ref{introduction master}) and (\ref{introduction fixed-point}), we see that the regularity of $H$ in $\mathcal{L}(X,Y)$ depends on the regularity of $\Phi$. To study the weak solution we require that $H$ is $\mathrm{W}_1$-Lipschitz continuous in $\mathcal{L}(X,Y)$, that is, we need to study the Lipschitz continuity of the fixed-point map $\Phi$. With a simple assumption on convexity (see Assumption \ref{convexity assumption}), we show that $\Phi$ has $\mathrm{W}_1$-Lipschitz continuity straightforwardly. We emphasize that this assumption is also crucial for the a priori Lipschitz estimate of the solution to the master equation. We also prove a stability result for $\Phi$ with respect to different $H$. By contrast, \cite{jackson2025quantitative,jackson2026unconditional} relied on a similar displacement $\lambda$-monotonicity condition to establish Lipschitz continuity of the fixed-point map, together with an extra condition to derive estimates for the classical solutions, leading to short time horizons for their well-posedness results unless $\lambda=0$. On this basis, by a canonical transformation in \cite{arnold1989mechanics,bansil2025hidden,bansil2025classical}, it is expected to conclude a global well-posedness result for all $\lambda\geq 0$ in \cite{jackson2026unconditional}. Nevertheless, our assumptions are not completely equivalent to theirs and still cover different cases even when $\lambda=0$ (see Remark \ref{remark compare with alpar}).

\subsection*{Organization}
The remainder of this manuscript is organized as follows. In Section 2, we provide the problem formulation, the key monotonicity conditions, and an analysis on the fixed-point map. In Section 3, we investigate the global well-posedness of the FBSDEs for MFGC, through which we define the candidate solution for the master equation. In Section 4, we introduce the notion of weak solution and prove its global well-posedness. In Section 5, we adapt several types of non-smooth solution in the literature to MFGC setting, and examine the relation among them and the weak solution.

\section{Preliminaries}
\subsection{The probability spaces and function spaces}
Throughout the paper, we shall work with the following probability spaces. Let $[0,T]$ be a finite time horizon. Let $(\Omega_0,\mathbb{F}^0,\mathbb{P}_0)$ and $(\Omega_1,\mathbb{F}^1,\mathbb{P}_1)$ be two filtered probability spaces on which there are defined $d$-dimensional Brownian motions $B^0$ and $B$, respectively. For $\mathbb{F}^i=\{\mathcal{F}_t^i\}_{0\leq t\leq T}$, $i=0,1$, we assume $\mathcal{F}_t^0=\mathcal{F}_t^{B^0}$, $\mathcal{F}_t^1=\mathcal{F}_0^1\vee\mathcal{F}_t^B$, and $\mathbb{P}_1$ has no atom in $\mathcal{F}_0^1$ so it can support any measure on $\R^d$ with finite second order moment. We denote by
$\mathbb{E}_0:=\mathbb{E}^{\mathbb{P}_0}$ and 
$\mathbb{E}_1:=\mathbb{E}^{\mathbb{P}_1}$
the expectations with respect to $\mathbb{P}_0$ and $\mathbb{P}_1$, respectively. Consider the product space
\begin{equation}
    \Omega:=\Omega_0\times\Omega_1,\quad \mathbb{F}=\{\mathcal{F}_t\}_{0\leq t\leq T}:=\{\mathcal{F}_t^0\otimes\mathcal{F}_t^1\}_{0\leq t\leq T},\quad \mathbb{P}:=\mathbb{P}_0\otimes\mathbb{P}_1,\quad \E:=\E^{\mathbb{P}}.
\end{equation}
Given $\mathcal{F}_t$-measurable random variable $\xi$, we use $\tilde{\xi},\bar{\xi}$ to denote conditionally independent copies of $\xi$ by possible extending to product sample space, conditioning on $\mathcal{F}_t^0$.
\par For any dimension $d$ and any constant $p\geq 1$, let $\mathcal{P}(\R^d)$ denote the set of probability measures on $\R^d$. For $\mu\in \mathcal{P}(\R^d)$, define $$M_p(\mu):=\left(\int_{\mathbb{R}^d}\vert x \vert^p\mu(\mathrm{d}x)\right)^{\frac{1}{p}}.$$ We denote by $\mathcal{P}_p(\R^d)$ the subset of probability measures with finite $p$-th moment, equipped with the $p$-Wasserstein distance $\mathrm{W}_p$.   Moreover, for any sub-$\sigma$-algebra $\mathcal{G}\subset\mathcal{F}_T$, $\mathbb{L}^p(\mathcal{G})$ denotes the set of $\R^d$-valued, $\mathcal{G}$-measurable, and $p$-integrable random variables. For any $\mu\in\mathcal{P}_p(\R^d)$, let $\mathbb{L}^p(\mathcal{G};\mu)$ denote the set of $\xi\in\mathbb{L}^p(\mathcal{G})$ with law $\mathcal{L}(\xi)=\mu$. Similarly, for any sub-filtration $\mathbb{G}\subset\mathbb{F}$, $\mathbb{L}(\mathbb{G};\R^d)$ denotes the set of $\mathbb{G}$-progressively measurable $\R^d$-valued processes.
\par For a function $V:\mathcal{P}_2(\R^d)\rightarrow\R$, we use $\partial_{\mu}V:\mathcal{P}_2(\R^d)\times\R^d\rightarrow\R^d$ to denote its Lions derivative. Let $C^0(\mathcal{P}_2(\R^d))$ denote the set of $\mathrm{W}_1$-continuous $V$; $C^1(\mathcal{P}_2(\R^d))$ denote the subset of $V\in C^0(\mathcal{P}_2(\R^d))$ such that $\partial_{\mu}V$ exists and is continuous on $\mathcal{P}_2(\R^d)\times\R^d$. Moreover, denote $\Theta:=[0,T]\times\R^d\times\mathcal{P}_2(\R^d)$. Let $C^{1,2,2}(\Theta)$ denote the set of $V\in C^0(\Theta)$ such that $\partial_t V,\partial_x V,\partial_{\mu}V,\partial_{xx}V,\partial_{x\mu}V,\partial_{\tilde{x}\mu}V,\partial_{\mu\mu}V$ exist and are continuous. Let $C_{Lip}^0(\Theta)$ denote the set of $V\in C^0(\Theta)$ such that $V$ is uniformly Lipschitz continuous in $(x,\mu)$, under $\mathrm{W}_1$ for $\mu$, uniformly in $t\in[0,T]$; $C_{Lip}^{0,1}(\Theta)$ denote the set of $V\in C^0(\Theta)$ such that $\partial_x V$ exists and $\partial_xV\in C_{Lip}^0(\Theta;\R^d)$. For $V\in C_{Lip}^0(\Theta:\R^d)$, $\partial_x V$ exists almost everywhere, then we say $V$ is curl-free in $x$ if $\partial_x V=\partial_x V^{\top}$ for a.e. $x$.
\par Given $t_0\in[0,T]$, denote $B_t^{t_0}:=B_t-B_{t_0},\,B_t^{0,t_0}:=B_t^0-B_{t_0}^0,\,t\in[t_0,T]$. Let $\mathcal{A}_{t_0}$ denote the set of admissible controls $\alpha:[t_0,T]\times\R^d\times C([t_0,T];\R^d)\rightarrow\R^d$ which are progressively measurable and adapted in the path variable and square-integrable. Let $\mathbb{L}^2(\mathbb{F}^{B^{0,t_0}};\mathcal{P}_2(\R^{2d}))$ denote the set of $\mathbb{F}^{B^{0,t_0}}$-progressively measurable stochastic probability measure flows $\{\nu.\}=\{\nu_t\}_{t\in[t_0,T]}\subset\mathcal{P}_2(\R^{2d})$.
\subsection{The master equation and mean field game of controls system}
The data of our mean field game of controls (MFGC) is given by
\[
L:\R^d\times\R^d\times\mathcal{P}_2(\R^{2d})\rightarrow\R;\quad G:\R^d\times\mathcal{P}_2(\R^d)\rightarrow\R
\]
and $\sigma,\,\sigma_0\in[0,\infty)$. For simplicity, we assume that $L$ does not depend on time.
\par Given $t_0\in[0,T]$, $x\in\R^d$, $\alpha\in\mathcal{A}_{t_0}$, and $\{\nu.\}\in\mathbb{L}^2(\mathbb{F}^{B^{0,t_0}};\mathcal{P}_2(\R^{2d}))$, the state of the agent satisfies the following controlled stochastic differential equation (SDE) on $[t_0,T]$:
\begin{equation}\label{controlled SDE}
\begin{gathered}
    X_t^{\{\nu.\},\alpha}=x+\int_{t_0}^t \alpha_s\,\mathrm{d}s+\sigma B_t^{t_0}+\sigma_0 B_t^{0,t_0};\\
    \text{where }X^{\{\nu.\},\alpha}=X^{t_0,\{\nu.\};x,\alpha},\quad \alpha_t:=\alpha(t,X_t^{\{\nu.\},\alpha},B_{[t_0,t]}^{0,t_0}).
\end{gathered}
\end{equation}
Let $\pi_1\#\nu_T$ denote the first $\R^d$ marginal measure of $\nu_t$. The expected cost is given by
\begin{equation}
    J(t_0,x;\{\nu.\},\alpha):=\E\Big[G(X_T^{\{\nu.\},\alpha},\pi_1\#\nu_T)+\int_{t_0}^T L(X_t^{\{\nu.\},\alpha},\alpha_t,\nu_t)\,\mathrm{d}t\Big].
\end{equation}
\begin{definition}
\label{MFE}
    For any $(t,\mu)\in[0,T]\times\mathcal{P}_2(\R^d)$, we say $(\alpha^*,\{\nu.^*\})\in\mathcal{A}_t\times\mathbb{L}^2(\mathbb{F}^{B^{0,t}};\mathcal{P}_2(\R^{2d}))$ is a mean field equilibrium (MFE) at $(t,\mu)$ if
    \begin{equation}
    \begin{gathered}
        J(t,x;\{\nu.^*\},\alpha^*)=\inf_{\alpha\in\mathcal{A}_t}J(t,x;\{\nu.^*\},\alpha),\quad\text{for }\mu-\text{a.e. }x\in\R^d;\\
        \pi_1\#\nu_t^*=\mu,\quad \nu_s^*:=\mathcal{L}^0\big(X_s^*,\alpha^*(s,X_s^*,B_{[t,s]}^{0,t})\big),\quad \text{where}\\
        X_s^*=\xi+\int_t^s \alpha^*(r,X_r^*,B_{[t,r]}^{0,t})\,\mathrm{d}r+\sigma B_s^t+\sigma_0 B_s^{0,t},\quad\xi\in\mathbb{L}^2(\mathcal{F}_t^1,\mu).
    \end{gathered}
    \end{equation}
\end{definition}
\noindent Here we use $\mathcal{L}^0$ to denote the law conditioning on the common noise, that is, $\mathcal{L}^0(X_t):=\mathcal{L}(X_t|\mathcal{F}_t^0)$ for an $\mathcal{F}_t$-measurable random variable $X_t$.
\par When there is a unique MFE for each $(t,\mu)\in[0,T]\times\mathcal{P}_2(\R^d)$, denoted as $(\alpha^*(t,\mu;\cdot),\{\nu.^*(t,\mu)\})$, then the game problem leads to the following value function for the agent:
\begin{equation}\label{value function}
    V(t,x,\mu):=J(t,x;\{\nu.^*(t,\mu)\},\alpha^*(t,\mu;\cdot))\quad\text{for any }x\in\R^d.
\end{equation}
We define the following Hamiltonian: for $(x,p,\nu)\in\R^d\times\R^d\times\mathcal{P}_2(\R^{2d})$,
\begin{equation}
    H(x,p,\nu):=\inf_{a\in\R^d}h(x,p,\nu,a),\quad h(x,p,\nu,a):=p\cdot a+L(x,a,\nu).
\end{equation}
By It\^{o}'s formula, the value function (\ref{value function}) formally satisfies the following master equation: for independent copies $\xi,\tilde{\xi},\bar{\xi}$ with law $\mu$,
\begin{equation}
\label{master equation}
\begin{gathered}
    \mathcal{L}V(t,x,\mu):=\partial_t V+\frac{\widehat{\sigma}^2}{2}\tr(\partial_{xx}V)+H(x,\partial_x V,\rho)+\mathcal{M}V=0,\\
    V(T,x,\mu)=G(x,\mu),\quad\text{where}\\
    \mathcal{M}V(t,x,\mu):=\tr\Big(\bar{\widetilde{\E}}\Big[\frac{\widehat{\sigma}^2}{2}\partial_{\tilde{x}\mu}V(t,x,\mu,\tilde{\xi})+\sigma_0^2\partial_{x\mu}V(t,x,\mu,\tilde{\xi})+\frac{\sigma_0^2}{2}\partial_{\mu\mu}V(t,x,\mu,\bar{\xi},\tilde{\xi})\\
    +\partial_{\mu}V(t,x,\mu,\tilde{\xi})\cdot\partial_p H\big(\tilde{\xi},\partial_x V(t,\tilde{\xi},\mu),\rho\big)\Big]\Big),\quad \rho=\big(id,\partial_p H(\cdot,\partial_x V(t,\cdot,\mu),\rho)\big)\#\mu.
\end{gathered}
\end{equation}
To deal with the fixed-point condition $\rho=\big(id,\partial_p H(\cdot,\partial_x V(t,\cdot,\mu),\rho)\big)\#\mu$, we consider the map
\begin{equation}\label{fixed-point map}
    \rho\mapsto\mathcal{L}\big(X,\partial_p H(X,Y,\rho)\big),
\end{equation}
where $\widehat{\sigma}^2:=\sigma^2+\sigma_0^2$.
\par Throughout this paper, we always assume the above map (\ref{fixed-point map}) has a unique fixed point for any $\mathcal{L}(X,Y)\in\mathcal{P}_2(\R^{2d})$, denoted as $\Phi(\mathcal{L}(X,Y))$, and $\Phi$ is a Lipschitz continuous map under $\mathrm{W}_1$. For example, the fixed-point condition in (\ref{master equation}) can be rewritten as
\begin{equation}
    \rho=\Phi\big(\mathcal{L}(\xi,\partial_x V(t,\xi,\mu))\big).
\end{equation}
\begin{remark}
    We note that $\Phi:\mathcal{P}_2(\R^{2d})\rightarrow\mathcal{P}_2(\R^{2d})$ is a map from infinite-dimensional space to infinite-dimensional space. Due to this complicated structure, we do not expect enough regularities for $\Phi$. However, to study the well-posedness of the master equation, the regularities of $H$ with respect to the measure variable is necessary.
    \par To overcome this difficulty, we shall study the weak solution to the master equation under the condition that $\Phi$ is $\mathrm{W}_1$-Lipschitz continuous. In the later discussion in Section \ref{subsection fixed-point}, we will prove that under appropriate conditions, there exists a unique $\mathrm{W}_1$-Lipschitz continuous map $\Phi$.
\end{remark}

\par The master equation (\ref{master equation}) is associated to the following forward-backward stochastic partial differential equation (FBSPDE) known as the MFGC system: for $\mathcal{F}_{t_0}^1$-measurable random variable $\xi$,
\begin{equation}
\label{MFGC system}
\left\{\begin{aligned}
    &\mathrm{d}\mu_t(x)=\Big[\frac{\widehat{\sigma}^2}{2}\mathrm{tr}(\partial_{xx}\mu_t(x))-\divergence\big(\mu_t(x)\partial_p H(x,\partial_x v(t,x),\rho_t)\big)\Big]\,\mathrm{d}t-\sigma_0\partial_x\mu_t(x)\cdot \mathrm{d}B_t^0,\\
    &\mathrm{d}v(t,x)=w(t,x)\cdot \mathrm{d}B_t^0-\Big[\frac{\widehat{\sigma}^2}{2}\tr(\partial_{xx}v(t,x))+\sigma_0\tr(\partial_x w(t,x))\\
    &\qquad\qquad\quad+H(x,\partial_x v(t,x),\rho_t)\Big]\mathrm{d}t,\\
    &\rho_t=\big(id,\partial_p H(\cdot,\partial_x v(t,\cdot),\rho_t)\big)\#\mu_t,\quad\mu_{t_0}=\mathcal{L}(\xi),\quad v(T,x)=G(x,\mu_T).
\end{aligned}\right.
\end{equation}
The first equation is a stochastic partial differential equation (SPDE) with solution $\mu$, the second equation is a backward stochastic partial differential equation (BSPDE) with solution $(v,w)$, and $\mu,w,v$ are all $\mathbb{F}^0$-progressively measurable. For the SPDE, $\mu_t$ should be understood as a solution in the sense of distribution.
\par We shall further consider the following vectorial master equation, which is a formal differentiation with respect to $x_i$ of the master equation (\ref{master equation}): for $i=1,\ldots,d$,
\begin{equation}\label{vectorial master equation}
\begin{gathered}
    \partial_t U^i(t,x,\mu)+\frac{\widehat{\sigma}^2}{2}\tr(\partial_{xx}U^i)+\partial_{x_i} H(x,U,\rho)+\partial_p H(x,U,\rho)\cdot\partial_{x_i} U\\
    +\tr\Big(\bar{\widetilde{\E}}\big[\frac{\widehat{\sigma}^2}{2}\partial_{\tilde{x}\mu}U^i(t,x,\mu,\tilde{\xi})+\sigma_0^2\partial_{x\mu}U^i(t,x,\mu,\tilde{\xi})+\frac{\sigma_0^2}{2}\partial_{\mu\mu}U^i(t,x,\mu,\bar{\xi},\tilde{\xi})\\
    +\partial_{\mu}U^i(t,x,\mu,\tilde{\xi})\cdot\partial_p H(\tilde{\xi},U(t,\tilde{\xi},\mu),\rho)\big]\Big)=0,\\
    \rho:=\big(id,\partial_p H(\cdot,U(t,\cdot,\mu),\rho)\big)\#\mu,\quad U^i(T,x,\mu)=\partial_{x_i} G(x,\mu).
\end{gathered}
\end{equation}
Accordingly, we can differentiate the MFGC system (\ref{MFGC system}) with respect to $x_i$ to obtain the following vectorial MFGC system: for $i=1,\ldots,d,$
\begin{equation}\label{vectorial MFGC system}
\left\{\begin{aligned}
    &\mathrm{d}\mu_t(x)=\Big[\frac{\widehat{\sigma}^2}{2}\mathrm{tr}(\partial_{xx}\mu_t(x))-\divergence\big(\mu_t(x)\partial_p H(x,u(t,x),\rho_t)\big)\Big]\,\mathrm{d}t-\sigma_0\partial_x\mu_t(x)\cdot \mathrm{d}B_t^0,\\
    &\mathrm{d}u^i(t,x)=r^i(t,x)\,\mathrm{d}B_t^0-\Big[\frac{\widehat{\sigma}^2}{2}\tr(\partial_{xx}u^i(t,x))+\sigma_0\tr(\partial_x r^i(t,x))\\
    &\qquad\qquad\quad+\partial_{x_i}H(x,u(t,x),\rho_t)+\partial_p H(x,u(t,x),\rho_t)\cdot\partial_{x_i} u(t,x)\Big]\mathrm{d}t,\\
    &\rho_t=\big(id,\partial_p H(\cdot,u(t,\cdot),\rho_t)\big)\#\mu_t,\quad\mu_{t_0}=\mathcal{L}(\xi),\quad u^i(T,x)=\partial_{x_i}G(x,\mu_T).
\end{aligned}\right.
\end{equation}
If the master equation (\ref{master equation}) has a classical solution $V$, then we have the following relation among the master equation (\ref{master equation}), MFGC system (\ref{MFGC system}), and their vectorial versions (\ref{vectorial master equation})(\ref{vectorial MFGC system}): for any $(t_0,\mu_{t_0})$, we have $v(t,x)=V(t,x,\mu_t)$, $u(t,x)=U(t,x,\mu_t)$, $U(t,x,\mu)=\partial_x V(t,x,\mu)$ and $u(t,x)=\partial_x v(t,x)$.
\par Alternatively, given $t_0\in[0,T]$ and $\xi\in\mathbb{L}^2(\mathcal{F}_{t_0})$, we may consider the following Mckean-Vlasov forward-backward stochastic differential equation (FBSDE) on $[t_0,T]$:
\begin{equation}\label{FBSDE xi}
\left\{\begin{aligned}
    &X_t^{\xi}=\xi+\int_{t_0}^t\partial_p H(X_s^{\xi},\nabla Y_s^{\xi},\rho_s)\,\mathrm{d}s+\sigma B_t^{t_0}+\sigma_0 B_t^{0,t_0},\\
    &\nabla Y_t^{\xi}=\partial_x G(X_T^{\xi},\mu_T)+\int_t^T\partial_x H(X_s^{\xi},\nabla Y_s^{\xi},\rho_s)\,\mathrm{d}s\\
    &\qquad\qquad-\int_t^T\nabla Z_s^{\xi} \,\mathrm{d}B_s-\int_t^T\nabla Z_s^{0,\xi} \,\mathrm{d}B_s^0,\\
    &\rho_t=\mathcal{L}^0\big(X_t^{\xi},\partial_p H(X_t^{\xi},\nabla Y_t^{\xi},\rho_t)\big),\quad\mu_t=\mathcal{L}^0(X_t^{\xi}).
\end{aligned}\right.
\end{equation}
Given the $\mu,\rho$ in (\ref{FBSDE xi}), we further consider the FBSDE
\begin{equation}\label{FBSDE xi,x}
\left\{\begin{aligned}
    &X_t^{\xi,x}=x+\int_{t_0}^t\partial_p H(X_s^{\xi,x},\nabla Y_s^{\xi,x},\rho_s)\,\mathrm{d}s+\sigma B_t^{t_0}+\sigma_0 B_t^{0,t_0},\\
    &\nabla Y_t^{\xi,x}=\partial_x G(X_T^{\xi,x},\mu_T)+\int_t^T\partial_x H(X_s^{\xi,x},\nabla Y_s^{\xi,x},\rho_s)\,\mathrm{d}s\\
    &\qquad\qquad-\int_t^T\nabla Z_s^{\xi,x} \,\mathrm{d}B_s-\int_t^T\nabla Z_s^{0,\xi,x}\,\mathrm{d}B_s^0
\end{aligned}\right.
\end{equation}
The FBSDEs (\ref{FBSDE xi})-(\ref{FBSDE xi,x}) connect to the master equation as follows:
\begin{equation}
\begin{gathered}
    \nabla Y_t^{\xi}=\partial_x V(t,X_t^{\xi},\mu_t),\,\nabla Z_t^{\xi}=\sigma\partial_{xx}V(t,X_t^{\xi},\mu_t),\\
    \nabla Y_t^{\xi,x}=\partial_x V(t,X_t^{\xi,x},\mu_t),\,\nabla Z_t^{\xi,x}=\sigma\partial_{xx}V(t,X_t^{\xi,x},\mu_t).
\end{gathered}
\end{equation}
Moreover, given the $(X^{\xi,x},\nabla Y^{\xi,x})$ in (\ref{FBSDE xi,x}) and $\mu,\rho$ in (\ref{FBSDE xi}), for $\xi\in\mathbb{L}^2(\mathcal{F}_{t_0}^1;\mu)$, we can define
\begin{equation}\label{define V}
    V(t_0,x,\mu):=\E\Big[G(X_T^{\xi,x},\mu_T)+\int_{t_0}^T\big[H(X_s^{\xi,x},\nabla Y_s^{\xi,x},\rho_s)-\nabla Y_s^{\xi,x}\cdot\partial_p H(X_s^{\xi,x},\nabla Y_s^{\xi,x},\rho_s)\big]\,\mathrm{d}s\Big],
\end{equation}
which solves the master equation (\ref{master equation}).
\begin{remark}\label{remark hat H}
    Recalling the map $\Phi$ defined by (\ref{fixed-point map}), we can define the composite Hamiltonian $\widehat{H}(\cdot,\cdot,\nu):=H(\cdot,\cdot,\Phi(\nu))$. With this notation, all the terms involving Hamiltonian $H$ in each system introduced in this subsection, can be represented by $\widehat{H}$. For example, in the FBSDE (\ref{FBSDE xi}), the fixed-point equation implies $\rho_t=\Phi(\mathcal{L}^0(X_t^{\xi},\nabla Y_t^{\xi}))$, and thus $\partial_p H(X_s^{\xi},\nabla Y_s^{\xi},\rho_s)=\partial_p\widehat{H}(X_s^{\xi},\nabla Y_s^{\xi},\mathcal{L}^0(X_s^{\xi},\nabla Y_s^{\xi}))$. This kind of representation with $\widehat{H}$ is convenient for the mollification argument in later discussion.
\end{remark}
\subsection{Main assumptions and the monotonicities}
We first give the following regularity assumptions on $G$ and $L$. Since we focus on the weak solutions, we do not require any differentiability in the measure variable.
\begin{assumption}\label{regularity assumption}
    $\partial_x G:\R^d\times\mathcal{P}_2(\R^d)\rightarrow\R^d$ exists and is uniformly Lipschitz continuous in $x,\mu$ with Lipschitz constant $C_{Lip}$; $\partial_x L,\partial_a L:\R^d\times\R^d\times\mathcal{P}_2(\R^{2d})\rightarrow\R^d$ exist and are uniformly Lipschitz continuous in $x,a,\rho$ with Lipschitz constant $C_{Lip}$, where the Lipschitz continuity in measure is under $\mathrm{W}_1$.
\end{assumption}
\par The follow assumption states that $G$ and $L$ are continuous and have at most quadratic growth. In the later Remark \ref{remark: W1 continuous}, we will explain why we require the growth in measure to depend on $M_1(\mu)$ rather than $M_2(\mu)$.
\begin{assumption}\label{quadratic growth assumption}
    $G\in C^0(\R^d\times\mathcal{P}_2(\R^d)),\,L\in C^0(\R^d\times\R^d\times\mathcal{P}_2(\R^{2d}))$, where the continuity in measure is under $\mathrm{W}_1$. Moreover, there exists $C>0$ such that for any $x,a\in\R^d,\mu\in\mathcal{P}_2(\R^d),\rho\in\mathcal{P}_2(\R^{2d})$,
    \begin{equation}
    \begin{gathered}
        |G(x,\mu)|\leq C\Big(1+|x|^2+\big(M_1(\mu)\big)^2\Big),\\
        |L(x,a,\rho)|\leq C\Big(1+|x|^2+|a|^2+\big(M_1(\rho)\big)^2\Big).
    \end{gathered}
    \end{equation}
\end{assumption}

We also make the following convexity assumption on $L$:
\begin{assumption}\label{convexity assumption}
    (i) There exists $c_0>0$ such that for any $x,a_1,a_2\in\R^d,\rho\in\mathcal{P}_2(\R^{2d})$,
    \begin{equation}
        L(x,a_2,\rho)-L(x,a_1,\rho)-\partial_a L(x,a_1,\rho)\cdot(a_2-a_1)\geq\frac{c_0}{2}|a_2-a_1|^2.
    \end{equation}
    (ii) There exists $c_1<c_0$ such that for any $x,a\in\R^d$ and square-integrable $\R^d$-valued random variables $X^1, X^2$, $\alpha^1, \alpha^2$,
    \begin{equation}
        |\partial_a L(x,a,\mathcal{L}(X^1,\alpha^1))-\partial_a L(x,a,\mathcal{L}(X^2,\alpha^2))|\leq C\E[|X^1-X^2|]+c_1\E[|\alpha^1-\alpha^2|].
    \end{equation}
\end{assumption}
\par Next we provide the crucial Lasry-Lions and displacement $\lambda$-monotonicity conditions on the terminal cost $G$ and the nonseparable Lagrangian $L$:
\begin{definition}
    Let $G:\R^d\times\mathcal{P}_2(\R^d)\rightarrow\R$.\\
    (i) We say $G$ is Lasry-Lions monotone, if for any square-integrable $\R^d$-valued random variables $X^1,X^2$,
    \begin{equation}\label{G Lasry-Lions}
        \E[G(X^1,\mathcal{L}(X^1))-G(X^1,\mathcal{L}(X^2))-G(X^2,\mathcal{L}(X^1))+G(X^2,\mathcal{L}(X^2))]\geq 0,
    \end{equation}
    (ii) We say $G$ is displacement $\lambda$-monotone for some $\lambda\geq 0$, if $G(\cdot,\mu)\in C^1(\R^d)$ and for any square-integrable $\R^d$-valued random variables $X^1,X^2$,
    \begin{equation}\label{G displacement}
        \E[\langle\partial_x G(X^1,\mathcal{L}(X^1))-\partial_x G(X^2,\mathcal{L}(X^2)),X^1-X^2\rangle]\geq -\lambda\E[|X^1-X^2|^2].
    \end{equation}
\end{definition}
\begin{definition}
    Let $L:\R^d\times\R^d\times\mathcal{P}_2(\R^{2d})\rightarrow \R$.\\
    (i) We say $L$ is Lasry-Lions monotone, if for any square-integrable $\R^d$-valued random variables $X^1, X^2$, $\alpha^1, \alpha^2$,
    \begin{equation}\label{L Lasry-Lions}
    \begin{aligned}
        &\E[L(X^1,\alpha^1,\mathcal{L}(X^1,\alpha^1))-L(X^1,\alpha^1,\mathcal{L}(X^2,\alpha^2))\\
        &\quad-L(X^2,\alpha^2,\mathcal{L}(X^1,\alpha^1))+L(X^2,\alpha^2,\mathcal{L}(X^2,\alpha^2))]\geq 0.
    \end{aligned}
    \end{equation}
    (ii) We say $L$ is displacement $\lambda$-monotone for some $\lambda\geq 0$, if $L(\cdot,\cdot,\rho)\in C^1(\R^d\times\R^d)$ and for any square-integrable $\R^d$-valued random variables $X^1,X^2,\alpha^1,\alpha^2$,
    \begin{equation}\label{L displacement}
    \begin{aligned}
        &\E[\langle\partial_a L(X^1,\alpha^1,\mathcal{L}(X^1,\alpha^1))-\partial_a L(X^2,\alpha^2,\mathcal{L}(X^2,\alpha^2)),\alpha^1-\alpha^2\rangle\\
        &\quad+\langle\partial_x L(X^1,\alpha^1,\mathcal{L}(X^1,\alpha^1))-\partial_x L(X^2,\alpha^2,\mathcal{L}(X^2,\alpha^2)),X^1-X^2\rangle]\\
        \geq&2\lambda\E[\langle X^1-X^2,\alpha^1-\alpha^2\rangle].
    \end{aligned}
    \end{equation}
\end{definition}
\begin{remark}
    By the Legendre transform, the above assumptions on the Lagrangian
$L$ can be equivalently formulated in terms of the Hamiltonian $H$:\\
    (i) The Lipschitz continuity of $\partial_x L,\ \partial_a L$ and convexity of $L$ in Assumption \ref{regularity assumption} and \ref{convexity assumption}(i) are equivalent to the Lipschitz continuity of $\partial_x H,\ \partial_p H$, and the concavity of $H$: for any $x,p_1,p_2\in\R^d,\rho\in\mathcal{P}_2(\R^{2d})$,
    \begin{equation}
        H(x,p_2,\rho)-H(x,p_1,\rho)-\partial_p H(x,p_1,\rho)\cdot(p_2-p_1)\leq-\frac{C_{Lip}^{-1}}{2}|p_2-p_1|^2.
    \end{equation}
    (ii) Recall 
$
\widehat H
$
introduced in Remark 2.3. By \cite[Assumption 3.1]{liu2026global}, for any $\psi\in C^0_{Lip}(\R^{d}\times\mathcal{P}_2(\R^{d});\R^{d})$, if \(L\) satisfies the Lasry-Lions monotonicity condition \eqref{L Lasry-Lions}, then, under Assumption \ref{convexity assumption} (i), for any $\mu_1,\mu_2\in\mathcal{P}_2(\R^{d})$, denote $\rho_i=(id,\psi(\cdot,\mu_i))\#\mu_i$, $i=1,2$,
\begin{equation}
\begin{aligned}
&
\int_{\mathbb R^d}
\Big[
\widehat H(x,\psi(x,\mu_1),\rho_1)
-\widehat H(x,\psi(x,\mu_2),\rho_2)
\Big](\mu_1-\mu_2)(\mathrm{d}x)
\\
&-
\int_{\mathbb R^d}
(\psi(x,\mu_1)-\psi(x,\mu_2))\cdot
\big(\partial_p\widehat H(x,\psi(x,\mu_1),\rho_1)\,\mu_1(\mathrm{d}x)\\
&-
\partial_p\widehat H(x,\psi(x,\mu_2),\rho_2)\,\mu_2(\mathrm{d}x)\geq 0.
\end{aligned}
\end{equation} 
Moreover, for any square-integrable $\R^d$-valued random variables $X^1,X^2,P^1,P^2$, denote $\rho_i:=\Phi(\mathcal L(X^i,P^i))$, $i=1,2.$, the displacement $\lambda$-monotonicity condition becomes
    \begin{equation}\label{H displacement}
    \begin{aligned}
        &\E\big[\big\langle\partial_x H(X^1,P^1,\rho_1)-\partial_x H(X^2,P^2,\rho_2),X^1-X^2\big\rangle\\
        &\quad-\big\langle\partial_p H(X^1,P^1,\rho_1)-\partial_p H(X^2,P^2,\rho_2),P^1-P^2\big\rangle\big]\\
        \geq&2\lambda\E\big[\big\langle \partial_p H(X^1,P^1,\rho_1)-\partial_p H(X^2,P^2,\rho_2),X^1-X^2\big\rangle\big].
    \end{aligned}
    \end{equation}

\end{remark}
\par In addition, we need the following two assumptions, which are used to derive the a priori Lipschitz continuity of the solution with respect to $x$. In particular, the first one is applied under Lasry-Lions monotonicity condition, and the second one is applied under displacement $\lambda$-monotonicity condition.
\begin{assumption}\label{G L convexity assumption}
    $G(\cdot,\mu)$ is convex on $\mathbb{R}^d$ for any $\mu\in\mathcal{P}_2(\mathbb{R}^d)$, and that $L(\cdot,\cdot,\rho)$ is convex on $\mathbb{R}^d\times \mathbb{R}^d$ for any $\rho\in \mathcal{P}_2(\mathbb{R}^{2d})$.
\end{assumption}

\begin{assumption}\label{locally bounded assumption}
    For any $p\in\R^d$, there exists $C(p)>0$ such that for any $\rho\in\mathcal{P}_2(\R^{2d})$,
    \begin{equation}\label{locally bounded}
        \big|[\partial_a L(0,\cdot,\rho)]^{-1}(p)\big|\leq C(p),
    \end{equation}
    where the existence of $\partial_a L^{-1}$ is ensured by Assumption \ref{convexity assumption}(i).
\end{assumption}

\subsection{The fixed-point map}\label{subsection fixed-point}
In this subsection we give a brief analysis on the fixed-point map $\Phi$ defined by (\ref{fixed-point map}). To ensure the well-posedness of the master equation, we require that $\partial_x H(\cdot,\cdot,\Phi(\nu)),\,\partial_p H(\cdot,\cdot,\Phi(\nu))$ are $\mathrm{W}_1$-Lipschitz continuous in $\nu$. Therefore, we are looking for a $\mathrm{W}_1$-Lipschitz continuous map $\Phi:\mathcal{P}_2(\R^{2d})\rightarrow\mathcal{P}_2(\R^{2d})$.
\par In \cite[Lemma 3.2]{jackson2025quantitative}, the authors proved that for any bounded random variables $X,Y$, there exists a fixed point with bounded support to the map (\ref{fixed-point map}) under Assumption \ref{regularity assumption} and the following coercivity condition.
\begin{assumption}\label{coercivity assumption}
    There exists $C>0$ such that for any $(x,a,\rho)\in\R^d\times\R^d\times\mathcal{P}_2(\R^{2d})$,
    \begin{equation}
        \frac{1}{C}|a|^2-C\big(1+|x|^2+\mathrm{W}_2(\rho,\delta_{(0,0)})\big)\leq L(x,a,\rho)\leq C|a|^2+C\big(1+|x|^2+\mathrm{W}_2(\rho,\delta_{(0,0)})\big).
    \end{equation}
\end{assumption}
\par With this existence result, moreover, we benefit from the convexity conditions Assumption \ref{convexity assumption} and prove that (\ref{fixed-point map}) uniquely defines a $\mathrm{W}_1$-Lipschitz continuous map $\Phi$, which is necessary in our later analysis on master equation.
\begin{proposition}\label{existence fixed-point}
    Let Assumptions \ref{regularity assumption}, \ref{convexity assumption}, \ref{coercivity assumption} hold. Then (\ref{fixed-point map}) uniquely defines a $\mathrm{W}_1$-Lipschitz continuous map $\Phi:\mathcal{P}_2(\R^{2d})\rightarrow\mathcal{P}_2(\R^{2d})$.
\end{proposition}
\begin{proof}
    For any $\nu^1,\nu^2\in\mathcal{P}_{\infty}(\R^{2d})$, where $\mathcal{P}_{\infty}$ denotes the set of probability measure with bounded support, we choose bounded random variables $(X^i,Y^i)$ such that $\nu^i=\mathcal{L}(X^i,Y^i),i=1,2$, and $\mathrm{W}_1(\nu^1,\nu^2)=\E[|X^1-X^2|+|Y^1-Y^2|]$. By \cite[Lemma 3.2]{jackson2025quantitative}, there exist $\rho^i:=\Phi(\nu^i)\in\mathcal{P}_{\infty}(\R^{2d})$ such that $$\rho^i=\mathcal{L}\big(X^i,\partial_p H(X^i,Y^i,\rho^i)\big),\ i=1,2.$$
    Let $\alpha^i:=\partial_p H\big(X^i,Y^i,\rho^i\big)$, then we have the relation $\partial_a L\big(X^i,\alpha^i,\rho^i\big)=\partial_a L\big(X^i,\alpha^i,\mathcal{L}(X^i,\alpha^i)\big)=-Y^i$.
    \par By Assumptions \ref{regularity assumption}, \ref{convexity assumption}, we can compute that
    \begin{equation}\label{fixed point estimate}
    \begin{aligned}
        &|Y^1-Y^2|\cdot|\alpha^1-\alpha^2|\\
        \geq&\langle\partial_a L\big(X^1,\alpha^1,\mathcal{L}(X^1,\alpha^1)\big)-\partial_a L\big(X^2,\alpha^2,\mathcal{L}(X^2,\alpha^2)\big),\alpha^1-\alpha^2\rangle\\
        =&\langle\partial_a L\big(X^1,\alpha^1,\mathcal{L}(X^1,\alpha^1)\big)-\partial_a L\big(X^1,\alpha^2,\mathcal{L}(X^1,\alpha^1)\big),\alpha^1-\alpha^2\rangle\\
        &+\langle\partial_a L\big(X^1,\alpha^2,\mathcal{L}(X^1,\alpha^1)\big)-\partial_a L\big(X^2,\alpha^2,\mathcal{L}(X^2,\alpha^2)\big),\alpha^1-\alpha^2\rangle\\
        \geq&c_0|\alpha^1-\alpha^2|^2-|\partial_a L\big(X^1,\alpha^2,\mathcal{L}(X^1,\alpha^1)\big)-\partial_a L\big(X^2,\alpha^2,\mathcal{L}(X^2,\alpha^2)\big)|\cdot|\alpha^1-\alpha^2|,\\
    \end{aligned}
    \end{equation}
    which implies
    \[
    \begin{aligned}
        &\E[|Y^1-Y^2|]\\
        \geq&c_0\E[|\alpha^1-\alpha^2|]-\E[|\partial_a L\big(X^1,\alpha^2,\mathcal{L}(X^1,\alpha^1)\big)-\partial_a L\big(X^2,\alpha^2,\mathcal{L}(X^2,\alpha^2)\big)|]\\
        \geq&(c_0-c_1)\E[|\alpha^1-\alpha^2|]-C\E[|X^1-X^2|].
    \end{aligned}
    \]
    Since $c_0>c_1$, we have
    \begin{equation}
        \E[|\alpha^1-\alpha^2|]\leq C\E[|X^1-X^2|+|Y^1-Y^2|].
    \end{equation}
    Therefore, we conclude that
    \begin{equation}
        \mathrm{W}_1(\rho^1,\rho^2)\leq\E[|X^1-X^2|+|\alpha^1-\alpha^2|]\leq C\E[|X^1-X^2|+|Y^1-Y^2|],
    \end{equation}
    that is,
    \begin{equation}\label{fixed point W1}
        \mathrm{W}_1(\Phi(\nu^1),\Phi(\nu^2))\leq C\mathrm{W}_1(\nu^1,\nu^2).
    \end{equation}
    This implies (\ref{fixed-point map}) uniquely defines a $\mathrm{W}_1$-Lipschitz continuous map $\Phi$ on $\mathcal{P}_{\infty}(\R^{2d})$. 
    \par Next we extend this $\Phi$ to a $\mathrm{W}_1$-Lipschitz continuous map $\Phi:\mathcal{P}_2(\R^{2d})\rightarrow\mathcal{P}_2(\R^{2d})$. For this aim, we shall use a $\mathrm{W}_2$-Lipschitz estimate for $\Phi$. Note that (\ref{fixed point estimate}) also implies
    \begin{equation}
    \begin{aligned}
        &\E[|Y^1-Y^2|\cdot|\alpha^1-\alpha^2|]\\
        \geq&c_0\E[|\alpha^1-\alpha^2|^2]-c_1\E[|\alpha^1-\alpha^2|]^2\\
        &-C\E[|X^1-X^2|]\E[|\alpha^1-\alpha^2|]-C\E[|X^1-X^2|\cdot|\alpha^1-\alpha^2|],
    \end{aligned}
    \end{equation}
    which implies
    \begin{equation}
        \E[|\alpha^1-\alpha^2|^2]\leq C\E[|X^1-X^2|^2+|Y^1-Y^2|^2],
    \end{equation}
    so we can similarly show that for any $\nu^1,\nu^2\in\mathcal{P}_{\infty}(\R^{2d})$,
    \begin{equation}\label{fixed point W2}
        \mathrm{W}_2(\Phi(\nu^1),\Phi(\nu^2))\leq C\mathrm{W}_2(\nu^1,\nu^2).
    \end{equation}
    Since $\mathcal{P}_{\infty}(\R^{2d})$ is dense in $\mathcal{P}_2(\R^{2d})$, we can uniquely extend $\Phi$ to a $\mathrm{W}_1$-Lipschitz continuous map on $\mathcal{P}_2(\R^{2d})$. In fact, for any $\nu\in\mathcal{P}_2(\R^{2d})$, we can choose $\nu_n\in\mathcal{P}_{\infty}(\R^{2d})$ such that $\mathrm{W}_2(\nu_n,\nu)\rightarrow 0,\,n\rightarrow\infty$. By (\ref{fixed point W2}), we see that $\Phi(\nu_n)$ is Cauchy sequence in $\mathcal{P}_2(\R^{2d})$, so there exists a unique $\rho\in\mathcal{P}_2(\R^{2d})$ such that $\mathrm{W}_2(\Phi(\nu_n),\rho)\rightarrow 0,\,n\rightarrow\infty$, and we define $\Phi(\nu):=\rho$, which does not depend on the chosen of approximating sequence. Since $\mathrm{W}_1$-distance is dominated by $\mathrm{W}_2$-distance, we can show that the extension remains the $\mathrm{W}_1$-Lipschitz estimate (\ref{fixed point W1}). Finally, by a similar approximating procedure, we can show that $\Phi(\nu)$ still serves as the fixed point of (\ref{fixed-point map}).
\end{proof}
\par We can also show a stability result with respect to the given data for MFGC.
\begin{proposition}\label{fixed-point stability}
    Suppose $L^1,L^2$ satisfy Assumptions \ref{regularity assumption}, \ref{convexity assumption}, \ref{coercivity assumption}. Let $H^1,H^2$ be the corresponding Hamiltonians and $\Phi^1,\Phi^2$ be the corresponding fixed-point maps. Then we have the following stability estimate: for any $\nu=\mathcal{L}(X,Y)$,
    \begin{equation}
        \mathrm{W}_1(\Phi^1(\nu),\Phi^2(\nu))\leq C\E[|\partial_p H^1(X,Y,\Phi^2(\nu))-\partial_p H^2(X,Y,\Phi^2(\nu))|].
    \end{equation}
\end{proposition}
\begin{proof}
    Denote $\rho^i:=\Phi^i(\nu),\,i=1,2$. Note that the duality gives
    \begin{equation}
        \partial_a L^1\big(X,\partial_p H^1(X,Y,\rho^i),\rho^i\big)=-Y,\,i=1,2.
    \end{equation}
    Hence the strong convexity of $L^1$ implies that
    \begin{equation}
    \begin{aligned}
        &c_0|\partial_pH^1(X,P,\rho^1)-\partial_p H^1(X,P,\rho^2)|\\
        \leq&|\partial_a L^1\big(X,\partial_pH^1(X,P,\rho^1),\rho^1\big)-\partial_a L^1\big(X,\partial_p H^1(X,P,\rho^2),\rho^1\big)|\\
        =&|\partial_a L^1\big(X,\partial_pH^1(X,P,\rho^2),\rho^2\big)-\partial_a L^1\big(X,\partial_p H^1(X,P,\rho^2),\rho^1\big)|.\\
    \end{aligned}
    \end{equation}
    Since $\rho^i=\mathcal{L}\big(X,\partial_p H^i(X,Y,\rho^i)\big)$, by Assumption \ref{convexity assumption}(ii), we have
    \begin{equation}
        \E[|\partial_pH^1(X,P,\rho^1)-\partial_p H^1(X,P,\rho^2)|]\leq \frac{c_1}{c_0}\E[|\partial_p H^1(X,Y,\rho^1)-\partial_p H^2(X,Y,\rho^2)|].
    \end{equation}
    With this estimate, we have
    \begin{equation}
    \begin{aligned}
        &\E[|\partial_p H^1(X,Y,\rho^1)-\partial_p H^2(X,Y,\rho^2)|]\\
        \leq&\E[|\partial_p H^1(X,Y,\rho^1)-\partial_p H^1(X,Y,\rho^2)|+|\partial_p H^1(X,Y,\rho^2)-\partial_p H^2(X,Y,\rho^2)|]\\
        \leq&\frac{c_1}{c_0}\E[|\partial_p H^1(X,Y,\rho^1)-\partial_p H^2(X,Y,\rho^2)|+|\partial_p H^1(X,Y,\rho^2)-\partial_p H^2(X,Y,\rho^2)|].
    \end{aligned}
    \end{equation}
    Since $0<c_1<c_0$, we conclude that
    \begin{equation}
    \begin{aligned}
        \mathrm{W}_1(\rho^1,\rho^2)\leq&\E[|\partial_p H^1(X,Y,\rho^1)-\partial_p H^2(X,Y,\rho^2)|]\\
        \leq& C\E[|\partial_p H^1(X,Y,\rho^2)-\partial_p H^2(X,Y,\rho^2)|],
    \end{aligned}
    \end{equation}
    which completes the proof.
\end{proof}

\subsection{The second-order parabolic subdifferential and Hilbert approach}\label{Second-order parabolic subdifferential and Hilbert approach}
In this context, to define the monotone solution, we introduce the following function spaces on the Hilbert space and the notion of the second-order parabolic subdifferential. For further details, please refer to  \cite{cardaliaguet2022monotone,lions1989viscosity}.

Let $\mathbb{H}$ be a separable Hilbert space. For any $x, y \in \mathbb{H}$, let $\langle x, y\rangle$ denote the inner product and $\|x\|_2$ the norm. Let $L'(\mathbb{H})$ represent the space of all symmetric bounded bilinear forms on $\mathbb{H}$. Let $\mathbb{H}_N$ be an increasing sequence of finite-dimensional subspaces of $\mathbb{H}$ such that $\bigcup_N\mathbb{H}_N$ is dense in $\mathbb{H}$.
Denote by \(P_N\) and \(Q_N\) the orthogonal projections onto
\(\mathbb H_N\) and \(\mathbb H_N^\perp\), respectively.
In this paper, we take $$\mathbb{H}=\mathbb{L}^2(\mathcal{F}_0^1;\mathbb{R}^{2d}), \ \langle x,y\rangle_{\mathbb{H}}:=\mathbb{E}_1[x\cdot y],\ \|x\|_2:=\big(\mathbb{E}_1[\vert x\vert^2]\big)^{\frac{1}{2}}, \ \text{for any} \ x,y \in\mathbb{H}.$$
\par We say that $f\in \mathbb{C}^{1,2}([0,T]\times \mathbb{H})$ if $\partial_t f$, $D_xf$ and $D_{xx}f$ exist and are continuous on $[0,T]\times \mathbb{H}$. Here, $D_x f$ and $D_{xx} f$ denote the first-order and second-order Fréchet derivatives of $f$ with respect to $x$, respectively. 

Consider a lower semicontinuous map $f: [0,T]\times \mathbb{H}\rightarrow \mathbb{R}$. Given $t_0 \in [0,T]$ and $x_0 \in \mathbb{H}$, the second-order parabolic subdifferential of $f$ at the point $(t_0, x_0)$ is defined as follows:
\begin{equation*}
\begin{aligned}
    &D^{1,2,-}f(t_0,x_0)\\
    :=&\Big\{(q,p,\chi)\in \mathbb{R}\times \mathbb{H}\times L'(\mathbb{H})\Big\vert\liminf\limits_{|t-t_0|+\|x-x_0\|_2^2\rightarrow 0}\frac{1}{\vert t-t_0\vert+\|x-x_0\|^2_2}\Big[f(t,x)-f(t_0,x_0)\\
    &-q(t-t_0)-\langle p,x-x_0\rangle_{\mathbb{H}}-\frac{1}{2}\langle\chi(x-x_0),x-x_0\rangle_{\mathbb{H}}\Big]\geq0\Big\}.
\end{aligned}
\end{equation*}
The following lemma provides an equivalent characterization of the second-order parabolic subdifferential in terms of smooth test functions. We postpone its proof to the Appendix.
\begin{lemma}\label{relationship}
Let $(t_0,x_0)\in[0,T]\times\mathbb{H}$ be given. Then $(q,p,\chi)\in D^{1,2,-}f(t_0,x_0)$ if and only if there exists a function $\varphi\in \mathbb{C}^{1,2}([0,T]\times \mathbb{H})$ such that $f-\varphi$ attains a local minimum at $(t_0,x_0)$, and the following holds:
\begin{equation*}
    \begin{aligned}
        \big(\partial_t\varphi(t_0,x_0),D_{x}\varphi(t_0,x_0),
        D_{xx}\varphi(t_0,x_0)\big)=(q,p,\chi).
    \end{aligned}
\end{equation*}
\end{lemma}
Moreover, we write 
\begin{equation*}
\begin{aligned}
\overline{D}^{1,2,-} f(t,x):=&\Big\{(q, p,\chi) \in \mathbb{R}\times \mathbb{H}\times L^{\prime}(\mathbb{H})\Big\vert\ \text { there exist }\ \left(t_n, x_n, q_n, p_n, \chi_n\right) \in\\
&[0,T]\times \mathbb{H} \times \mathbb{R}\times \mathbb{H} \times L^{\prime}(\mathbb{H})  \text { such that } \left(q_n, p_n, \chi_n\right) \in D^{1,2,-} f\left(t_n, x_n\right)\\
&\text { and }\left(t_n, x_n, f(t_n, x_n),q_n, p_n, \chi_n\right) \xrightarrow {n \rightarrow \infty}(t, x, f(t,x), q, p, \chi)\Big\} .
\end{aligned}
\end{equation*}
\par Next we introduce the Hilbert space approach. Consider a fully nonlinear second-order degenerate parabolic equation of the following form
\begin{equation}\label{eq:fully_nonlinear}
    -\partial_t u(t, x) +f(x,u,D_xu,D_{xx}u)=0,\ u(T,x)=g(x),
\end{equation}
where 
$u:[0,T]\times\mathbb{H}\rightarrow\mathbb{R}$ is the unknown function. We assume that $f$ is degenerate elliptic with respect to its last argument, namely,
$f(x,u,p,X)\geq f(x,u,p,Y)$, for $X\leq Y, X,Y\in L'(\mathbb{H})$, where
\[
X\leq Y
\ \text{if and only if} \
(Y-X)(h,h)\geq0,
\ \text{for any}\ h\in\mathbb{H}.
\]
\par To handle the potential lack of smoothness in $u$, we adopt the framework of viscosity supersolutions, adapted to the Hilbert space setting.
\begin{definition}\label{def viscosity supersolution}
A lower semicontinuous function $u: [0, T] \times \mathbb{H} \to \mathbb{R}$ with $u(T,\cdot)=g$ is called a viscosity supersolution to \eqref{eq:fully_nonlinear} if, for every test function $\varphi \in \mathbb{C}^{1,2}([0, T] \times \mathbb{H})$ such that $u - \varphi$ attains a local minimum at $(t^*, x^*) \in [0, T) \times \mathbb{H}$, the following inequality holds:
\begin{equation*}
   0\leq -\partial_t\varphi(t^*,x^*)+f(x^*,u(t^*,x^*),D_x\varphi(t^*,x^*),D_{xx}\varphi(t^*,x^*)).
\end{equation*}
\end{definition}
We next establish a parabolic version of \cite[Lemma 4]{lions1989viscosity}, which may be regarded as an infinite-dimensional analogue of the Jensen-Ishii lemma from \cite{ishii1990viscosity}.
Before this, for $\alpha, \varepsilon>0$, we introduce the function $\Psi: ([0,T])^2 \times (\mathbb{H})^2 \to \mathbb{R}$ defined by
\begin{equation*}
\Psi(t,s,x,y)=u(t,x)+v(s,y)+\frac{1}{2\alpha}\| x-y\|^2_2+\frac{1}{2\alpha}\vert t-s \vert^2+\varepsilon\phi(t,x)+\varepsilon\phi(s,y),
\end{equation*} 
where $\phi$ goes to $+\infty$ at infinity.
\begin{lemma}\label{lem1}
For any $\delta > 0$, assume there exist $p_1, p_2 \in \mathbb{H}$ and $p_3, p_4 \in \mathbb{R}$ satisfying $\| p_1 \|_2 + \| p_2 \|_2 + |p_3| + |p_4| \leq \delta$, such that the perturbed function
\begin{equation*}
\Psi(t,s,x,y) - \langle p_1, x \rangle - \langle p_2, y \rangle - p_3 t - p_4 s
\end{equation*}
attains a strict minimum at $(\bar{t}, \bar{s}, \bar{x}, \bar{y})\in(0,T)^2\times(\mathbb{H})^2$. Then, for each $N \geq 1$, there exist operators $\mathcal{X}_N, \mathcal{Y}_N$ satisfying $\mathcal{X}_N = P_N \mathcal{X}_N P_N$ and $\mathcal{Y}_N = P_N \mathcal{Y}_N P_N$, such that
   \begin{equation}\label{matrixinequality}
    -\frac{1}{\alpha}\begin{pmatrix}
        I &-I\\
        -I &I
    \end{pmatrix}\leq \begin{pmatrix}
        \mathcal{X}_N &0\\
        0   &\mathcal{Y}_N
    \end{pmatrix} \leq \frac{2}{\alpha}\begin{pmatrix}
        I &0\\
        0 &I
    \end{pmatrix}.
   \end{equation}
   Moreover, the following conditions hold:
   \begin{equation}\label{t23}
     \left(-\frac{\bar{t}-\bar{s}}{\alpha},-\frac{\bar{x}-\bar{y}}{\alpha},\mathcal{X}_N-\frac{2}{\alpha}Q_N\right)\in \overline{D}^{1,2,-}u'(\bar{t},\bar{x}),  
   \end{equation}
   \begin{equation}\label{t24}
     \left(\frac{\bar{t}-\bar{s}}{\alpha},\frac{\bar{x}-\bar{y}}{\alpha},\mathcal{Y}_N-\frac{2}{\alpha}Q_N\right)\in \overline{D}^{1,2,-}v'(\bar{s},\bar{y}),  
   \end{equation}
   where 
   \begin{equation*}
   \begin{aligned}
     u'(t,x)=&u(t,x)+\varepsilon\phi(t,x)-\langle p_1,x\rangle-p_3t,\\
     v'(s,y)=&v(s,y)+\varepsilon\phi(s,y)-\langle p_2,y\rangle-p_4s. 
     \end{aligned}
   \end{equation*} 
\end{lemma}
\begin{proof}
Following the argument used in
\cite[Corollary 3.29]{FabbriGozziSwiech2017}, We can naturally extend $u$ and $v$ to the enlarged space $\widetilde{\mathbb{H}}:=\mathbb{R}\times \mathbb{H}$, while ensuring that all their structural properties are preserved and that the minimum is still attained at $(\bar{t},\bar{s},\bar{x},\bar{y})$. By applying \cite[Lemma 4]{lions1989viscosity}, one can derive \eqref{t23} and \eqref{t24}, and obtain $\widetilde{\mathcal{X}}_N, \widetilde{\mathcal{Y}}_N$ which satisfy an analogue of \eqref{matrixinequality} on $\widetilde{\mathbb{H}}\times \widetilde{\mathbb{H}}$. Restricting this relation to $\{0\}\times \mathbb{H}\times \{0\}\times \mathbb{H}$ then yields \eqref{matrixinequality}. 
\end{proof}
By requiring $\mathbb{H}_N$ to contain all constant random variables, we ensure that the operator $Q_N$ satisfies
\begin{equation}\label{Q_N equation}
\sum_{j=1}^dQ_N\big((\boldsymbol{e}_j,\boldsymbol{e}_j),
(\boldsymbol{e}_j,\boldsymbol{e}_j)\big)=0,    
\end{equation}
where $\{\boldsymbol{e}_j\}_{1\leq j\leq d}$ denotes the canonical basis of $\mathbb{R}^d$, which can also be viewed as constant random variables in $\mathbb{L}^2$.
\begin{lemma}\label{slope estimate}
Let
$u:[0,T]\times\mathbb H\longrightarrow\mathbb R$
be lower semicontinuous. Assume that there exists \(C>0\) such
that
\begin{equation}\label{eq:weighted-spatial-Lipschitz}
    |u(t,x)-u(t,y)|
    \leq
    C\bigl(1+\|x\|_2+\|y\|_2\bigr)\|x-y\|_2
\end{equation}
for every \(t\in[0,T]\) and \(x,y\in\mathbb H\).
If $(q,p,X)\in\overline D^{1,2,-}u(t,x)$,
then
\begin{equation*}
    \|p\|_2
    \leq C\bigl(1+\|x\|_2\bigr).
\end{equation*}
In particular, the same conclusion holds if
    $(q,p,X)\in D^{1,2,-}u(t,x)$.
\end{lemma}
\begin{proof}
Suppose that $(q,p,X)\in D^{1,2,-}u(t,x).$
If \(p=0\), the conclusion is immediate. Suppose that \(p\neq0\)
and set
\[
    e:=\frac{p}{\|p\|_2}.
\]
Let \(\varepsilon>0\), by the definition of \(D^{1,2,-}u(t,x)\),
\[
    \liminf_{\varepsilon\rightarrow0}
    \frac{u(t,x+\varepsilon e)-u(t,x)}{\varepsilon}
    \geq
    \langle p,e\rangle_{\mathbb H}
    =\|p\|_2.
\]
On the other hand, by
\eqref{eq:weighted-spatial-Lipschitz},
\[
\begin{aligned}
    \frac{|u(t,x+\varepsilon e)-u(t,x)|}{\varepsilon}
    &\leq
    C\bigl(
       1+\|x+\varepsilon e\|_2+\|x\|_2
     \bigr).
\end{aligned}
\]
Letting \(\varepsilon\rightarrow0\), we obtain
    $\|p\|_2
    \leq
    C\bigl(1+\|x\|_2\bigr).$
Now suppose that $(q,p,X)\in\overline D^{1,2,-}u(t,x).$
Then, there exist
    $(t_n,x_n,q_n,p_n,X_n)$
such that $(q_n,p_n,X_n)\in D^{1,2,-}u(t_n,x_n)$ and
\[
    (t_n,x_n,u(t_n,x_n),q_n,p_n,X_n)
    \longrightarrow
    (t,x,u(t,x),q,p,X).
\]
By the above result for $D^{1,2-}u(t,x)$,
    $\|p_n\|_2
    \leq C\bigl(1+\|x_n\|_2\bigr).$
Letting $n\rightarrow\infty$, we derive the results.
\end{proof}

\section{Well-posedness for FBSDE}
We first study the well-posedness for FBSDEs (\ref{FBSDE xi})-(\ref{FBSDE xi,x}), and use their decoupling fields to construct a candidate for the weak solution. The key is to derive uniform a priori Lipschitz estimates in $x$ and $\xi$ under $\mathrm{W}_1$-distance.
\par Recalling $\widehat{H}$ in Remark \ref{remark hat H} and $\Phi$ is Lipschitz continuous under $\mathrm{W}_1$, we know $\partial_p\widehat{H},\partial_p\widehat{H}$ are Lipschitz continuous, $\mathrm{W}_1$ in measure. We can obtain the following local well-posedness and stability result for these systems (\ref{FBSDE xi})-(\ref{FBSDE xi,x}) by standard FBSDE theory; see e.g. \cite[II, Theorem 5.4]{carmona2018probabilistic}.
\begin{proposition}\label{local FBSDE}
    Let Assumptions \ref{regularity assumption} and \ref{convexity assumption}(i) hold, where we only require the Lipschitz continuity of $G,L$ with respect to the measure variable under $\mathrm{W}_2$. Then there exists $\delta>0$ depending only on $d$ and $C_{Lip}$ in Assumption \ref{regularity assumption} such that whenever $T<\delta$, for any $t_0\in[0,T]$, $\xi\in\mathbb{L}^2(\mathcal{F}_{t_0})$ and $x\in\R^d$, the FBSDEs (\ref{FBSDE xi}) and (\ref{FBSDE xi,x}) have a unique strong solution. In particular, $\delta$ depends on the $\mathrm{W}_2$-Lipschitz constant of $\partial_x G$ with respect to $\mu$ instead of the $\mathrm{W}_1$-constant. Given the $(X^{\xi,x},\nabla Y^{\xi,x})$ in (\ref{FBSDE xi,x}) and $\mu,\rho$ in (\ref{FBSDE xi}), for $\xi\in\mathbb{L}^2(\mathcal{F}_{t_0}^1;\mu)$, the $V$ given by (\ref{define V}) is well-defined.
    \par Moreover, for any $(\bar{G},\bar{L})$ satisfying Assumption \ref{regularity assumption}, let $\bar{\widehat{H}}$ denote the corresponding composite Hamiltonian and let $(\bar{X}^{\xi,x},\nabla\bar{Y}^{\xi,x},\nabla\bar{Z}^{\xi,x},\nabla\bar{Z}^{0,\xi,x})$ be the solution, then we have the stability estimate
    \begin{equation}
    \begin{aligned}
        &\E\Big[\sup_{t\in[t_0,T]}(|X_t^{\xi,x}-\bar{X}_t^{\xi,x}|^2+|\nabla Y_t^{\xi,x}-\nabla\bar{Y}_t^{\xi,x}|^2)\\
        &\qquad+\int_{t_0}^T(|\nabla Z_t^{\xi,x}-\nabla\bar{Z}_t^{\xi,x}|^2+|\nabla Z_t^{0,\xi,x}-\nabla \bar{Z}_t^{0,\xi,x}|^2)\,\mathrm{d}t\Big]\\
        \leq&C\E\Big[|(\partial_x G-\partial_x\bar{G})(X_T^{\xi,x},\mu_T)|^2\\
        &\qquad+\int_{t_0}^T|(\partial_p \widehat{H}-\partial_p\bar{\widehat{H}},\partial_x H-\partial_x\bar{\widehat{H}})(X_t^{\xi,x},\nabla Y_t^{\xi,x},\nu_t)|^2\,\mathrm{d}t\Big],
    \end{aligned}
    \end{equation}
    where $\nu_t:=\mathcal{L}^0(X_t^{\xi},\nabla Y_t^{\xi})$.
\end{proposition}
\begin{remark}
    By Proposition \ref{fixed-point stability}, the above stability estimate can be given in terms of $H,\bar{H}$ rather than $\widehat{H},\bar{\widehat{H}}$. In fact, if we use $H,\Phi$ and $\bar{H},\bar{\Phi}$ to denote the corresponding Hamiltonian and fixed-point map, we have
    \[
    \begin{aligned}
        &\E[|\partial_p\widehat{H}(X_t^{\xi,x},\nabla Y_t^{\xi,x},\nu_t)-\partial_p\bar{\widehat{H}}(X_t^{\xi,x},\nabla Y_t^{\xi,x},\nu_t)|]\\
        =&\E[|\partial_pH\big(X_t^{\xi,x},\nabla Y_t^{\xi,x},\rho_t)-\partial_p\bar{H}\big(X_t^{\xi,x},\nabla Y_t^{\xi,x},\rho_t)|]\\
        \leq&\E[|\partial_pH\big(X_t^{\xi,x},\nabla Y_t^{\xi,x},\rho_t\big)-\partial_p\bar{H}\big(X_t^{\xi,x},\nabla Y_t^{\xi,x},\rho_t\big)|+|\partial_p\bar{H}\big(X_t^{\xi,x},\nabla Y_t^{\xi,x},\rho_t\big)-\partial_p\bar{H}\big(X_t^{\xi,x},\nabla Y_t^{\xi,x},\bar{\rho}_t\big)|]\\
        \leq& C\E[|\partial_pH\big(X_t^{\xi,x},\nabla Y_t^{\xi,x},\rho_t\big)-\partial_p\bar{H}\big(X_t^{\xi,x},\nabla Y_t^{\xi,x},\rho_t\big)|],
    \end{aligned}
    \]
    where $\rho_t:=\Phi(\nu_t),\,\bar{\rho}_t=\bar{\Phi}(\nu_t)$.
    \par We use $\widehat{H}$ in the above proposition due to the mollification argument in the later discussion. We shall directly mollify $\widehat{H}$ rather than $H$, so the stability in terms of $\widehat{H}$ will contribute to the convergence for these mollifiers.
\end{remark}
In \cite{gangbo2022displacement,mou2024nonsmooth}, the Lipschitz estimates in $x$ were derived from backward stochastic differential equation (BSDE) estimates, where Girsanov's theorem was applied and the presence of idiosyncratic noise was essential. Due to the possible degeneracy of the idiosyncratic noise in our framework, we instead deduce Lipschitz continuity in $x$ directly from the structure of the control problem.
\begin{proposition}\label{relation u and v LL}
Let Assumptions \ref{regularity assumption}, \ref{quadratic growth assumption}, \ref{convexity assumption} (i), \ref{G L convexity assumption} hold. Fix $t_0\in[0,T]$ and let $\{\rho_t\}_{t\in[t_0,T]}\in\mathbb{L}^2(\mathbb{F}^{B^0};\mathcal{P}_2(\R^{2d}))$, $\mu_t:=\pi_1\#\rho_t$, suppose that, for any $x\in\R^d$, FBSDE (\ref{FBSDE xi,x}) admits a unique solution $(X^{\xi,x},\nabla Y^{\xi,x},\nabla Z^{\xi,x},\nabla Z^{\xi,x,0})$. Define \begin{equation}\label{v definition}
v(t_0,x):=\E_{\mathcal{F}_{t_0}^0}\Big[G(X_T^{\xi,x},\mu_T)+\int_{t_0}^T\big[H(X_s^{\xi,x},\nabla Y_s^{\xi,x},\rho_s)-\nabla Y_s^{\xi,x}\cdot\partial_p H(X_s^{\xi,x},\nabla Y_s^{\xi,x},\rho_s)\big]\,\mathrm{d}s\Big].
\end{equation} 
Then,\\
(i) $v$ coincides with the value function of the following optimal control problem:   
$$v(t_0,x)=\mathrm{ess}\inf_{\alpha}\mathcal{J}(\alpha,t_0,x), \ \text{where} \ \mathcal{J}(\alpha,t_0,x)=\mathbb{E}_{\mathcal{F}_{t_0}^0}\Big[G(X_T^{\alpha,x},\mu_T)+\int_{t_0}^TL(X_t^{\alpha,x},\alpha_t,\rho_t)\mathrm{d}t\Big],$$
The essential infimum is taken over all admissible controls 
$\alpha$ that are progressively measurable and square-integrable, and the controlled state process $X_t^{\alpha,x}$ satisfies
\begin{equation*}  \mathrm{d}X_t^{\alpha,x}=\alpha_t\mathrm{d}t+\sigma\,\mathrm{d}B_t+\sigma_0 \,\mathrm{d}B^0_t,\ X_{t_0}^{\alpha,x}=x.
\end{equation*}
Moreover, $\partial_x v$ is Lipschitz continuous with respect to $x$, with a universal Lipschitz
constant.\\
(ii) Define $u(t,x):=\partial_x v(t,x)$. Then $u$ is the decoupling field of the FBSDE (\ref{FBSDE xi,x}) with the given $\rho$:
$$\nabla Y_{t_0}^{\xi,x}=\partial_xv(t_0,x),$$ 
\end{proposition}
\begin{proof}
(i) Let $\bar{\alpha}_t := \partial_p H(X_t^{\xi,x}, \nabla Y_t^{\xi,x}, \rho_t)$ be the candidate optimal control. By the Lipschitz continuity of $\partial_pH$, together with the square-integrability of $X^{\xi,x}, \nabla Y^{\xi,x}$, and $\rho$, we have
\begin{equation*}
\mathbb{E}[\int_{t_0}^T\vert \bar{\alpha}_t\vert^2\mathrm{d}t]\leq C\mathbb{E}[\int_{t_0}^T(1+\vert X_t^{\xi,x}\vert^2+\vert \nabla Y_t^{\xi,x}\vert^2+\|\rho_t\|^2_2)\mathrm{d}t]<\infty,    
\end{equation*}
which implies that $\bar{\alpha}$ is admissible. For any admissible control $\alpha$, define $\Delta \alpha := \alpha - \bar{\alpha}$ and $\Delta X := X^{\alpha,x} - X^{\xi,x}$. Applying Itô's formula to $\nabla Y_{t}^{\xi,x} \cdot \Delta X_t$, we obtain
\begin{equation}\label{Ito equality}
    \begin{aligned}
        &\mathbb{E}_{\mathcal{F}_{t_0}^0}\Big[\partial_xG(X_T^{\xi,x},\mu_T)\cdot \Delta X_T \Big]
        \\
        =&\mathbb{E}_{\mathcal{F}_{t_0}^0}\Big[\int_{t_0}^T\Big(\nabla Y_{t}^{\xi,x}\cdot\Delta \alpha_t -\Delta X_t\cdot \partial_x H(X_t^{\xi,x},\nabla Y_{t}^{\xi,x},\rho_t)\Big)\mathrm{d}t\Big]\\
        =-&\mathbb{E}_{\mathcal{F}_{t_0}^0}\Big[\int_{t_0}^T\Big(\partial_aL(X_t^{\xi,x},\bar{\alpha}_t,\rho_t)\cdot\Delta \alpha_t+\Delta X_t\cdot \partial_xL(X_t^{\xi,x},\bar{\alpha}_t,\rho_t)\Big)\mathrm{d}t\Big],
    \end{aligned}
\end{equation}
where the last equality follows from 
$$ \partial_aL(X_t^{\xi,x},\bar{\alpha}_t,\rho_t)=-\nabla Y_t^{\xi,x} \ \text{and} \ \partial_xH(X_t^{\xi,x},\nabla Y_t^{\xi,x},\rho_t)=\partial_xL(X_t^{\xi,x},\bar{\alpha}_t,\rho_t).$$ By the convexity of $G(\cdot,\mu)$ and $L(\cdot,\cdot,\rho)$, it follows that
\begin{equation*}
    \begin{aligned}  &\mathbb{E}_{\mathcal{F}_{t_0}^0}\Big[\int_{t_0}^T\big(L(X_t^{\xi,x},\bar{\alpha}_t,\rho_t)-L(X_t^{\alpha,x},\alpha_t,\rho_t)\big)\mathrm{d}t\Big]
        \leq \mathbb{E}_{\mathcal{F}_{t_0}^0}\Big[G(X_T^{\alpha,x},\mu_T)-G(X_T^{\xi,x},\mu_T)\Big].
    \end{aligned}
\end{equation*}
Therefore,
\begin{equation*}
    \mathcal{J}(\bar{\alpha},t_0,x)\leq \mathcal{J}(\alpha,t_0,x),
\end{equation*}
for any admissible $\alpha$. On the other hand, by the definition of $\bar{\alpha}$, the duality relation between $L$ and $H$, and \eqref{v definition}
$$v(t_0,x)=\mathbb{E}_{\mathcal{F}_{t_0}^0}\Big[G(X_T^{\xi,x},\mu_T)+\int_{t_0}^TL(X_t^{\xi,x},\bar{\alpha}_t,\rho_t)\mathrm{d}t\Big]=\mathcal{J}(\bar{\alpha},t_0,x).$$
Hence,
$$v(t_0,x)=\mathrm{ess}\inf_{\alpha}\mathcal{J}(\alpha,t_0,x).$$ 
To establish convexity, fix $x_1,x_2\in\mathbb R^d$ and $\theta\in[0,1]$. For $i=1,2$, let $\bar\alpha^i$ be the optimal control for the initial state $x_i$ constructed above, and let $X^i$ be the corresponding state process. Then
$$
\mathcal J(\bar\alpha^i,t_0,x_i)=v(t_0,x_i),
\qquad \text{a.s.}
$$
Set
$
x_\theta:=(1-\theta)x_1+\theta x_2$,
$
\alpha^\theta:=(1-\theta)\bar\alpha^1+\theta\bar\alpha^2.
$
The control $\alpha^\theta$ is admissible, with corresponding state process
$
X_t^{\alpha^\theta,x_\theta}
=(1-\theta)X_t^1+\theta X_t^2,
\ t\in[t_0,T].
$
By the convexity of $G(\cdot,\mu_T)$ and $L(\cdot,\cdot,\rho_t)$, it follows that
$$
\begin{aligned}
v(t_0,x_\theta)
&\leq \mathcal J(\alpha^\theta,t_0,x_\theta)\\
&\leq (1-\theta)\mathcal J(\bar\alpha^1,t_0,x_1)
+\theta\mathcal J(\bar\alpha^2,t_0,x_2)\\
&=(1-\theta)v(t_0,x_1)+\theta v(t_0,x_2),
\qquad \text{a.s.}
\end{aligned}
$$
Thus, $v(t_0,\cdot)$ is convex.
\par Under the Assumption \ref{regularity assumption}, 
$L(\cdot,\alpha,\rho)$ and 
$G(\cdot,\mu)$ are semiconcave with a universal constant $L$. Following an argument similar to that in \cite[Proposition 3.2]{bansil2025global}, $v(t_0,\cdot)$ is semiconcave in $x$ with a constant $C(1+T)$, uniformly in time. The combination of convexity and semiconcavity implies that 
$\partial_xv$ is Lipschitz continuous in 
$x$ with a universal constant.\\
(ii) Fix $x \in \mathbb{R}^d$ and let $\bar{\alpha}$ be the optimal control for the initial state $x$. For any $x' \in \mathbb{R}^d$, we have
\begin{equation*}
\begin{aligned}
X_t^{\bar{\alpha},x'}-X_t^{\bar{\alpha},x}=x'-x, \ \text{and} \ X^{\bar{\alpha},x}=X^{\xi,x}.
\end{aligned}
\end{equation*}
Since $v(t_0,x')\leq \mathcal{J}(\bar{\alpha},t_0,x')$, we have
\begin{equation*}
\begin{aligned}
&v(t_0,x')-v(t_0,x)\\
\leq &\mathbb{E}_{\mathcal{F}_{t_0}^0}\big[\int_{t_0}^T\big(L(X_t^{\bar{\alpha},x'},\bar{\alpha}_t,\rho_t)-L(X_t^{\xi,x},\bar{\alpha}_t,\rho_t)\big)\mathrm{d}t
+G(X_T^{\bar{\alpha},x'},\mu_T)-G(X_T^{\xi,x},\mu_T)\big]\\
\leq&\mathbb{E}_{\mathcal{F}_{t_0}^0}\big[\int_{t_0}^T\partial_xL(X_t^{\xi,x},\bar{\alpha}_t,\rho_t)\cdot(x'-x)\mathrm{d}t+\partial_xG(X_T^{\xi,x},\mu_T)\cdot(x'-x)\big]+C\vert x'-x\vert^2,
\end{aligned}
\end{equation*}
where the second inequality from the uniformly Lipschitz continuity of $\partial_xL$ and $\partial_xG$. By the BSDE for $\nabla Y^{\xi,x}$,
$$\mathbb{E}_{\mathcal{F}_{t_0}^0}\big[\int_{t_0}^T\partial_xL(X_t^{\xi,x},\bar{\alpha}_t,\rho_t)\mathrm{d}t+\partial_xG(X_T^{\xi,x},\mu_T)\big]=\nabla Y_{t_0}^{\xi,x}.$$
Consequently, 
$$v(t_0,x')-v(t_0,x)
\leq \nabla Y_{t_0}^{\xi,x}\cdot(x'-x)+C\vert x'-x\vert^2.$$
This implies that $\nabla Y_{t_0}^{\xi,x}$ belongs to the superdifferential of $v(t_0,\cdot)$ at $x$. Since $v(t_0,\cdot)$ is known to be differentiable from (i), the superdifferential reduces to the exact gradient, yielding $\nabla Y_{t_0}^{\xi,x} = \partial_x v(t_0,x)$. 

\end{proof}

\begin{proposition}\label{relation u and v disp lamda}  
    Let Assumptions \ref{regularity assumption}, \ref{quadratic growth assumption}, \ref{convexity assumption} (i), \ref{locally bounded assumption} hold. Assume further $G$ and $L$ satisfy the displacement $\lambda$-monotonicity conditions (\ref{G displacement}) (\ref{L displacement}). Then the results in Proposition \ref{relation u and v LL} still hold.
\end{proposition}
\begin{proof}
Following the canonical transformation in \cite{arnold1989mechanics,bansil2025hidden,bansil2025classical}, we define
$$G_{\lambda}(x,\mu)=G(x,\mu)+\frac{\lambda}{2}\vert x\vert^2,\ L_{\lambda}(x,a,\rho)=L(x,a,\rho)-\lambda x\cdot a,\ H_{\lambda}(x,p,\rho)=H(x,p-\lambda x,\rho).$$ By the displacement $\lambda$-monotonicity conditions (\ref{G displacement})(\ref{L displacement}), $G_{\lambda}$ and $L_{\lambda}$ are displacement monotone. Hence, by \cite[Lemmas 2.3 and 2.5]{meszaros2024mean}, $G_{\lambda}(\cdot,\mu)$ is convex in $x$ and $L_{\lambda}(\cdot,\cdot,\rho)$ is jointly convex in $(x,a)$. 
Thus, $G_{\lambda}$ and $L_{\lambda}$ satisfy the convexity assumptions of Proposition \ref{relation u and v LL}. Moreover, \(G_\lambda\) and \(L_\lambda\) inherit the regularity and quadratic growth properties required in Proposition \ref{relation u and v LL}.
For an admissible control $\alpha$, define
$$\mathcal{J}_{\lambda}(\alpha,t_0,x)=\mathbb{E}_{\mathcal{F}_{t_0}^0}\Big[G_{\lambda}(X_T^{\alpha,x},\mu_T)+\int_{t_0}^TL_{\lambda}(X_t^{\alpha,x},\alpha_t,\rho_t)\mathrm{d}t\Big].$$
Applying It\^o's formula to $|X^{\alpha,x}|^2$ and taking the
conditional expectation with respect to $\mathcal F_{t_0}^0$, we obtain
\begin{equation}\label{relationship J and Jlambda}
\mathcal{J}_{\lambda}(\alpha,t_0,x)=\mathcal{J}(\alpha,t_0,x)  +\frac{\lambda}{2}\vert x\vert^2+ \frac{\lambda d\widehat{\sigma}^2}{2}(T-t_0).
\end{equation}
In particular, the difference between $\mathcal{J}_{\lambda}$ and $\mathcal{J}$ is independent of the control $\alpha$. Hence the two control problems have the same optimal controls. 
Set $$\nabla Y_t^{\lambda,\xi,x}=\nabla Y_t^{\xi,x}+\lambda X_t^{\xi,x},\  \nabla Z_t^{\lambda,\xi,x}=\nabla Z_t^{\xi,x}+\lambda\sigma I_d,\ \nabla Z_t^{0,\lambda,\xi,x}=\nabla Z_t^{0,\xi,x}+\lambda\sigma_0 I_d.$$
Then
\begin{equation*}
    \begin{aligned}
        \nabla Y_t^{\lambda,\xi,x}=\partial_x G_{\lambda}(X_T^{\xi,x},\mu_T)+\int_t^T\partial_x H_{\lambda}(X_s^{\xi,x},\nabla Y_s^{\lambda,\xi,x},\rho_s)\,\mathrm{d}s
    -\int_t^T\nabla Z_s^{\lambda,\xi,x} \,\mathrm{d}B_s-\int_t^T\nabla Z_s^{0,\lambda,\xi,x}\,\mathrm{d}B_s^0.
    \end{aligned}
\end{equation*}
Define
$$\begin{aligned}v_{\lambda}(t_0,x):=\E_{\mathcal{F}_{t_0}^0}\Big[G_{\lambda}(X_T^{\xi,x},\mu_T)+\int_{t_0}^TL_{\lambda}(X_s^{\xi,x},\bar{\alpha}_s,\rho_s)\,\mathrm{d}s\Big],
\end{aligned}$$
where $\bar{\alpha}_t := \partial_p H(X_t^{\xi,x}, \nabla Y_t^{\xi,x}, \rho_t)$.
Therefore, Proposition \ref{relation u and v LL} applies to the transformed problem, then $$v_{\lambda}(t_0,x)=\mathrm{ess}\inf_{\alpha}\mathcal{J}_{\lambda}(\alpha,t_0,x),\ \nabla Y_{t_0}^{\lambda,\xi,x}=\partial_xv_{\lambda}(t_0,x).$$ 
By the definitions of \(v\) and \(v_\lambda\), together with \eqref{relationship J and Jlambda} applied to \(\bar\alpha\), we have $v_{\lambda}(t_0,x)=v(t_0,x)+\frac{\lambda}{2}\vert x\vert^2+ \frac{\lambda d\widehat{\sigma}^2}{2}(T-t_0)$. Since $v_{\lambda}(t_0,\cdot)$ is differentiable by Proposition \ref{relation u and v LL}, $\partial_xv_{\lambda}(t_0,x)=\partial_xv(t_0,x)+\lambda x$. Since 
$\nabla Y_{t_0}^{\lambda,\xi,x}=\nabla Y_{t_0}^{\xi,x}+\lambda x$, and $\nabla Y_{t_0}^{\lambda,\xi,x}=\partial_xv_{\lambda}(t_0,x)$, it follows that
$$\nabla Y_{t_0}^{\xi,x}=\partial_xv(t_0,x).
$$
Finally, by \eqref{relationship J and Jlambda},
$$v(t_0,x)=v_{\lambda}(t_0,x)-\frac{\lambda}{2}\vert x\vert^2- \frac{\lambda d \widehat{\sigma}^2}{2}(T-t_0)=\mathrm{ess}\inf_{\alpha}\mathcal{J}(\alpha,t_0,x).$$
Since $\partial_xv_{\lambda}(t_0,x)$ is Lipschitz continuous by Proposition \ref{relation u and v LL},
$$\partial_xv(t_0,x)=\partial_xv_{\lambda}(t_0,x)-\lambda x$$
is also Lipschitz continuous, with a constant independent of $t_0$. This proves the result.
\end{proof}
\par The following two theorem and their corollary establish the $\mathrm{W}_1$-Lipschitz estimates in $\xi$, which rely heavily on the Lasry-Lions monotone or displacement $\lambda$-monotone structure.
\begin{theorem}\label{LL Lipschitz}
    Let Assumptions \ref{regularity assumption}, \ref{convexity assumption}, \ref{G L convexity assumption} hold. Suppose G and $L$ satisfy the Lasry-Lions monotonicity conditions (\ref{G Lasry-Lions})(\ref{L Lasry-Lions}), and for any initial conditions $(x_i,\xi_i)\in\R^d\times\mathbb{L}^2(\mathcal{F}_{t_0}),\,i=1,2$, the FBSDEs (\ref{FBSDE xi})-(\ref{FBSDE xi,x}) have strong solutions $(\Psi^{\xi_i},\Psi^{\xi_i,x_i}),\,\Psi=X,\nabla Y,\nabla Z,\nabla Z^0$. Then we have the following stability estimate:
    \begin{equation}
        |\nabla Y_{t_0}^{\xi_1,x_1}-\nabla Y_{t_0}^{\xi_2,x_2}|\leq C\big(|x_1-x_2|+\E_{\mathcal{F}_{t_0}^0}[|\xi_1-\xi_2|]\big),
    \end{equation}
    where $C$ depends only on $T,d$ and the parameters in the assumptions.
\end{theorem}
\begin{proof}
    For $i=1,2$, given $\rho_t^i:=\Phi\big(\mathcal{L}^0(X_t^{\xi_i},\nabla Y_t^{\xi_i})\big)$. Then we can define the corresponding $u^i$ in Proposition \ref{relation u and v LL} and define $\nabla\mathcal{Y}_t^i:=u^j(t,X_t^{\xi_i}),\,\nabla\mathcal{Y}_t^{i,x}:=u^j(t,X_t^{\xi_i,x})$ for $j=3-i$. Then we define for $j=3-i$,
    \begin{equation}
        v^j(t_0,x)=\E_{\mathcal{F}_{t_0}^0}\Big[G(X_T^{\xi_i,x},\mu_T^j)+\int_{t_0}^T\big[H(X_s^{\xi_i,x},\nabla\mathcal{Y}_s^{i,x},\rho_s^j)-\nabla\mathcal{Y}_s^{i,x}\cdot\partial_p H(X_s^{\xi_i,x},\nabla Y_s^{\xi_i,x},\rho_s^i)\big]\,\mathrm{d}s\Big]
    \end{equation}
    Then we have 
    \begin{equation}
    \begin{aligned}
        v^j(t,X_t^{\xi_i})=\E_{\mathcal{F}_t}\Big[G(X_T^{\xi_i},\mu_T^j)+\int_t^T\big[H(X_s^{\xi_i},\nabla\mathcal{Y}_s^i,\rho_s^j)-\nabla\mathcal{Y}_s^i\cdot\partial_p H(X_s^{\xi_i},\nabla Y_s^{\xi_i},\rho_s^i)\big]\,\mathrm{d}s\Big].
    \end{aligned}
    \end{equation}
    By Proposition \ref{relation u and v LL}, we have $\partial_x v^i=u^i,\,i=1,2$. We also note that $Y_t^{\xi_i}=v^i(t,X_t^{\xi_i}),\,\nabla Y_t^{\xi_i}=u^i(t,X_t^{\xi_i}),\,i=1,2$. For simplicity, we denote $\mathcal{Y}_t^i:=v^j(t,X_t^{\xi_i})$.
    \par Denote $\alpha_t^i:=\partial_p H(X_t^{\xi_i},\nabla Y_t^{\xi_i},\rho_t^i)$ and $\beta_t^i:=\partial_p H(X_t^{\xi_i},\nabla\mathcal{Y}_t^i,\rho_t^j),\, j=3-i$. Using the relation between the Hamiltonian $H$ and the Lagrangian $L$, in particular $\partial_a L(X_t^{\xi_i},\alpha_t^i,\rho_t^i)=-\nabla Y_t^{\xi_i}$, $\partial_a L(X_t^{\xi_i},\beta_t^i,\rho_t^j)=-\nabla\mathcal{Y}_t^i$, we can compute that
    \[
    \begin{aligned}
        &\E_{\mathcal{F}_{t}^0}[Y_t^{\xi_1}-\mathcal{Y}_t^1+Y_t^{\xi_2}-\mathcal{Y}_t^2]\\
        =&\E_{\mathcal{F}_t^0}[G(X_T^{\xi_1},\mu_T^1)-G(X_T^{\xi_1},\mu_T^2)+G(X_T^{\xi_2},\mu_T^2)-G(X_T^{\xi_2},\mu_T^1)]\\
        &+\E_{\mathcal{F}_t^0}\int_t^T\Big[H(X_s^{\xi_1},\nabla Y_s^{\xi_1},\rho_s^1)-H(X_s^{\xi_1},\nabla\mathcal{Y}_s^1,\rho_s^2)\\
        &\qquad\qquad+H(X_s^{\xi_2},\nabla Y_s^{\xi_2},\rho_s^2)-H(X_s^{\xi_2},\nabla\mathcal{Y}_s^2,\rho_s^1)\\
        &\qquad\qquad-\nabla Y_s^{\xi_1}\cdot\partial_p H(X_s^{\xi_1},\nabla Y_s^{\xi_1},\rho_s^1)+\nabla\mathcal{Y}_s^1\cdot\partial_p H(X_s^{\xi_1},\nabla Y_s^{\xi_1},\rho_s^1)\\
        &\qquad\qquad-\nabla Y_s^{\xi_2}\cdot\partial_p H(X_s^{\xi_2},\nabla Y_s^{\xi_2},\rho_s^2)+\nabla\mathcal{Y}_s^2\cdot\partial_p H(X_s^{\xi_2},\nabla Y_s^{\xi_2},\rho_s^2)\Big]\,\mathrm{d}s\\
        =&\E_{\mathcal{F}_t^0}[G(X_T^{\xi_1},\mu_T^1)-G(X_T^{\xi_1},\mu_T^2)+G(X_T^{\xi_2},\mu_T^2)-G(X_T^{\xi_2},\mu_T^1)]\\
        &+\E_{\mathcal{F}_t^0}\int_t^T\Big[L(X_s^{\xi_1},\alpha_s^1,\rho_s^1)-L(X_s^{\xi_1},\alpha_s^1,\rho_s^2)+L(X_s^{\xi_2},\alpha_s^2,\rho_s^2)-L(X_s^{\xi_2},\alpha_s^2,\rho_s^1)\Big]\,\mathrm{d}s\\
        &+\E_{\mathcal{F}_t^0}\int_t^T\Big[L(X_s^{\xi_1},\alpha_s^1,\rho_s^2)-L(X_s^{\xi_1},\beta_s^1,\rho_s^2)-(\alpha_s^1-\beta_s^1)\cdot\partial_a L(X_s^{\xi_1},\beta_s^1,\rho_s^2)\\
        &\qquad\qquad+L(X_s^{\xi_2},\alpha_s^2,\rho_s^1)-L(X_s^{\xi_2},\beta_s^2,\rho_s^1)-(\alpha_s^2-\beta_s^2)\cdot\partial_a L(X_s^{\xi_2},\beta_s^2,\rho_s^1)\Big]\,\mathrm{d}s.
    \end{aligned}
    \]
    Since $G$ and $L$ satisfy the Lasry-Lions monotonicity conditions (\ref{G Lasry-Lions})(\ref{L Lasry-Lions}), and $L(x,a,\rho)$ is convex in $a$, we conclude that for any $t\in[t_0,T]$,
    \begin{equation}
        \E_{\mathcal{F}_t^0}[Y_t^{\xi_1}-\mathcal{Y}_t^1+Y_t^{\xi_2}-\mathcal{Y}_t^2]\geq 0.
    \end{equation}
    Using the monotonicity and the convexity for $L$, we have
    \[
    \begin{aligned}
        0\leq&\E_{\mathcal{F}_{t_0}^0}[Y_t^{\xi_1}-\mathcal{Y}_t^1+Y_t^{\xi_2}-\mathcal{Y}_t^2]\\
        =&\E_{\mathcal{F}_{t_0}^0}[Y_{t_0}^{\xi_1}-\mathcal{Y}_{t_0}^1+Y_{t_0}^{\xi_2}-\mathcal{Y}_{t_0}^2]\\
        &-\E_{\mathcal{F}_{t_0}^0}\int_{t_0}^t\Big[L(X_s^{\xi_1},\alpha_s^1,\rho_s^1)-L(X_s^{\xi_1},\alpha_s^1,\rho_s^2)+L(X_s^{\xi_2},\alpha_s^2,\rho_s^2)-L(X_s^{\xi_2},\alpha_s^2,\rho_s^1)\,\mathrm{d}s\Big]\\
        &-\E_{\mathcal{F}_{t_0}^0}\int_{t_0}^t\Big[L(X_s^{\xi_1},\alpha_s^1,\rho_s^2)-L(X_s^{\xi_1},\beta_s^1,\rho_s^2)-(\alpha_s^1-\beta_s^1)\cdot\partial_a L(X_s^{\xi_1},\beta_s^1,\rho_s^2)\\
        &\qquad\qquad+L(X_s^{\xi_2},\alpha_s^2,\rho_s^1)-L(X_s^{\xi_2},\beta_s^2,\rho_s^1)-(\alpha_s^2-\beta_s^2)\cdot\partial_a L(X_s^{\xi_2},\beta_s^2,\rho_s^1)\Big]\,\mathrm{d}s\\
        \leq&\E_{\mathcal{F}_{t_0}^0}[Y_{t_0}^{\xi_1}-\mathcal{Y}_{t_0}^1+Y_{t_0}^{\xi_2}-\mathcal{Y}_{t_0}^2]\\
        &-\frac{c_0}{2}\E_{\mathcal{F}_{t_0}^0}\int_{t_0}^t\Big[|\alpha_s^1-\beta_s^1|^2+|\alpha_s^2-\beta_s^2|^2\Big]\,\mathrm{d}s.
    \end{aligned}
    \]
    This implies
    \begin{equation}\label{LL Lipschitz 1}
    \begin{aligned}
        &\E_{\mathcal{F}_{t_0}^0}\int_{t_0}^t\Big[|\alpha_s^1-\beta_s^1|^2+|\alpha_s^2-\beta_s^2|^2\Big]\,\mathrm{d}s\\
        \leq&C\E_{\mathcal{F}_{t_0}^0}[Y_{t_0}^{\xi_1}-\mathcal{Y}_{t_0}^1+Y_{t_0}^{\xi_2}-\mathcal{Y}_{t_0}^2]\\
        =&C\E_{\mathcal{F}_{t_0}^0}[v^1(t_0,\xi_1)-v^2(t_0,\xi_1)+v^2(t_0,\xi_2)-v^1(t_0,\xi_2)]\\
        =&C\E_{\mathcal{F}_{t_0}^0}\int_0^1\big(\partial_x v^1(t_0,\xi_{\theta})-\partial_x v^2(t_0,\xi_{\theta})\big)\cdot(\xi_1-\xi_2)\,\mathrm{d}\theta,
    \end{aligned}
    \end{equation}
    where $\xi_{\theta}:=\theta\xi_1+(1-\theta)\xi_2$.
    \par The uniform Lipschitz continuity of $\partial_x v^i(t,\cdot)$ implies $|\nabla Y_t^{\xi_i}-\nabla\mathcal{Y}_t^{j}|\leq C|X_t^{\xi_1}-X_t^{\xi_2}|$, so we have
    \begin{equation}\label{LL Lipschitz 2}
        |\alpha_t^1-\alpha_t^2|\leq|\alpha_t^1-\beta_t^1|+|\beta_t^1-\alpha_t^2|\leq|\alpha_t^1-\beta_t^1|+C|X_t^{\xi_1}-X_t^{\xi_2}|.
    \end{equation}
    From the SDE for $X^{\xi_i}$ (\ref{FBSDE xi}), using the estimates (\ref{LL Lipschitz 1})(\ref{LL Lipschitz 2}), we have
    \[
    \begin{aligned}
        &\E_{\mathcal{F}_{t_0}^0}[|X_t^{\xi_1}-X_t^{\xi_2}|]^2\\
        \leq&C\Big(\E_{\mathcal{F}_{t_0}^0}[|\xi_1-\xi_2|]^2+\int_{t_0}^t\E_{\mathcal{F}_{t_0}^0}[|\alpha_s^1-\alpha_s^2|]^2\,\mathrm{d}s\Big)\\
        \leq&C\Big(\E_{\mathcal{F}_{t_0}^0}[|\xi_1-\xi_2|]^2+\int_{t_0}^t\E_{\mathcal{F}_{t_0}^0}[|X_s^{\xi_1}-X_s^{\xi_2}|]^2\,\mathrm{d}s\\
        &\qquad+\E_{\mathcal{F}_{t_0}^0}\int_0^1\big(\partial_x v^1(t_0,\xi_{\theta})-\partial_x v^2(t_0,\xi_{\theta})\big)\cdot(\xi_1-\xi_2)\,\mathrm{d}\theta\Big).
    \end{aligned}
    \]
    Applying Gronwall's inequality, we conclude that
    \begin{equation}\label{LL Lipschitz 3}
    \begin{aligned}
        &\E_{\mathcal{F}_{t_0}^0}[|X_t^{\xi_1}-X_t^{\xi_2}|]^2\\
        \leq&C\Big(\E_{\mathcal{F}_{t_0}^0}[|\xi_1-\xi_2|]^2+\E_{\mathcal{F}_{t_0}^0}\int_0^1\big(\partial_x v^1(t_0,\xi_{\theta})-\partial_x v^2(t_0,\xi_{\theta})\big)\cdot(\xi_1-\xi_2)\,\mathrm{d}\theta\Big).
    \end{aligned}
    \end{equation}
    We shall repeatedly use the following estimates,
    \begin{equation}
        \mathrm{W}_1(\mu_t^1,\mu_t^2)\leq\E_{\mathcal{F}_t^0}[|X_t^{\xi_1}-X_t^{\xi_2}|],\quad\mathrm{W}_1(\rho_t^1,\rho_t^2)\leq\E_{\mathcal{F}_t^0}[|X_t^{\xi_1}-X_t^{\xi_2}|+|\alpha_t^1-\alpha_t^2|]
    \end{equation}
    combined with the tower property.
    \par Similarly, from the SDE for $X^{\xi_i,x_i}$ (\ref{FBSDE xi,x}), using (\ref{LL Lipschitz 1})(\ref{LL Lipschitz 2})(\ref{LL Lipschitz 3}), we can obtain the estimate
    \begin{equation}\label{LL Lipschitz 4}
    \begin{aligned}
        &\E_{\mathcal{F}_t^0}[|X_t^{\xi_1,x_1}-X_t^{\xi_2,x_2}|]^2\\
        \leq&C\Big(|x_1-x_2|^2+\E_{\mathcal{F}_{t_0}^0}[|\xi_1-\xi_2|]^2+\int_{t_0}^t\E_{\mathcal{F}_{t_0}^0}[|\nabla Y_s^{\xi_1,x_1}-\nabla Y_s^{\xi_2,x_2}|]^2\,\mathrm{d}s\\
        &\quad+\E_{\mathcal{F}_{t_0}^0}\int_0^1\big(\partial_x v^1(t_0,\xi_{\theta})-\partial_x v^2(t_0,\xi_{\theta})\big)\cdot(\xi_1-\xi_2)\,\mathrm{d}\theta\Big).
    \end{aligned}
    \end{equation}
    Next, from the BSDE for $\nabla Y^{\xi_i,x_i}$, combining with (\ref{LL Lipschitz 1})(\ref{LL Lipschitz 2})(\ref{LL Lipschitz 3})(\ref{LL Lipschitz 4}), we conclude that
    \[
    \begin{aligned}
        &|\nabla Y_{t_0}^{\xi_1,x_1}-\nabla Y_{t_0}^{\xi_2,x_2}|^2\\
        \leq&C\Big(\E_{\mathcal{F}_{t_0}^0}[|X_T^{\xi_1,x_1}-X_T^{\xi_2,x_2}|]^2+\E_{\mathcal{F}_{t_0}^0}[|X_T^{\xi_1}-X_T^{\xi_2}|]^2+\int_{t_0}^T\Big[\E_{\mathcal{F}_{t_0}^0}[|X_t^{\xi_1,x_1}-X_t^{\xi_2,x_2}|]^2\\
        &\qquad+\E_{\mathcal{F}_{t_0}^0}[|\nabla Y_t^{\xi_1,x_1}-\nabla Y_t^{\xi_2,x_2}|]^2+\E_{\mathcal{F}_{t_0}^0}[|X_t^{\xi_1}-X_t^{\xi_2}|]^2+\E_{\mathcal{F}_{t_0}^0}[|\alpha_t^1-\alpha_t^2|]^2\Big]\,\mathrm{d}t\Big)\\
        \leq&C\Big(|x_1-x_2|^2+\E_{\mathcal{F}_{t_0}^0}[|\xi_1-\xi_2|]^2+\int_{t_0}^T\E_{\mathcal{F}_{t_0}^0}[|\nabla Y_t^{\xi_1,x_1}-\nabla Y_t^{\xi_2,x_2}|]^2\,\mathrm{d}t\\
        &\qquad+\E_{\mathcal{F}_{t_0}^0}\int_0^1\big(\partial_x v^1(t_0,\xi_{\theta})-\partial_x v^2(t_0,\xi_{\theta})\big)\cdot(\xi_1-\xi_2)\,\mathrm{d}\theta\Big),
    \end{aligned}
    \]
    which implies
    \begin{equation}\label{LL Lipschitz 5}
    \begin{aligned}
        &[|\nabla Y_{t_0}^{\xi_1,x_1}-\nabla Y_{t_0}^{\xi_2,x_2}|]^2\\
        \leq&C\Big(|x_1-x_2|^2+\E_{\mathcal{F}_{t_0}^0}[|\xi_1-\xi_2|]^2+\E_{\mathcal{F}_{t_0}^0}\int_0^1\big(\partial_x v^1(t_0,\xi_{\theta})-\partial_x v^2(t_0,\xi_{\theta})\big)\cdot(\xi_1-\xi_2)\,\mathrm{d}\theta\Big).
    \end{aligned}
    \end{equation}  
    Setting $x_1=x_2=x$ and taking supremum over $x\in\R^d$, we obtain
    \begin{equation}
    \begin{aligned}
        &\|\partial_x v^1(t_0,\cdot)-\partial_x v^2(t_0,\cdot)\|_{L^{\infty}}^2\\
        \leq& C\Big(\E_{\mathcal{F}_{t_0}^0}[|\xi_1-\xi_2|]^2+\E_{\mathcal{F}_{t_0}^0}\int_0^1\big(\partial_x v^1(t_0,\xi_{\theta})-\partial_x v^2(t_0,\xi_{\theta})\big)\cdot(\xi_1-\xi_2)\,\mathrm{d}\theta\Big)\\
        \leq&C\Big(\E_{\mathcal{F}_{t_0}^0}[|\xi_1-\xi_2|]^2+\|\partial_x v^1(t_0,\cdot)-\partial_x v^2(t_0,\cdot)\|_{L^{\infty}}\E_{\mathcal{F}_{t_0}^0}[|\xi_1-\xi_2|]\Big).
    \end{aligned}
    \end{equation}
    Young's inequality gives
    \begin{equation}
        \|\partial_x v^1(t_0,\cdot)-\partial_x v^2(t_0,\cdot)\|_{L^{\infty}}\leq C\E_{\mathcal{F}_{t_0}^0}[|\xi_1-\xi_2|].
    \end{equation}
    Finally substituting back into (\ref{LL Lipschitz 5}), we conclude that
    \begin{equation}
        |\nabla Y_{t_0}^{\xi_1,x_1}-\nabla Y_{t_0}^{\xi_2,x_2}|^2\leq C\big(|x_1-x_2|^2+\E_{\mathcal{F}_{t_0}^0}[|\xi_1-\xi_2|]^2\big),
    \end{equation}
    which completes the proof.
\end{proof}

\begin{theorem}\label{disp Lipschitz}
    Let Assumptions \ref{regularity assumption}, \ref{convexity assumption}, \ref{locally bounded assumption} hold. Suppose G and $L$ satisfies the displacement $\lambda$-monotonicity conditions (\ref{G displacement})(\ref{L displacement}), and for any initial conditions $(x_i,\xi_i)\in\R^d\times\mathbb{L}^2(\mathcal{F}_{t_0}),\,i=1,2$ the FBSDEs (\ref{FBSDE xi})-(\ref{FBSDE xi,x}) have a strong solution $(\Psi^{\xi_i},\Psi^{\xi_i,x_i}),\,\Psi=X,\nabla Y,\nabla Z,\nabla Z^0$. Then we have the following stability estimate:
    \begin{equation}
        |\nabla Y_{t_0}^{\xi_1,x_1}-\nabla Y_{t_0}^{\xi_2,x_2}|\leq C\big(|x_1-x_2|+(\E_{\mathcal{F}_{t_0}^0}[|\xi_1-\xi_2|^2])^{\frac12}\big),
    \end{equation}
    where $C$ depends only on $T,d$ and the parameters in the assumptions.
\end{theorem}
\begin{proof}
    By It\^{o}'s formula, we can compute that
    \[
    \begin{aligned}
        &\E_{\mathcal{F}_t^0}[(\nabla Y_t^{\xi_1}-\nabla Y_t^{\xi_2})\cdot(X_t^{\xi_1}-X_t^{\xi_2})+\lambda|X_t^{\xi_1}-X_t^{\xi_2}|^2]\\
        =&\E_{\mathcal{F}_t^0}[\big(\partial_x G(X_T^{\xi_1},\mu_T^1)-\partial_x G(X_T^{\xi_2},\mu_T^2)\big)\cdot(X_T^{\xi_1}-X_T^{\xi_2})+\lambda|X_T^{\xi_1}-X_T^{\xi_2}|^2]\\
        &+\int_t^T\E_{\mathcal{F}_t^0}\Big[\big(\partial_x H(X_s^{\xi_1},\nabla Y_s^{\xi_1},\rho_s^1)-\partial_x H(X_s^{\xi_2},\nabla Y_s^{\xi_2},\rho_s^2)\big)\cdot(X_s^{\xi_1}-X_s^{\xi_2})\\
        &\qquad\qquad-(\nabla Y_s^{\xi_1}-\nabla Y_s^{\xi_2})\cdot\big(\partial_p H(X_s^{\xi_1},\nabla Y_s^{\xi_1},\rho_s^1)-\partial_p H(X_s^{\xi_2},\nabla Y_s^{\xi_2},\rho_s^2)\big)\\
        &\qquad\qquad-2\lambda(X_s^{\xi_1}-X_s^{\xi_2})\cdot\big(\partial_p H(X_s^{\xi_1},\nabla Y_s^{\xi_1},\rho_s^1)-\partial_p H(X_s^{\xi_2},\nabla Y_s^{\xi_2},\rho_s^2)\big)\Big]\,\mathrm{d}s\\
        =&\E_{\mathcal{F}_t^0}[\big(\partial_x G(X_T^{\xi_1},\mu_T^1)-\partial_x G(X_T^{\xi_2},\mu_T^2)\big)\cdot(X_T^{\xi_1}-X_T^{\xi_2})+\lambda|X_T^{\xi_1}-X_T^{\xi_2}|^2]\\
        &+\int_t^T\E_{\mathcal{F}_t^0}\Big[\big(\partial_x L(X_s^{\xi_1},\alpha_s^1,\rho_s^1)-\partial_x L(X_s^{\xi_2},\alpha_s^2,\rho_s^2)\big)\cdot(X_s^{\xi_1}-X_s^{\xi_2})\\
        &\qquad\qquad+\big(\partial_a L(X_s^{\xi_1},\alpha_s^1,\rho_s^1)-\partial_a L(X_s^{\xi_2},\alpha_s^2,\rho_s^2)\big)\cdot(\alpha_s^1-\alpha_s^2)\\
        &\qquad\qquad-2\lambda(X_s^{\xi_1}-X_s^{\xi_2})\cdot(\alpha_s^1-\alpha_s^2)\Big]\,\mathrm{d}s,
    \end{aligned}
    \]
    where we use the relation between the Hamiltonian $H$ and Lagrangian $L$. Since $G$ and $L$ satisfy the displacement $\lambda$-monotonicity (\ref{G displacement})(\ref{L displacement}), we conclude that for any $t\in[t_0,T]$,
    \begin{equation}\label{v dis lambda}
        \E_{\mathcal{F}_t^0}[(\nabla Y_t^{\xi_1}-\nabla Y_t^{\xi_2})\cdot(X_t^{\xi_1}-X_t^{\xi_2})+\lambda|X_t^{\xi_1}-X_t^{\xi_2}|^2]\geq 0.
    \end{equation}
    Now integrating from $t_0$ to $t$, using Assumptions \ref{regularity assumption}, \ref{convexity assumption}, we have
    \[
    \begin{aligned}
        0\leq&\E_{\mathcal{F}_{t_0}^0}[(\nabla Y_t^{\xi_1}-\nabla Y_t^{\xi_2})\cdot(X_t^{\xi_1}-X_t^{\xi_2})+\lambda|X_t^{\xi_1}-X_t^{\xi_2}|^2]\\
        =&\E_{\mathcal{F}_{t_0}^0}[(\nabla Y_{t_0}^{\xi_1}-\nabla Y_{t_0}^{\xi_2})\cdot(\xi_1-\xi_2)+\lambda|\xi_1-\xi_2|^2]\\
        &+\E_{\mathcal{F}_{t_0}^0}\int_{t_0}^t\Big[-\big(\partial_x L(X_s^{\xi_1},\alpha_s^1,\rho_s^1)-\partial_x L(X_s^{\xi_2},\alpha_s^2,\rho_s^2)\big)\cdot(X_s^{\xi_1}-X_s^{\xi_2})\\
        &\qquad\qquad-\big(\partial_a L(X_s^{\xi_1},\alpha_s^1,\rho_s^1)-\partial_a L(X_s^{\xi_2},\alpha_s^2,\rho_s^2)\big)\cdot(\alpha_s^1-\alpha_s^2)\\
        &\qquad\qquad+2\lambda(X_s^{\xi_1}-X_s^{\xi_2})\cdot(\alpha_s^1-\alpha_s^2)\Big]\,\mathrm{d}s\\
        \leq&\E_{\mathcal{F}_{t_0}^0}[(\nabla Y_{t_0}^{\xi_1}-\nabla Y_{t_0}^{\xi_2})\cdot(\xi_1-\xi_2)+\lambda|\xi_1-\xi_2|^2]\\
        &+\int_{t_0}^t\Big[C\E_{\mathcal{F}_{t_0}^0}[|X_s^{\xi_1}-X_s^{\xi_2}|^2]+C\E_{\mathcal{F}_{t_0}^0}\big[\E_{\mathcal{F}_s^0}[|X_s^{\xi_1}-X_s^{\xi_2}|]^2\big]\\
        &\qquad\quad+C\E_{\mathcal{F}_{t_0}^0}\big[\E_{\mathcal{F}_s^0}[|X_s^{\xi_1}-X_s^{\xi_2}|]\E_{\mathcal{F}_s^0}[|\alpha_s^1-\alpha_s^2|]\big]+C\E_{\mathcal{F}_{t_0}^0}[|X_s^{\xi_1}-X_s^{\xi_2}|\cdot|\alpha_s^1-\alpha_s^2|]\\
        &\qquad\quad-c_0\E_{\mathcal{F}_{t_0}^0}[|\alpha_s^1-\alpha_s^2|^2]+c_1\E_{\mathcal{F}_{t_0}^0}\big[\E_{\mathcal{F}_s^0}[|\alpha_s^1-\alpha_s^2|]^2\big]\Big]\,\mathrm{d}s
    \end{aligned}
    \]
    Note that we have $c_1<c_0$ from Assumption \ref{convexity assumption}(ii). By Jensen's inequality and Young's inequality, we obtain
    \begin{equation}\label{disp Lipschitz 1}
    \begin{aligned}
        &\E_{\mathcal{F}_{t_0}^0}\int_{t_0}^t|\alpha_s^1-\alpha_s^2|^2\,\mathrm{d}s\\
        \leq&C\E_{\mathcal{F}_{t_0}^0}[(\nabla Y_{t_0}^{\xi_1}-\nabla Y_{t_0}^{\xi_2})\cdot(\xi_1-\xi_2)+\lambda|\xi_1-\xi_2|^2]+C\int_{t_0}^t\E_{\mathcal{F}_{t_0}^0}[|X_s^{\xi_1}-X_s^{\xi_2}|^2\,\mathrm{d}s\\
        \leq&C\Big(\E_{\mathcal{F}_{t_0}^0}[|\xi_1-\xi_2|^2]+\E_{\mathcal{F}_{t_0}^0}[|\nabla Y_{t_0}^{\xi_1}-\nabla Y_{t_0}^{\xi_2}|^2]^{\frac12}\E_{\mathcal{F}_{t_0}^0}[|\xi_1-\xi_2|^2]^{\frac12}+\int_{t_0}^t\E_{\mathcal{F}_{t_0}^0}[|X_s^{\xi_1}-X_s^{\xi_2}|^2]\,\mathrm{d}s\Big).
    \end{aligned}
    \end{equation}
    From the SDE for $X^{\xi_i}$ (\ref{FBSDE xi}), applying (\ref{disp Lipschitz 1}) and Gronwall's inequality, we have
    \begin{equation}\label{disp Lipschitz 2}
        \E_{\mathcal{F}_{t_0}^0}[|X_t^{\xi_1}-X_t^{\xi_2}|^2]
        \leq C\Big(\E_{\mathcal{F}_{t_0}^0}[\xi_1-\xi_2|^2]+\E_{\mathcal{F}_{t_0}^0}[|\nabla Y_{t_0}^{\xi_1}-\nabla Y_{t_0}^{\xi_2}|^2]^{\frac12}\E_{\mathcal{F}_{t_0}^0}[|\xi_1-\xi_2|^2]^{\frac12}\Big).
    \end{equation}
    Similarly, from the SDE for (\ref{FBSDE xi,x}), using (\ref{disp Lipschitz 1})(\ref{disp Lipschitz 2}), we have
    \begin{equation}\label{disp Lipschitz 3}
    \begin{aligned}
        &\E_{\mathcal{F}_{t_0}^0}[|X_t^{\xi_1,x_1}-X_t^{\xi_2,x_2}|^2]\\
        \leq&C\Big(|x_1-x_2|^2+\E_{\mathcal{F}_{t_0}^0}[|\xi_1-\xi_2|^2]+C\E_{\mathcal{F}_{t_0}^0}[|\nabla Y_{t_0}^{\xi_1}-\nabla Y_{t_0}^{\xi_2}|^2]^{\frac12}\E_{\mathcal{F}_{t_0}^0}[|\xi_1-\xi_2|^2]^{\frac12}\\
        &\quad+\int_{t_0}^t\E_{\mathcal{F}_{t_0}^0}[|\nabla Y_s^{\xi_1,x_1}-\nabla Y_s^{\xi_2,x_2}|^2]\,\mathrm{d}s\Big).
    \end{aligned}
    \end{equation}
    Combining with (\ref{disp Lipschitz 1})(\ref{disp Lipschitz 2})(\ref{disp Lipschitz 3}), we estimate the BSDE for $\nabla Y^{\xi_i,x_i}$ (\ref{FBSDE xi,x})
    \begin{equation}
    \begin{aligned}
        &|\nabla Y_{t_0}^{\xi_1,x_1}-\nabla Y_{t_0}^{\xi_2,x_2}|^2\\
        \leq&C\Big(|x_1-x_2|^2+\E_{\mathcal{F}_{t_0}^0}[|\xi_1-\xi_2|^2]+\E_{\mathcal{F}_{t_0}^0}[|\nabla Y_{t_0}^{\xi_1}-\nabla Y_{t_0}^{\xi_2}|^2]^{\frac12}\E_{\mathcal{F}_{t_0}^0}[|\xi_1-\xi_2|^2]^{\frac12}\\
        &\quad+\int_{t_0}^T\E_{\mathcal{F}_{t_0}^0}[|\nabla Y_t^{\xi_1,x_1}-\nabla Y_t^{\xi_2,x_2}|^2]\,\mathrm{d}t\Big),
    \end{aligned}
    \end{equation}
    which implies by Gronwall's inequality
    \begin{equation}\label{disp lipschitz 4}
    \begin{aligned}
        &|\nabla Y_{t_0}^{\xi_1,x_1}-\nabla Y_{t_0}^{\xi_2,x_2}|^2\\
        \leq&C\Big(|x_1-x_2|^2+\E_{\mathcal{F}_{t_0}^0}[|\xi_1-\xi_2|^2]+\E_{\mathcal{F}_{t_0}^0}[|\nabla Y_{t_0}^{\xi_1}-\nabla Y_{t_0}^{\xi_2}|^2]^{\frac12}\E_{\mathcal{F}_{t_0}^0}[|\xi_1-\xi_2|^2]^{\frac12}\Big).
    \end{aligned}
    \end{equation}
    Finally, we need to estimate $\E_{\mathcal{F}_{t_0}^0}[|\nabla Y_{t_0}^{\xi_1}-\nabla Y_{t_0}^{\xi_2}|^2]$. Recall that $\nabla Y_{t_0}^{\xi_i}=\partial_x v^i(t_0,\xi_i)$ and $\nabla Y_{t_0}^{\xi_i,x}=\partial_x v^i(t_0,x)$ in the proof of Theorem \ref{LL Lipschitz}. Setting $x_1=x_2=x$ in (\ref{disp lipschitz 4}), we have
    \begin{equation}
    \begin{aligned}
        \|\partial_x v^1(t_0,\cdot)-\partial_x v^2(t_0,\cdot)\|^2_{L^{\infty}}\leq C\Big(\E_{\mathcal{F}_{t_0}^0}[|\xi_1-\xi_2|^2]+\E_{\mathcal{F}_{t_0}^0}[|\nabla Y_{t_0}^{\xi_1}-\nabla Y_{t_0}^{\xi_2}|^2]^{\frac12}\E_{\mathcal{F}_{t_0}^0}[|\xi_1-\xi_2|^2]^{\frac12}\Big).
    \end{aligned}
    \end{equation}
    Using the above estimate, by the Lipschitz continuity of $\partial_x v^i(t_0,\cdot)$, we have
    \begin{equation}
    \begin{aligned}
        &\E_{\mathcal{F}_{t_0}^0}[|\nabla Y_{t_0}^{\xi_1}-\nabla Y_{t_0}^{\xi_2}|^2]\\
        =&\E_{\mathcal{F}_{t_0}^0}[|\partial_x v^1(t_0,\xi_1)-\partial_x v^2(t_0,\xi_2)|^2]\\
        \leq&C\Big(\E_{\mathcal{F}_{t_0}^0}[|\partial_x v^1(t_0,\xi_1)-\partial_x v^1(t_0,\xi_2)|^2]+\E_{\mathcal{F}_{t_0}^0}[|\partial_x v^1(t_0,\xi_2)-\partial_x v^2(t_0,\xi_2)|^2]\Big)\\
        \leq&C\Big(\E_{\mathcal{F}_{t_0}^0}[|\xi_1-\xi_2|^2]+\|\partial_x v^1(t_0,\cdot)-\partial_x v^2(t_0,\cdot)\|_{L^{\infty}}^2\Big)\\
        \leq&C\Big(\E_{\mathcal{F}_{t_0}^0}[|\xi_1-\xi_2|^2]+\E_{\mathcal{F}_{t_0}^0}[|\nabla Y_{t_0}^{\xi_1}-\nabla Y_{t_0}^{\xi_2}|^2]^{\frac12}\E_{\mathcal{F}_{t_0}^0}[|\xi_1-\xi_2|^2]^{\frac12}\Big),
    \end{aligned}
    \end{equation}
    so Young's inequality gives
    \begin{equation}
        \E_{\mathcal{F}_{t_0}^0}[|\nabla Y_{t_0}^{\xi_1}-\nabla Y_{t_0}^{\xi_2}|^2]\leq C\E_{\mathcal{F}_{t_0}^0}[|\xi_1-\xi_2|^2].
    \end{equation}
    Substituting this into (\ref{disp lipschitz 4}), we conclude by Young's inequality that
    \begin{equation}
        |\nabla Y_{t_0}^{\xi_1,x_1}-\nabla Y_{t_0}^{\xi_2,x_2}|^2\leq C\Big(|x_1-x_2|^2+\E_{\mathcal{F}_{t_0}^0}[|\xi_1-\xi_2|^2]\Big),
    \end{equation}
    which completes the proof.
\end{proof}
Note that in Theorem \ref{LL Lipschitz}, we obtain the $\mathrm{W}_1$-Lipschitz estimate with respect to the measure variable under Lasry-Lions monotonicity condition. In fact, applying the method of representation formula for classical solutions, we can show that the $\mathrm{W}_2$-Lipschitz continuity in Theorem \ref{disp Lipschitz} can imply the $\mathrm{W}_1$-Lipschitz continuity under displacement $\lambda$-monotonicity.
\begin{corollary}\label{disp W1 Lipschitz}
    Let Assumptions \ref{regularity assumption}, \ref{convexity assumption}, \ref{locally bounded assumption} hold. Suppose $G$ and $L$ satisfies the displacement $\lambda$-monotonicities (\ref{G displacement})(\ref{L displacement}). Suppose $T<\delta$, where $\delta$ is defined in Proposition \ref{local FBSDE}. Then for any initial conditions $(x_i,\xi_i)\in\R^d\times\mathbb{L}^2(\mathcal{F}_{t_0}),\,i=1,2$, the local solutions $\nabla Y^{\xi_i,x_i}$ of (\ref{FBSDE xi})-(\ref{FBSDE xi,x}) satisfy the following $\mathrm{W}_1$-Lipschitz estimate:
    \begin{equation}
        |\nabla Y_{t_0}^{\xi_1,x_1}-\nabla Y_{t_0}^{\xi_2,x_2}|\leq C\big(|x_1-x_2|+\E_{\mathcal{F}_{t_0}^0}[|\xi_1-\xi_2|]\big),
    \end{equation}
    where $C$ depends only on $d$ and the parameters in the assumptions.
\end{corollary}
\begin{proof}
    The local well-posedness of FBSDEs (\ref{FBSDE xi})-(\ref{FBSDE xi,x}) is guaranteed by Proposition \ref{local FBSDE}. Recalling $\widehat{H}$ in Remark \ref{remark hat H}, we mollify $(G,\widehat{H})$ to $(G_n,\widehat{H}_n)$ by the method in \cite{mou2024nonsmooth,cosso2023smooth}. Let $\nabla Y^{n,\xi_i,x_i}$ denote the solution to the corresponding FBSDE, and define $U_n(t_0,x_i,\mu_{t_0}^i):=\nabla Y_{t_0}^{n,\xi_i,x_i}$. By \cite[Theorem 6.1]{liu2026global}, $U_n$ is a classical solution to the vectorial master equation (\ref{vectorial master equation}) and $\partial_{\mu}U_n$ is uniformly bounded. By the convergence of mollifiers, this implies the desired $\mathrm{W}_1$-Lipschitz estimate.
\end{proof}
\begin{remark}
    In \cite{liu2026global}, the classical solution theory for the master equation was studied under non-degenerate idiosyncratic noise ($\sigma>0$), while we only assume $\sigma\geq 0$ in the present paper. However, this non-degeneracy only contributes to the estimates in the space variable there, which is covered by our Propositions \ref{relation u and v LL} and \ref{relation u and v disp lamda}.
\end{remark}
With the a priori $\mathrm{W}_1$-Lipschitz estimate, we can show the global well-posedness for the FBSDEs (\ref{FBSDE xi})-(\ref{FBSDE xi,x}), and then we can define the candidate solution to the master equation using the decoupling field.
\begin{theorem}\label{FBSDE well-posed}
    Let Assumptions \ref{regularity assumption}, \ref{quadratic growth assumption}, \ref{convexity assumption} hold. Assume further either of the following monotonicity conditions hold:\\
    (i) $G$ and $L$ satisfy the Lasry-Lions monotonicity conditions (\ref{G Lasry-Lions})(\ref{L Lasry-Lions}) and Assumption \ref{G L convexity assumption};\\
    (ii) $G$ and $L$ satisfy the displacement $\lambda$-monotonicity conditions (\ref{G displacement})(\ref{L displacement}) and Assumption \ref{locally bounded assumption}.\\
    Then, we have\\
    (i) The FBSDEs (\ref{FBSDE xi})-(\ref{FBSDE xi,x}) are well-posed. Consequently, for any $(t_0,x,\mu)\in[0,T]\times\R^d\times\mathcal{P}_2(\R^d)$ and $\xi\in\mathbb{L}^2(\mathcal{F}_{t_0}^1;\mu)$, we can define $U(t_0,x,\mu):=\nabla Y_{t_0}^{\xi,x}$;\\
    (ii) Given the $(X^{\xi,x},\nabla Y^{\xi,x})$ in (\ref{FBSDE xi,x}) and $\mu,\rho$ in (\ref{FBSDE xi}). For any $(t_0,x,\mu)\in[0,T]\times\R^d\times\mathcal{P}_2(\R^d)$ and $\xi\in\mathbb{L}^2(\mathcal{F}_{t_0}^1;\mu)$, we can define function $V$ by (\ref{define V}). Then we have $\partial_x V=U$;\\
    (iii) $V\in C_{Lip}^{0,1}(\Theta)$. In particular, $\partial_x V=U$ is $\frac12$-H\"{o}lder continuous in $t$, uniformly with respect to \((x,\mu)\) up to a linear growth factor. More precisely, for any $0\leq t_0\leq t\leq T$, $x\in\R^d$, and $\mu\in\mathcal{P}_2(\R^d)$,
   \begin{equation}\label{U holder continuous}
    \begin{aligned}
        |U(t_0,x,\mu)-U(t,x,\mu)|
        \leq
        C\big(1+|x|+M_1(\mu)\big)\sqrt{t-t_0}.
        \end{aligned}
    \end{equation}
\end{theorem}
\begin{proof}
    (i) We first gather the $\mathrm{W}_1$-Lipschitz estimates in measure. The $\mathrm{W}_1$-Lipschitz estimate of $\nabla Y^{\xi,x}$ under Lasry-Lions monotonicity condition is proved in Theorem \ref{LL Lipschitz}. Under displacement $\lambda$-monotonicity condition, we note that $\delta$ in Proposition \ref{local FBSDE} depends only on the $\mathrm{W}_2$-Lipschitz constant of data, which is shown to be uniform in Theorem \ref{disp Lipschitz}. Hence we can divide $[0,T]$ into finite time intervals and Corollary \ref{disp W1 Lipschitz} provides a uniform $\mathrm{W}_1$-Lipschitz estimate on $[0,T]$. 
    \par The construction of a global solution with local solutions follows from a backward induction, where the uniform $\mathrm{W}_1$-Lipschitz estimate for $\nabla Y^{\xi,x}$ is crucial. We refer to \cite[Theorem 4.4]{mou2024nonsmooth} for details. The uniqueness follows directly from Theorem \ref{LL Lipschitz} and \ref{disp Lipschitz}.\\
    (ii) By the quadratic growth condition for $G$ and $L$ in Assumption \ref{quadratic growth assumption}, and the $L^2$-estimates for FBSDEs (\ref{FBSDE xi})-(\ref{FBSDE xi,x}), we know
    \begin{equation}
        \E\Big[\Big|G(X_T^{\xi,x},\mu_T)+\int_t^T\big[H(X_s^{\xi,x},\nabla Y_s^{\xi,x},\rho_s)-\nabla Y_s^{\xi,x}\cdot\partial_p H(X_s^{\xi,x},\nabla Y_s^{\xi,x},\rho_s)\big]\,\mathrm{d}s\Big|\Big]<\infty,
    \end{equation}
    so the $V$ in (\ref{define V}) is well-defined. Propositions \ref{relation u and v LL} and \ref{relation u and v disp lamda} yield $\partial_x V = U$.\\
    (iii) We first show $U=\partial_x V\in C_{Lip}^0(\Theta)$. The Lipschitz continuity of $U:=\partial_x V$ in $(x,\mu)$ follows from Theorem \ref{LL Lipschitz} and Corollary \ref{disp W1 Lipschitz}, so next we investigate the continuity of $U$ in $t$.
    \par By the Lipschitz continuity of $U$, 
    \begin{equation*}
        \mathbb{E}[|\nabla Y_t^\xi|]=\mathbb{E}[|U(t,X_t^\xi,\mu_t)|]
        \leq
        C\mathbb{E}\big[1+|X_t^\xi|+W_1(\mu_t,\delta_0)\big]\leq C\mathbb{E}\big[1+|X_t^\xi|\big].
    \end{equation*}
    By  Assumption~\ref{regularity assumption} and the above inequality,
    \begin{equation}
    \label{eq:H-population}
    \begin{aligned}
        &\mathbb{E}\big[|\partial_xH(X_t^\xi,\nabla Y_t^\xi,\rho_t)|+|\partial_pH(X_t^\xi,\nabla Y_t^\xi,\rho_t)|\big]\\
        \leq&
        C\mathbb{E}\big[1+|X_t^\xi|+|\nabla Y_t^\xi|+\mathrm{W}_1(\rho_t,\delta_{(0,0)})\big]\\
        \leq &C\mathbb{E}\big[1+|X_t^\xi|\big].
    \end{aligned}
    \end{equation}
    Indeed, $$\mathbb{E}\big[\mathrm{W}_1(\rho_t,\delta_{(0,0)})\big]=\mathbb{E}\Big[\mathrm{W}_1\Big(\Phi\big(\mathcal{L}^0(X_t^\xi,\nabla Y_t^\xi)\big),\delta_{(0,0)}\Big)\Big]\leq C\big(1+\mathbb{E}[|X_t^\xi|+|\nabla Y_t^\xi|]\big).$$
    Combining \eqref{eq:H-population} with the forward equation in \eqref{FBSDE xi}, we obtain
    \[
        \sup_{t_0\leq t\leq T}\mathbb{E}[|X_t^\xi|]
        \leq
        M_1(\mu)+C\sqrt{T-t_0}
        +C\int_{t_0}^T\big(1+\mathbb{E}[|X_t^\xi|]\big)\,\mathrm{d}t.
    \]
    Gronwall's inequality yields
    \begin{equation*}
    \label{eq:population-moment}
        \sup_{t_0\leq t\leq T}
        \mathbb{E}\big[
        |X_t^\xi|+|\nabla Y_t^\xi|
        +M_1(\rho_t)
        \big]
        \leq
        C\big(1+M_1(\mu)\big).
    \end{equation*}
    Similarly, 
    \begin{equation*}
    \label{eq:decoupled-moment}
        \sup_{t_0\leq t\leq T}
        \mathbb{E}\big[
        |X_t^{\xi,x}|+|\nabla Y_t^{\xi,x}|
        \big]
        \leq
        C\big(1+|x|+M_1(\mu)\big),
    \end{equation*}
    which proves
    \begin{equation}
    \label{eq:global-moment}
    \begin{aligned}
        \sup_{t_0\leq t\leq T}\mathbb{E}\Big[
        &|X_t^\xi|+|\nabla Y_t^\xi|
        +|X_t^{\xi,x}|+|\nabla Y_t^{\xi,x}|
        +M_1(\rho_t)\Big]
        \leq C\big(1+|x|+M_1(\mu)\big),
    \end{aligned}
    \end{equation}
    Let $t_0\leq t\leq
    T$. By the backward equation \eqref{FBSDE xi,x},
    \begin{equation*}
    \begin{aligned}
        U(t_0,x,\mu)
        =
        \mathbb{E}\Big[
        U(t,X_t^{\xi,x},\mu_t)
        +\int_{t_0}^t
        \partial_xH(X_s^{\xi,x},\nabla Y_s^{\xi,x},\rho_s)
        \,\mathrm{d}s
        \Big].
    \end{aligned}
    \end{equation*}
    One can easily see that
    \begin{equation}\label{eq:decoupled-increment}
    \begin{aligned}
        \sup_{t_0\leq s\leq t}\mathbb{E}\big[|X_s^{\xi,x}-x|\big]
        &\leq
        C\big(1+|x|+M_1(\mu)\big)
        \sqrt{t-t_0},
        \\
        \sup_{t_0\leq s\leq t}\mathbb{E}\big[|X_s^{\xi}-\xi|\big]
        &\leq
        C\big(1+M_1(\mu)\big)
        \sqrt{t-t_0}.
    \end{aligned}
    \end{equation}
    Moreover, 
    \eqref{eq:global-moment} implies
    \begin{equation}
    \label{eq:driver-time-estimate}
    \begin{aligned}
        \mathbb{E}\big[\int_{t_0}^t
        |\partial_xH(X_s^{\xi,x},\nabla Y_s^{\xi,x},\rho_s)|
        \,\mathrm{d}s\big]
        \leq
        C\big(1+|x|+M_1(\mu)\big)(t-t_0).
    \end{aligned}
    \end{equation}
    Using \eqref{eq:decoupled-increment}-\eqref{eq:driver-time-estimate}, we conclude that
    \[
    \begin{aligned}
        |U(t_0,x,\mu)-U(t,x,\mu)|
        &\leq
        \mathbb{E}\big[|U(t,X_t^{\xi,x},\mu_t)-U(t,x,\mu) |+\int_{t_0}^t\vert \partial_xH(X_r^{\xi,x},\nabla Y_r^{\xi,x},\rho_r)\vert\big]\\
        \leq& C\sup_{t_0\leq s\leq t}\mathbb{E}[|X_s^{\xi,x}-x|+|X_s^{\xi}-\xi|]+
        C\big(1+|x|+M_1(\mu)\big)
        (t-t_0) \\
        &\leq
        C\big(1+|x|+M_1(\mu)\big)\sqrt{t-t_0}.
        \end{aligned}
    \]
    
    Finally, we need to show $V\in C^0(\Theta)$. Fix $t_0\in[0,T],x\in\R^d$, and let $\mu_n=\mathcal{L}(\xi_n)\subset\mathcal{P}_2(\R^d)$ such that $\E[|\xi_n-\xi|]\rightarrow 0,\,n\rightarrow\infty$. To pass the limit into the expectation, we need the uniform integrability of the quadratic cost $G$ and $L$, which requires more than second-order moment estimate. For this aim, we shall prove a fourth-order moment estimate. Denote $\rho_t^n:=\Phi\big(\mathcal{L}^0(X_t^{\xi_n,x},\nabla Y_t^{\xi_n,x})\big)$. With the growth condition, we can show by a similar argument that
    \begin{equation}\label{fourth moment estimate}
        \E\Big[\sup_{t_0\leq t\leq T}\big(|X_t^{\xi_n,x}|^4+(M_1(\rho_t^n))^4\big)\Big]\leq C\big(1+|x|^4+(M_1(\mu_n))^4\big).
    \end{equation}
    Since $\mathrm{W}_1(\mu_n,\mu)\rightarrow 0$, it is clear that the right hand side is uniformly bounded. This implies $|X_t^{\xi_n,x}|^2,(M_1(\rho_t))^2$ are uniformly integrable. Combining with the quadratic growth condition for $G$ and $L$, we see that the cost in $V(t_0,x,\mu_n)$ in definition (\ref{define V}) is uniformly integrable, so we can pass the limit by Vitali's theorem, and thus $V(t_0,x,\mu_n)\rightarrow V(t_0,x,\mu)$ as $n\rightarrow\infty$. A similar uniformly integrable argument shows that $V$ is continuous in $(t,x,\mu)$. Therefore, we conclude $V\in C_{Lip}^{0,1}(\Theta)$.
\end{proof}
\begin{remark}\label{remark: W1 continuous}
    The above arguments show why we require the quadratic growth to depend on the first moment of measure in Assumption \ref{quadratic growth assumption}. Otherwise, if the data are assumed to have quadratic growth in $M_2(\mu)$ and $M_2(\rho)$, the right hand side of (\ref{fourth moment estimate}) will change into $(M_2(\mu_n))^4$. However, since $\mathrm{W}_1(\mu_n,\mu)\rightarrow 0$ does not imply the uniform boundedness of $M_2(\mu_n)$, the uniform integrability will fail.
\end{remark}
\begin{remark}\label{remark compare with alpar}
    With further regularity assumptions in the measure variables, the arguments in this section can be also used to study the classical solutions. In \cite{jackson2026unconditional}, the authors also studied the classical solutions to the MFGC master equation under a similar displacement $\lambda$-monotonicity condition,
    \begin{equation}\label{similar displacement}
    \begin{aligned}
        &\E[\langle\partial_a L(X^1,\alpha^1,\mathcal{L}(X^1,\alpha^1))-\partial_a L(X^2,\alpha^2,\mathcal{L}(X^2,\alpha^2)),\alpha^1-\alpha^2\rangle\\
        &+\langle\partial_x L(X^1,\alpha^1,\mathcal{L}(X^1,\alpha^1))-\partial_x L(X^2,\alpha^2,\mathcal{L}(X^2,\alpha^2)),X^1-X^2\rangle]\\
        \geq& c_{\alpha}\E[|\alpha^1-\alpha^2|^2]-c_x\E[|X^1-X^2|^2],
    \end{aligned}
    \end{equation}
    where $c_{\alpha}>0$ and $c_x\geq 0$, together with an assumption on the monotonicity constants,
    \begin{equation}\label{similar coefficient}
        c_{\alpha}>\lambda T+\frac12 c_x T^2.
    \end{equation}
    \par Under Assumptions \ref{regularity assumption} and \ref{convexity assumption}, the strict inequality $c_1<c_0$ implies (\ref{similar displacement}) for some $c_{\alpha}>0$ and $c_x\geq 0$. However, these constants need not satisfy (\ref{similar coefficient}). For fixed constants, (\ref{similar coefficient}) restricts the time horizon unless $\lambda=c_x=0$, so the well-posedness result in \cite{jackson2026unconditional} is only available for small time duration $T$ with displacement semimonotone data ($\lambda>0$). Our additional displacement $\lambda$-monotonicity condition (\ref{L displacement}) retains the mixed term $2\lambda\E[\langle X^1-X^2,\alpha^1-\alpha^2\rangle]$, which ensures the propagation of monotonicity and allows our result to hold for arbitrary $T$.
    \par One could also use the canonical transformation in \cite{arnold1989mechanics,bansil2025hidden,bansil2025classical} and consider displacement monotone data $G_{\lambda}$ and $L_{\lambda}$. Nevertheless, ordinary displacement monotonicity of $L_{\lambda}$ does not imply (\ref{similar displacement}) with $c_x=0$ and $c_{\alpha}>0$. Although Assumption \ref{convexity assumption} gives strict control monotonicity when the state is fixed, it only provides this lower bounded for some $c_x\geq 0$ and $c_{\alpha}>0$ when both the state and control vary. Therefore, we still cover cases which do not satisfy the conditions in \cite{jackson2026unconditional} even in the regime $\lambda=0$.
    \par We also provide a quick example with $\lambda=0$. Let $G=0$ and
    \begin{equation}
        L(x,a,\mathcal{L}(X,\alpha)):=\frac12(x^2+a^2)-\frac12(x+a)\sin(\E[X]+\E[\alpha]).
    \end{equation}
    It can be verified that $G$ and $L$ satisfy Assumptions \ref{regularity assumption}, \ref{convexity assumption}, \ref{locally bounded assumption} and displacement monotonicity (\ref{G displacement})(\ref{L displacement}) with $\lambda=0$. However, $L$ cannot satisfy (\ref{similar displacement}) with $c_x=0$. Indeed, for deterministic pair $(X^1,\alpha^1)=(\epsilon,\epsilon)$ and $(X^2,\alpha^2)=(0,0)$, (\ref{similar displacement}) gives $2\epsilon^2-\epsilon\sin(2\epsilon)\geq c_{\alpha}\epsilon^2$, which cannot hold for $c_{\alpha}>0$ as $\epsilon\rightarrow 0$.
\end{remark}

\section{Weak solution}
\par The notion of weak solution follows from the idea in \cite{mou2024nonsmooth}. However, we shall begin with the vectorial MFGC system (\ref{vectorial MFGC system}) rather than the MFGC system (\ref{MFGC system}). This allows us to discuss data with quadratic growth. Given a function $V$, we shall use $U:=\partial_x V$ to decouple the vectorial MFGC system and consider the weak solutions to the SPDE and BSPDE separately. We first consider the SPDE
\begin{equation}\label{vectorial SPDE}
    \mathrm{d}\mu_t(x)=\Big[\frac{\widehat{\sigma}^2}{2}\mathrm{tr}(\partial_{xx}\mu_t(x))-\divergence\big(\mu_t(x)\partial_p H(x,U(t,x,\mu_t),\rho_t)\big)\Big]\,\mathrm{d}t-\sigma_0\partial_x\mu_t(x)\cdot \mathrm{d}B_t^0,
\end{equation}
and next we consider the weak solution to the BSPDE in (\ref{vectorial MFGC system}) for given $\mu$. We say $u=(u^1,\cdots,u^d)$ is a weak solution to the BSPDE if $u(T,x)=\partial_x G(x,\mu_T)$, and each $i=1,\cdots,d$ and any $\varphi\in C_c^{1,2}([0,T]\times\R^d)$,
\begin{equation}
\begin{aligned}
    &\mathrm{d}\int_{\R^d}u^i(t,x)\varphi(t,x)\,\mathrm{d}x\\
    =&\int_{\R^d}r^i(t,x)\varphi(t,x)\,\mathrm{d}x\, \mathrm{d}B_t^0+\int_{R^d}\Big[u^i(t,x)\partial_t\varphi(t,x)-\frac{\widehat{\sigma}^2}{2}u^i(t,x)\tr\big(\partial_{xx}\varphi(t,x)\big)\\
    &+\sigma_0 r^i(t,x)\cdot\partial_x\varphi(t,x)+H(x,u(t,x),\rho_t)\partial_{x_i}\varphi(t,x)\Big]\,\mathrm{d}x\,\mathrm{d}t.
\end{aligned}
\end{equation}
\begin{definition}
    (i) We say $U\in C_{Lip}^0(\Theta;\R^d)$ is a vectorial weak solution to the master equation (\ref{master equation}), if $U$ is curl-free in $x$, and for any $t_0\in[0,T]$ and $\mu_{t_0}\in\mathcal{P}_2(\R^d)$, the SPDE (\ref{vectorial SPDE}) has a weak solution $\mu$; moreover, for any such weak solution $\mu$, $u(t,x):=U(t,x,\mu_t)$ is a weak solution to the BSPDE in (\ref{vectorial MFGC system}).\\
    (ii) We say $V\in C_{Lip}^{0,1}(\Theta)$ is a weak solution to the master equation (\ref{master equation}), if $U:=\partial_x V$ is a vectorial weak solution; moreover, for any $t_0\in[0,T]$ and $x\in\R^d,\mu_{t_0}\in\mathcal{P}_2(\R^d)$,
    \begin{equation}\label{weak solution V}
    \begin{aligned}
        V(t_0,x,\mu_{t_0})=&\E_{\mathcal{F}_{t_0}^0}\Big[G(X_T^U,\mu_T)+\int_t^T\big[H(X_s^U,U(s,X_s^U,\mu_s),\rho_s)\\
        &\quad\qquad-U(s,X_s^U,\mu_s)\cdot\partial_p H(X_s^U,U(s,X_s^U,\mu_s),\rho_s)\big]\,\mathrm{d}s\Big],
    \end{aligned}
    \end{equation}
    where $\mu$ is the weak solution to the SPDE (\ref{vectorial SPDE}) with the given $t_0,\mu_{t_0}$, and $X^U$ is the unique solution to the SDE
    \begin{equation}\left\{
    \begin{aligned}
        &X_t^U=x+\int_{t_0}^t\partial_p H(X_s^U,U(s,X_s^U,\mu_s),\rho_s)\,\mathrm{d}s+\sigma B_t^{t_0}+\sigma_0 B_t^{0,t_0}\\
        &\rho_t=\big(id,\partial_pH(\cdot,U(t,\cdot,\mu_t),\rho_t)\big)\#\mu_t.
    \end{aligned}\right.
    \end{equation}
\end{definition}
\begin{theorem}\label{weak solution theorem}
    Let the assumptions in Theorem \ref{FBSDE well-posed} hold. Then the function $U$ defined in Theorem \ref{FBSDE well-posed} is a vectorial weak solution to the master equation (\ref{master equation}). Moreover, the function $V$ defined in Theorem \ref{FBSDE well-posed} is the unique weak solution to the master equation (\ref{master equation}).
\end{theorem}
\begin{proof}
    \textit{Existence.} We first show that $U$ is a vectorial weak solution when $T\leq\delta$ for the $\delta$ in Theorem \ref{local FBSDE}.
    \par For any $\xi$ satisfying $\mathcal{L}(\xi)=\mu_{t_0}$, the FBSDE (\ref{FBSDE xi}) is locally well-posed. Then we can see that $\mu_t:=\mathcal{L}^0(X_t^{\xi})$ is a weak solution to the SPDE (\ref{vectorial SPDE}).
    \par Recall the notation $\widehat{H}(\cdot,\cdot,\rho):=H(\cdot,\cdot,\Phi(\rho))$ in Remark \ref{remark hat H}. Let $G_n,\widehat{H}_n$ be the smooth mollifiers of $G,\widehat{H}$ constructed in \cite{mou2024nonsmooth,cosso2023smooth}. If we replace the data with $G_n,\widehat{H}_n$, the vectorial master equation (\ref{vectorial master equation}) admits a unique classical solution $U_n$ with bounded $\partial_{\mu}U_n$ by \cite{liu2026global}. By the stability of the FBSDE (\ref{FBSDE xi,x}), we have $\lim_{n\rightarrow\infty}(U_n-U)(t,x,\mu)=0$ for any $(t,x,\mu)$.
   
    \par For any weak solution $\mu$ to the SPDE (\ref{vectorial SPDE}), we have $\mu_t=\mathcal{L}^0(X_t)$, where $X$ is determined by
    \begin{equation}\label{SDE with law mu}
        X_t=\xi+\int_{t_0}^t\partial_p \widehat{H}(X_s,U(s,X_s,\mu_s),\nu_s)\,\mathrm{d}s+\sigma B_t^{t_0}+\sigma_0 B_t^{0,t_0},
    \end{equation}
    where $\mathcal{L}(\xi)=\mu_{t_0}$ and $\nu_t:=\mathcal{L}^0(X_t,U(t,X_t,\mu_t))$. Define
    \[
    u(t,x):=U(t,x,\mu_t),\, u_n(t,x):=U_n(t,x,\mu_t),\, r_n^i(t,x):=\sigma_0\tilde{\E}_{\mathcal{F}_t^0}[\partial_{\mu}U_n^i(t,x,\mu_t,\tilde{X}_t)].
    \]
    Applying It\^{o}'s formula on $u_n^i(t,x)=U_n^i(t,x,\mathcal{L}^0(X_t))$, and using the fact that $U_n^i$ satisfies the vectorial master equation (\ref{vectorial master equation}), we have
    \begin{equation}\label{weak solution Ito}
    \begin{aligned}
        \mathrm{d}u_n^i(t,x)=&r_n^i(t,x)\, \mathrm{d}B_t^0-\Big[\tr(\frac{\widehat{\sigma}^2}{2}\partial_{xx}u_n^i+\sigma_0\partial_x r_n^i)+\partial_{x_i}\widehat{H}_n(x,u_n,\nu_t^n)\\
        &+\partial_p \widehat{H}_n(x,u_n,\nu_t^n)\cdot\partial_{x_i}u_n(t,x)\Big]\,\mathrm{d}t+I_n^i(t,x)\,\mathrm{d}t,
    \end{aligned}
    \end{equation}
    where
    \[
    I_n^i(t,x):=\tilde{\E}_{\mathcal{F}_t^0}\big[\partial_{\mu}U_n^i(t,x,\mu_t,\tilde{X}_t)\cdot\big(\partial_p \widehat{H}(\tilde{X}_t,u(t,\tilde{X}_t),\nu_t)-\partial_p \widehat{H}_n(\tilde{X}_t,u_n(t,\tilde{X}_t),\nu_t^n)\big)\big],
    \]
    and $\nu_t:=\mathcal{L}^0(X_t,u(t,X_t)),\,\nu_t^n:=\mathcal{L}^0(X_t,u_n(t,X_t))$.
    For any test function $\varphi\in C_c^{\infty}([0,T]\times\R^d)$, we obtain from (\ref{weak solution Ito})
    \[
    \begin{aligned}
        &\int_{\R^d}u_n^i(t,x)\varphi(t,x)\,\mathrm{d}x\\
        =&\int_{\R^d}\partial_{x_i}G_n(x,\mu_T)\varphi(T,x)\,\mathrm{d}s+\int_t^T\int_{\R^d}\Big[\partial_t u_n^i+\tr(\frac{\widehat{\sigma}^2}{2}\partial_{xx}u_n^i+\sigma_0\partial_x r_n^i)\\
        &+\partial_{x_i}\widehat{H}_n(x,u_n,\nu_t^n)+\partial_p \widehat{H}_n(x,u_n,\nu_t^n)\cdot\partial_{x_i}u_n(t,x)-I_n^i(s,x)\Big]\varphi(s,x)\,\mathrm{d}x\,\mathrm{d}s\\
        &-\int_t^T\int_{\R^d}r_n^i\varphi(s,x)\,\mathrm{d}s\,\mathrm{d}B_s^0.
    \end{aligned}
    \]
    Denote $\nabla Y_t^{n,i}:=\int_{\R^d}u_n^i(t,x)\varphi(t,x)\,\mathrm{d}x,\,\nabla Z_t^{n,0,i}:=\int_{\R^d}r_n^i(t,x)\varphi(t,x)\,\mathrm{d}x$, and
    \[
    \begin{aligned}
    \Phi_t^{n,i}:=&\int_{\R^d}\Big[u_n^i\partial_t\varphi(s,x)-\frac{\widehat{\sigma}^2}{2}u_n^i\tr(\partial_{xx}\varphi(s,x))+\sigma_0 r_n^i\partial_x\varphi(s,x)\\
    &\qquad+\widehat{H}_n(x,u_n,\nu_s^n)\cdot\partial_{x_i}\varphi(s,x)\Big]\,\mathrm{d}x.
    \end{aligned}
    \]
    Moving the derivatives to the test function $\varphi$, we conclude from the above equality
    \begin{equation}\label{weak solution BSDE}
    \begin{aligned}
        \nabla Y_t^{n,i}=&\int_{\R^d}\partial_{x_i}G_n(x,\mu_T)\varphi(T,x)\,\mathrm{d}x-\int_t^T\Phi_s^{n,i}\,\mathrm{d}s\\
        &-\int_t^T\int_{\R^d}I_n^i(s,x)\varphi(s,x)\,\mathrm{d}x\,\mathrm{d}s-\int_t^T \nabla Z_s^{n,0,i}\,\mathrm{d}B_s^0.
    \end{aligned}
    \end{equation}
    \par Since $r_n^i$ is uniformly bounded, we can extract a weak convergent subsequence $r_{n_k}^i\rightharpoonup r^i$ in $\mathbb{L}^2(\Omega\times[t_0,T]\times B_R)$, where $B_R:=\{x\in\R^d:|x|\leq R\}$. This implies the weak convergence
    \begin{equation}
        \int_{\R^d}r_{n_k}^i(s,x)\partial_x\varphi(s,x)\,\mathrm{d}x\rightharpoonup\int_{\R^d}r^i(s,x)\partial_x\varphi(s,x)\,\mathrm{d}x\quad\text{in}\ \mathbb{L}^2(\Omega\times[t_0,T]).
    \end{equation}
    Combing with the strong convergence of other terms in $\Phi_t^{n_k,i}$, we obtain
    \begin{equation}
    \begin{aligned}
        \Phi_t^{n_k,i}\rightharpoonup\Phi_t^i:=&\int_{\R^d}\Big[u^i\partial_t\varphi(s,x)-\frac{\widehat{\sigma}^2}{2}u^i\tr(\partial_{xx}\varphi(s,x))+\sigma_0 r^i\partial_x\varphi(s,x)\\
        &\qquad+\widehat{H}(x,u,\nu_s)\cdot\partial_{x_i}\varphi(s,x)\Big]\,\mathrm{d}x\quad\text{in}\ \mathbb{L}^2(\Omega\times[t_0,T]).
    \end{aligned}
    \end{equation}
    It follows similarly that
    \begin{equation}
        \int_{\R^d}r_{n_k}^i(s,x)\varphi(s,x)\,\mathrm{d}x\rightharpoonup\int_{\R^d}r^i(s,x)\varphi(s,x)\,\mathrm{d}x\quad\text{in}\ \mathbb{L}^2(\Omega\times[t_0,T]),
    \end{equation}
    which implies by It\^{o}'s isometry that
    \begin{equation}
        \int_t^T \nabla Z_s^{0,i}\,\mathrm{d}B_s^0\rightharpoonup\int_t^T \nabla Z_s^{n,0,i}\,\mathrm{d}B_s^0\quad\text{in}\ \mathbb{L}^2(\Omega).
    \end{equation}
    \par On the other hand, by the boundedness of $\partial_{\mu}U_n^i$ and convergence of $\partial_p H_n, u_n,\nu^n$, we have
    \begin{equation}
    \begin{aligned}
        &\E\int_{t_0}^T\Big|\int_{\R^d}I_n^i(t,x)\varphi(t,x)\,\mathrm{d}x\Big|^2\,\mathrm{d}t\\
        \leq&C\E\int_{t_0}^T\E_{\mathcal{F}_t^0}\big[\big|\partial_p \widehat{H}(\tilde{X}_t,u(t,\tilde{X}_t),\nu_t)-\partial_p \widehat{H}_n(\tilde{X}_t,u_n(t,\tilde{X}_t),\nu_t^n)\big|\big]^2\,\mathrm{d}t\\
        \leq& C\E\int_{t_0}^T\big|\partial_p \widehat{H}(\tilde{X}_t,u(t,\tilde{X}_t),\nu_t)-\partial_p \widehat{H}_n(\tilde{X}_t,u_n(t,\tilde{X}_t),\nu_t^n)\big|^2\,\mathrm{d}t\rightarrow 0,
    \end{aligned}
    \end{equation}
    so we conclude
    \begin{equation}
        \int_t^T\int_{\R^d}I_n^i(s,x)\varphi(s,x)\,\mathrm{d}x\,\mathrm{d}s\rightarrow 0\quad\text{in}\ \mathbb{L}^2(\Omega).
    \end{equation}
    Therefore, taking $k\rightarrow\infty$ in the BSDE (\ref{weak solution BSDE}) for $\nabla Y^{n_k,i}$, we obtain almost surely
    \begin{equation}\label{vectorial weak solution}
    \begin{aligned}
        &\int_{\R^d}u^i(t,x)\varphi(t,x)\,\mathrm{d}x\\
        =&\int_{\R^d}\partial_{x_i}G(x,\mu_T)\varphi(T,x)\,\mathrm{d}x-\int_t^T\int_{\R^d}\Big[u^i\partial_t\varphi(s,x)-\frac{\widehat{\sigma}^2}{2}u^i\tr(\partial_{xx}\varphi(s,x))\\
        &\qquad+\sigma_0 r^i\partial_x\varphi(s,x)+\widehat{H}(x,u,\nu_s)\cdot\partial_{x_i}\varphi(s,x)\Big]\,\mathrm{d}x\,\mathrm{d}s-\int_t^T\int_{\R^d}r^i\varphi(s,x)\,\mathrm{d}x\,\mathrm{d}B_s^0,
    \end{aligned}
    \end{equation}
    which shows $u^i$ is a weak solution for the BSPDE. Since $i$ is arbitrary, $U$ is a vectorial weak solution to the master equation. Finally, we conclude that the function $V$ constructed in Theorem \ref{FBSDE well-posed} is a weak solution to the master equation.
    \par When $T$ is arbitrary, we divide $[t_0,T]$ into $t_0<t_1<\cdots<t_n=T$ such that $t_i-t_{i-1}<\delta,\,i=1,\cdots,n$. Then we can repeat the previous construction forward from $[t_0,t_1]$ to obtain the weak solution $\mu$ for the SPDE. Define $u(t,x):=U(t,x,\mu_t)$. Similarly, we can repeat the argument backward from $[t_{n-1},t_n]$ to conclude $u$ is a weak solution to the BSPDE. Therefore, $U$ is a vectorial weak solution, and thus $V$ is a weak solution to the master equation on $[t_0,T]$.
    \par\textit{Uniqueness.} We still first consider the case $T\leq\delta$. Suppose $V$ is a weak solution to the master equation and define $U:=\partial_x V$. Let $\mu,X$ be given as (\ref{SDE with law mu}), then there exists a weak solution $(u,r)$ to the BSPDE in (\ref{vectorial MFGC system}) with $u(t,x):=U(t,x,\mu_t)$. Hence (\ref{vectorial weak solution}) holds for any $\varphi\in C_c^{\infty}(\R^d)$. 
    \par Let $\psi\in C_c^{\infty}(\R^d)$ be the smooth mollifier, that is $\psi\geq 0$ and $\int_{\R^d}\psi(x)\,\mathrm{d}x=1$. For fixed $x\in\R^d$, setting the test functions to be $\varphi_n^x(y)=n^d\psi\big(n(x-y)\big)$ in (\ref{vectorial weak solution}), we have
    \begin{equation}
        \mathrm{d} u_n^i(t,x)=\Phi_n^i(t,x)\,\mathrm{d}t+r_n^i(t,x)\,\mathrm{d}B_t^0,
    \end{equation}
    where
    \[
    \begin{gathered}
        u_n^i(t,x):=\int_{\R^d}u^i(t,y)n^d\psi(n(x-y))\,\mathrm{d}y,\quad r_n^i(t,x):=\int_{\R^d}r^i(t,y)n^d\psi(n(x-y))\,\mathrm{d}y,\\
        \Phi_n^i(t,x):=\int_{\R^d}\Big[-\frac{\widehat{\sigma}^2}{2}u^i(t,y)n^{d+2}\tr(\partial_{xx}\psi(n(x-y))\\
        -\sigma_0 r^i(t,y)\cdot n^{d+1}\partial_x\psi(n(x-y))-H(y,u(t,y),\rho_t)n^{d+1}\partial_{x_i}\psi(n(x-y))\Big]\,\mathrm{d}y.
    \end{gathered}
    \]
    Define $\nabla Y_t^{n,i}:=u_n^i(t,X_t)$. Then by It\^{o}-Wentzell's formula, we can compute that
    \[
    \begin{aligned}
        \mathrm{d}\nabla Y_t^{n,i}=&\Big[\Phi_n^i+\frac{\widehat{\sigma}^2}{2}\tr(\partial_{xx}u_n^i)+\sigma_0\tr(\partial_x r_n^i)+\partial_x u_n^i\cdot\partial_p H(X_t,u(t,X_t),\rho_t)\Big]\,\mathrm{d}t\\
        &+\sigma\partial_x u_n^i\, \mathrm{d}B_t+(\sigma_0\partial_x u_n^i+r_n^i)\,\mathrm{d}B_t^0\\
        =&-\Big[\int_{\R^d}H(y,u(t,y),\rho_t)n^{d+1}\partial_{x_i}\psi(n(X_t-y))\,\mathrm{d}y\\
        &\qquad-\partial_x u_n^i\cdot\partial_p H(X_t,u(t,X_t),\rho_t)\Big]\,\mathrm{d}t+\sigma\partial_x u_n^i\, \mathrm{d}B_t+(\sigma_0\partial_x u_n^i+r_n^i)\,\mathrm{d}B_t^0\\
        :=&-\tilde{\Phi}_n^i(t,X_t)\,\mathrm{d}t+\nabla Z_t^{n,i}\,\mathrm{d}B_t+\nabla Z_t^{n,0,i}\,\mathrm{d}B_t^0.
    \end{aligned}
    \]
    Next we compute the limit of $\tilde{\Phi}_n^i(t,x)$. Since $u(t,\cdot)$ is uniformly Lipschitz continuous, it has weak derivative $D_x u(t,\cdot)\in L^{\infty}(\mathbb{\R}^d;\R^{d\times d})$. Now we can apply Sobolev chain rule to $H(\cdot,u(t,\cdot),\rho_t)\in W^{1,\infty}_{\text{loc}}(\R^d)$ to compute that
    \begin{equation}
        \frac{\partial}{\partial x_i}H(x,u(t,x),\rho_t)=\partial_{x_i}H(x,u(t,x),\rho_t)+\partial_p H(x,u(t,x),\rho_t)\cdot D_{x_i}u.
    \end{equation}
    Computing the convolution with $\varphi_n^x$, we have
    \begin{equation}\label{Phi n1}
    \begin{aligned}
        \partial_{x_i}(H\ast\varphi_n^x)(x)=(\partial_{x_i}H\ast\varphi_n^x)(x)+(\partial_p H\cdot D_{x_i}u)\ast\varphi_n^x(x).
    \end{aligned}
    \end{equation}
    On the other hand, $u(t,x)=U(t,x,\mu_t)$ is curl-free in $x$, so
    \begin{equation}\label{Phi n2}
        \partial_x u_n^i(t,x)=(D_{x_i}u\ast\varphi_n^x)(x).
    \end{equation}
    Combing (\ref{Phi n1}) and (\ref{Phi n2}), we have
    \begin{equation}\label{Phi n formula}
        \tilde{\Phi}_n^i(t,x)=(\partial_{x_i}H\ast\varphi_n^x)(x)+\Big[\big((\partial_p H\cdot D_{x_i}u)\ast\varphi_n^x\big)(x)-(D_{x_i}u\ast\varphi_n^x)(x)\cdot\partial_p H(x,u(t,x),\rho_t)\Big].
    \end{equation}
    For the second term on the right hand side, we have
    \begin{equation}
    \begin{aligned}
        &\Big|\big((\partial_p H\cdot D_{x_i}u)\ast\varphi_n^x\big)(x)-(D_{x_i}u\ast\varphi_n^x)(x)\cdot\partial_p H(x,u(t,x),\rho_t)\Big|\\
        =&\Big|\int_{\R^d}\big(\partial_p H(y,u(t,y),\rho_t)-\partial_p H(x,u(t,x),\rho_t)\big)\cdot D_{x_i}u(t,y)n^d\psi(n(x-y))\,\mathrm{d}y\Big|\\
        \leq&\|D_{x_i}u\|_{L^{\infty}}\int_{\R^d}\big|\partial_p H(y,u(t,y),\rho_t)-\partial_p H(x,u(t,x),\rho_t)\big|n^d\psi(n(x-y))\,\mathrm{d}y,
    \end{aligned}
    \end{equation}
    which tends to $0$ as $n\rightarrow\infty$. Then (\ref{Phi n formula}) implies $\tilde{\Phi}_n^i(t,x)\rightarrow\partial_{x_i}H(x,u(t,x),\rho_t)$ as $n\rightarrow\infty$.
    \par Therefore, by standard BSDE arguments we can see that $\nabla Y_t:=u(t,X_t)$ serves as the unique solution to the BSDE
    \begin{equation}\label{uniqueness of the weak solution: BSDE}
        \nabla Y_t=\partial_x G(X_T,\mu_T)+\int_t^T\partial_x H(X_s,\nabla Y_s,\rho_s)\,\mathrm{d}s-\int_t^T\nabla Z_s \,\mathrm{d}B_s-\int_t^T\nabla Z_s^0 \,\mathrm{d}B_s^0,
    \end{equation}
    which implies the uniqueness of $\nabla Y_{t_0}:=u(t_0,X_{t_0})=\partial_x V(t_0,\xi,\mu_{t_0})$ and thus the uniqueness of $\partial_x V(t_0,\cdot,\cdot)$. Similarly, $\partial_x V(t,\cdot,\cdot)$ is also unique for $t\in[t_0,T]$. Finally, from the definition of weak solution, we see that if $U=\partial_x V$ is unique, then $V(t_0,x,\mu_{t_0})$ is unique. For arbitrary $T$, we similarly divide the time interval and repeat the above arguments backward, then we can conclude the uniqueness of weak solution.
\end{proof}
From the above argument, we see that a weak solution to the master equation is related to the decoupling field of a particular FBSDE. Using this fact, we can obtain the stability result for the vectorial weak solution. Indeed, due to the lack of Lipschitz continuity of $G,L$, we cannot directly study the stability of the weak solution.
\begin{corollary}\label{weak solution stability}
    For $i=1,2$, let $(G^i,L^i)$ satisfy Assumption \ref{regularity assumption}, and let $H^i$ be the corresponding Hamiltonian. Suppose $V^i\in C_{Lip}^{0,1}(\Theta)$ be a weak solution to the master equation (\ref{master equation}) with data $(G^i,L^i)$, corresponding Hamiltonian $H^i$ and the same initial condition $\mu_{t_0}\in\mathcal{P}_2(\R^d)$. Then the weak solution to the corresponding SPDE is given by $\mu_t^i:=\mathcal{L}^0(X_t^i)$, where $X^i$ solves the SDE  \begin{equation}\label{SDE stability}
        X_t^i=\xi+\int_{t_0}^t\partial_p H^i(X_s^i,\partial_x V^i(s,X_s^i,\mu_s^i),\rho_s^i)\,\mathrm{d}s+\sigma B_t^{t_0}+\sigma_0 B_t^{0,t_0},
    \end{equation}
    and $\rho_t^i=\Phi^i(X_t^i,\partial_x V^i(t,X_t^i,\mu_t^i))$, $\Phi^i$ is the fixed-point map associated to $H^i$.
    \par We have the following stability result: for any $\xi\in\mathbb{L}^2(\mathcal{F}_{t_0})$ with $\mathcal{L}(\xi)=\mu_{t_0}$,
    \begin{equation}
    \begin{aligned}
        &|\partial_x V^1(t_0,\xi,\mu_{t_0})-\partial_x V^2(t_0,\xi,\mu_{t_0})|\\
        \leq &C\mathbb{E}_{\mathcal{F}_{t_0}^0}\Big[\sup_{x\in\R^d}|(\partial_x G^1-\partial_x G^2)(x,\mu_T^2)|^2\\
        &\qquad+\int_{t_0}^T\sup_{(x,p)\in\R^d\times\R^d}|(\partial_x H^1-\partial_x H^2,\partial_p H^1-\partial_p H^2)(x,p,\rho_t^2)|^2\,\mathrm{d}t\Big]^{\frac12}.
    \end{aligned}
    \end{equation}
\end{corollary}
\begin{proof}
    From the uniform Lipschitz continuity of $\partial_x V^i$, it follows that the SDE (\ref{SDE stability}) is well-posed, and by Proposition \ref{fixed-point stability},
    \begin{equation}\label{stability for rho}
    \begin{aligned}
        \mathrm{W}_1(\rho_t^1,\rho_t^2)\leq& C\E_{\mathcal{F}_t^0}\big[|X_t^1-X_t^2|+|\partial_x V^1(t,X_t^1,\mu_t^1)-\partial_x V^2(t,X_t^2,\mu_t^2)|\\
        &+|(\partial_p H^1-\partial_p H^2)(X_t^2,\partial_x V^2(t,X_t^2,\mu_t^2),\rho_t^2)|\big],
    \end{aligned}
    \end{equation}
    which implies
    \begin{equation}\label{stability for X}
    \begin{aligned}
        \E_{\mathcal{F}_t^0}[|X_t^1-X_t^2|]\leq& C\int_{t_0}^t\E_{\mathcal{F}_s^0}\big[|(\partial_p H^1-\partial_p H^2)(X_s^2,\partial_x V^2(s,X_s^2,\mu_s^2),\rho_s^2)|\\
        &\qquad+|\partial_x V^1(s,X_s^1,\mu_s^1)-\partial_x V^2(s,X_s^2,\mu_s^2)|\big]\,\mathrm{d}s.
    \end{aligned}
    \end{equation}
    By the uniqueness argument in Theorem \ref{weak solution theorem}, we have $\partial_x V^i(t,X_t^i,\mu_t)=\nabla Y_t^i$, where $\nabla Y^i$ is the unique solution to the BSDE
    \begin{equation}
        \nabla Y_t^i=\partial_x G^i(X_T^i,\mu_T^i)+\int_t^T\partial_x H^i(X_s^i,\nabla Y_s^i,\rho_s^i)\,\mathrm{d}s-\int_t^T\nabla Z_s^i\,\mathrm{d}B_s-\int_t^T\nabla Z_s^{0,i}\,\mathrm{d}B_s^0.
    \end{equation}
    From standard BSDE theory, combining with (\ref{stability for rho}) and (\ref{stability for X}) we have the stability estimate
    \begin{equation}
    \begin{aligned}
        |\nabla Y_{t_0}^1-\nabla Y_{t_0}^2|^2\leq& C\mathbb{E}_{\mathcal{F}_{t_0}^0}\Big[\sup_{x\in\R^d}|(\partial_x G^1-\partial_x G^2)(x,\mu_T^2)|^2\\
        &+\int_{t_0}^T\sup_{(x,p)\in\R^d\times\R^d}|(\partial_x H^1-\partial_x H^2,\partial_p H^1-\partial_p H^2)(x,p,\rho_t^2)|^2\,\mathrm{d}t\Big].
    \end{aligned}
    \end{equation}
\end{proof}

\section{Comparison with other non-smooth solutions}
\subsection{Relations with other weak solutions}
In \cite{mou2024nonsmooth}, the authors provided three notions of non-smooth solution to master equation for classical MFG, denoted as weak solution, good solution and weak viscosity solution. In this subsection, we shall compare our weak solution for MFGC with these three notions respectively.
\par We first examine the relations between our weak solution and the weak solution introduced in \cite[Section 6]{mou2024nonsmooth}. Note that these two notions of weak solution are both defined in the spirit of integration by parts. However, in Theorem \ref{weak solution theorem} we only require the Lipschitz continuity of $\partial_x G,\partial_x L,\partial_a L$ and the Lipschitz continuity of $G,L$ themselves required in \cite{mou2024nonsmooth} is absent. In fact, if we complete these assumptions, \cite[Definition 6.1]{mou2024nonsmooth} can be adapted to the MFGC formulation and the function $V$ in Theorem \ref{FBSDE well-posed} is exactly the unique weak solution under either definition.
\begin{proposition}\label{relation weak solution}
    Let the assumptions in Theorem \ref{FBSDE well-posed} hold. Assume further $G$ is uniformly Lipschitz continuous in $(x,\mu)$, $L$ is uniformly Lipschitz continuous in $(x,\rho)$, and $\sup_{x,\rho}|\partial_a L(x,0,\rho)|<\infty$. Then the function $V$ given in Theorem \ref{FBSDE well-posed} can be used to decouple the MFGC system (\ref{MFGC system}) in the following sense: For any $0\leq t_0\leq T$ and $ \mu_{t_0}\in\mathcal{P}_2(\R^d)$, the SPDE
    \begin{equation}\label{relation SPDE}
        \mathrm{d}\mu_t(x)=\Big[\frac{\widehat{\sigma}^2}{2}\mathrm{tr}(\partial_{xx}\mu_t(x))-\divergence\big(\mu_t(x)\partial_p H(x,\partial_x V(t,x,\mu_t),\rho_t)\big)\Big]\,\mathrm{d}t-\sigma_0\partial_x\mu_t(x)\cdot \mathrm{d}B_t^0
    \end{equation}
    has a weak solution $\mu$. For the given $\mu$, the BSPDE in (\ref{MFGC system}) has a weak solution $(v,w)$ satisfying $V(t,x,\mu_t)=v(t,x)$.
\end{proposition}
The additional assumptions allow us to show $v$ is a weak solution to the BSPDE by the similar arguments in Theorem \ref{weak solution theorem}, so we omit the proof. We emphasize the difference is that our $v$ is defined by the feedback $U$ so we do not consider a direct scalar FBSDE formulation in \cite{mou2024nonsmooth} ($\sigma=1$) like
\begin{equation}
    \left\{\begin{aligned}
        &X_t^{\xi}=\xi+\int_{t_0}^t\partial_p H(X_s^{\xi},Z_s^{\xi},\rho_s)\,\mathrm{d}s+B_t^{t_0}+\sigma_0 B_t^{0,t_0},\\
        &Y_t^{\xi}=G(X_T^{\xi},\mu_T)+\int_t^T\Big[H(X_s^{\xi},Z_s^{\xi},\rho_s)-Z_s^{\xi}\cdot\partial_p H(X_s^{\xi},Z_s^{\xi},\rho_s)\Big]\,\mathrm{d}s\\
        &\qquad\quad-\int_t^T Z_s^{\xi}\cdot \mathrm{d}B_s-\int_t^T Z_s^0\cdot \mathrm{d}B_s^0,\\
        &\rho_t=\mathcal{L}^0\big(X_t^{\xi},\partial_p H(X_t^{\xi},Z_t^{\xi},\rho_t)\big),\quad\mu_t=\mathcal{L}^0(X_t^{\xi}).
    \end{aligned}\right.
    \end{equation}
In fact, the above system is not available when $\sigma=0$. Therefore, our method does not require non-degeneracy of the idiosyncratic noise, which is not covered by \cite{mou2024nonsmooth}.
\par Next, we shall compare our weak solution to the good solution and the weak viscosity solution defined in \cite[Section 5,7]{mou2024nonsmooth}. The notion of good solution is based on the stability argument. To be detailed, we mollify the data $(G,\widehat{H})$ in the master equation to $(G_n,\widehat{H}_n)$ and determine a solution $V_n$ to the master equation with the mollified data $(G_n,\widehat{H}_n)$. By the stability results, $V_n$ converges to some function $V$, which we define to be the good solution.
\par The way to define the weak viscosity solution is similar to that for the weak solution. Since there is no comparison principle for the master equation, standard viscosity solution theory fails here. However, it is feasible to consider the viscosity solution for the BSPDE in the MFGC system (\ref{MFGC system}). Similar to Proposition \ref{relation weak solution}, if $V$ can be used to decouple the MFGC system such that the SPDE (\ref{relation SPDE}) has a weak solution $\mu$, and $v(t,x):=V(t,x,\mu_t)$ serves as a viscosity solution to the BSPDE, then we call $V$ a weak viscosity solution to the master equation.
\par To give the strict definitions for these notions of non-smooth solution, a series of preparations are necessary. However, the definitions under MFG setting and MFGC setting are essentially the same, so we omit here and refer to \cite{mou2024nonsmooth}. Furthermore, once we include the additional assumptions for $G$ and $L$ as in Theorem \ref{relation weak solution}, we can similarly prove the equivalence among the notions of weak solution, good solution and weak viscosity solution as in \cite{mou2024nonsmooth}. Therefore, the function $V$ in Theorem \ref{weak solution theorem} is also the unique good solution and the unique weak viscosity solution to the MFGC master equation.

\subsection{Relations with Lipschitz solutions}
In \cite{bertucci2024lipschitz}, the authors introduced an alternative notion of non-smooth solution to the master equation, termed the Lipschitz solution. The idea is to fix the nonlinear term in the master equation, and then solve the linear transport equation explicitly. If a Lipschitz function serves as a fixed point in this procedure, it is called a Lipschitz solution. We show in this subsection that this definition is consistent with the weak solution. \par We consider the following linear transport equation, which is associated to the vectorial master equation (\ref{vectorial master equation}).
\begin{equation}\label{linear transport}
\begin{gathered}
    \partial_t U(t,x,\mu)+\frac{\widehat{\sigma}^2}{2}\tr(\partial_{xx}U)+A(t,x,\mu)+B(t,x,\mu)\cdot\partial_x U\\
    +\tr\Big(\bar{\tilde{\E}}\big[\frac{\widehat{\sigma}^2}{2}\partial_{\tilde{x}\mu}U(t,x,\mu,\tilde{\xi})+\sigma_0^2\partial_{x\mu}U(t,x,\mu,\tilde{\xi})+\frac{\sigma_0^2}{2}\partial_{\mu\mu}U(t,x,\mu,\bar{\xi},\tilde{\xi})\\
    +\partial_{\mu}U(t,x,\mu,\tilde{\xi})\cdot B(t,\tilde{\xi},\mu)\big]\Big)=0,\qquad U(T,x,\mu)=U_T(x,\mu),
\end{gathered}
\end{equation}
where $A, B, U_T$ are given functions. When these given data are Lipschitz continuous, the solution to the above linear transport equation can be given by the following Feynman-Kac representation formula,
\begin{equation}
    U(t,x,\mu)=\E\Big[U_T(X_T,\mu_T)+\int_t^T A(s,X_s,\mu_s)\,\mathrm{d}s\Big],
\end{equation}
where $X_s$ is the unique solution to the following SDE on $[t,T]$,
\begin{equation}
    \mathrm{d}X_s=B(s,X_s,\mu_s)\,\mathrm{d}s+\sigma \,\mathrm{d}B_s+\sigma_0\,\mathrm{d}B_s^0,\quad X_t=x,
\end{equation}
and $\mu_s$ is the unique solution to the following SPDE in the sense of distribution on $[t,T]$,
\begin{equation}
    \mathrm{d}\mu_s(x)=\Big[\frac{\widehat{\sigma}^2}{2}\mathrm{tr}(\partial_{xx}\mu_s(x))-\divergence\big(\mu_s(x)B(s,x,\mu_s)\big)\Big]\,\mathrm{d}s-\sigma_0\partial_x\mu_s(x)\cdot \mathrm{d}B_s^0,\quad\mu_{t}=\mu.
\end{equation}
Using the above representation formula, we can denote the solution to the linear transport equation (\ref{linear transport}) by $U=\Psi(T,A,B,U_T)$. Since we study the vectorial master equation, it is natural to require $U=\partial_x V$ to be curl free in $x$. Then the definition of Lipschitz solution is given as follows.
\begin{definition}
    Given a terminal condition $U_T$, a Lipschitz solution to (\ref{vectorial master equation}) on $[t_0,T]$ is a function $U\in C_{Lip}^0(\Theta)$ which is curl-free in $x$, satisfying
    \begin{equation}
        U=\Psi\big(T,\partial_x H(\cdot,U,\rho),\partial_p H(\cdot,U,\rho),\partial_x G\big),\quad\rho:=\Phi\big((id,U(t,\cdot,\mu))\#\mu\big).
    \end{equation}
\end{definition}
\begin{proposition}
    Let Assumptions \ref{regularity assumption},\ref{convexity assumption}(i) hold. Then $U$ is a vectorial weak solution to the master equation (\ref{master equation}) if and only if it is a Lipschitz solution to the vectorial master equation (\ref{vectorial master equation}).
\end{proposition}
\begin{proof}
    We first assume $U$ is a vectorial weak solution to the master equation. Then for any $t\in[t_0,T]$ and $\mu\in\mathcal{P}_2(\R^d)$, there exists a weak solution $\mu_s$ to the SPDE (\ref{vectorial SPDE}) on $[t,T]$ with initial condition $\mu_t=\mu$. Define $\rho_s:=\Phi\big((id,U(s,\cdot,\mu_s)\#\mu_s)\big)$. For this $(\mu_s,\rho_s)$, from the uniqueness argument in Theorem \ref{weak solution theorem}, we have $U(s,X_s^{\xi,x},\mu_s)=\nabla Y_s^{\xi,x}$, where $(X^{\xi,x},\nabla Y^{\xi,x})$ is given by the decoupled FBSDE,
    \begin{equation}\left\{
    \begin{aligned}
        &X_s^{\xi,x}=x+\int_t^s\partial_p H(X_r^{\xi,x},U(r,X_r^{\xi,x},\mu_r),\rho_r)\,\mathrm{d}r+\sigma\,\mathrm{d}B_s^t+\sigma_0\,\mathrm{d}B_s^{0,t},\\
        &\nabla Y_s^{\xi,x}=\partial_x G(X_T^{\xi,x},\mu_T)+\int_s^T\partial_x H(X_r^{\xi,x},\nabla Y_r^{\xi,x},\rho_r)\,\mathrm{d}r\\
        &\qquad\qquad-\int_s^T\nabla Z_s^{\xi,x}\,\mathrm{d}B_s-\int_s^T\nabla Z_s^{0,\xi,x}\,\mathrm{d}B_s^0.
    \end{aligned}\right.
    \end{equation}
    Taking the expectation, we have
    \begin{equation}
        U(t,x,\mu)=\nabla Y_t^{\xi,x}=\E\Big[\partial_x G(X_T^{\xi,x},\mu_T)+\int_t^T\partial_x H(X_s^{\xi,x},U(s,X_s^{\xi,x},\mu_s),\rho_s)\,\mathrm{d}s\Big].
    \end{equation}
    Therefore, $U$ is a Lipschitz solution to the vectorial master equation.
    \par Conversely, suppose $U$ is a Lipschitz solution to the master equation. Note that in the definition of Lipschitz solution, $\mu_t$ is exactly the weak solution to the SPDE (\ref{vectorial SPDE}). By the flow property, we have
    \begin{equation}\label{Lipschitz solution BSDE}
        U(t,X_t^{\xi,x},\mu_t)=\E_{\mathcal{F}_t}\Big[\partial_x G(X_T^{\xi,x},\mu_T)+\int_t^T\partial_x H(s,U(s,X_s^{\xi,x},\mu_s),\rho_s)\,\mathrm{d}s\Big],
    \end{equation}
    where $X^{\xi,x}$ is determined by the SDE
    \begin{equation}\label{Lipschitz solution SDE}
        X_t^{\xi,x}=x+\int_{t_0}^t\partial_p H(X_s^{\xi,x},U(s,X_s^{\xi,x},\mu_s),\rho_s)\,\mathrm{d}s+\sigma B_t^{t_0}+\sigma_0 B_t^{0,t_0}.
    \end{equation}
    Let $\delta$ be the constant in Proposition \ref{local FBSDE}, where the Lipschitz constant of $\partial_x G$ is replaced by that of $U$, then the FBSDE (\ref{FBSDE xi,x}) is well-posed on $[T-\delta,T]$. We can directly check that $X_t^{\xi,x}$ in (\ref{Lipschitz solution SDE}) and $\nabla Y_t^{\xi,x}:=U(t,X_t^{\xi,x},\mu_t)$ in (\ref{Lipschitz solution BSDE}) is the unique solution to the FBSDE (\ref{FBSDE xi,x}) on $[T-\delta,T]$. Since $U$ is uniformly Lipschitz on $[t_0,T]$, the local solution can be extended backward. Hence $U$ is the decoupling field of the FBSDE (\ref{FBSDE xi,x}), then from the existence argument in Theorem \ref{weak solution theorem}, we see that $u(t,x):=U(t,x,\mu_t)$ is a weak solution to the BSPDE in (\ref{vectorial MFGC system}). Therefore, $U$ is a vectorial weak solution to the master equation.
\end{proof}

\subsection{Relations with monotone solutions}
As we said in the introduction, we select the Hilbert-space approach of \cite{cardaliaguet2022monotone} to establish the equivalence between our notion of weak solution and monotone solutions when only common noise is present. In the displacement $\lambda$-monotonicity setting, we further extend the notion of a monotone solution to that of a displacement $\lambda$-monotone solution. 

\par To facilitate the presentation of the definitions, we introduce the following auxiliary functions: for any $t\in [0,T], (\xi_1,\xi_2) \in \mathbb{H}$, 
    \begin{equation}
    \begin{aligned}    
W(t,\xi_1,\xi_2):=&\mathbb{E}\big[U(t,\xi_1,\mu_1)\cdot(\xi_1-\xi_2)+\frac{\lambda}{2}\vert\xi_1-\xi_2\vert^2\big],\\
Q(t,\xi_1,\xi_2):=&\mathbb{E}\big[V(t,\xi_1,\mu_1)-V(t,\xi_2,\mu_1)\big],
\end{aligned}
\end{equation}
where $\mu_1=\mathcal{L}(\xi_1)$.
We now provide the precise definitions of these two notions.
\begin{definition}\label{displacement monotone solution definition}
A function 
$U\in C^0_{Lip}(\Theta)$ with $U(T,\cdot)=\partial_xG$ is called a displacement $\lambda$-monotone solution to the vectorial master equation \eqref{vectorial master equation} if $U$ is curl-free in $x$, and there exists a modulus of continuity $\omega:[0,\infty)\to[0,\infty)$ such that
\begin{equation}\label{U terminal uniform continuity}
    |U(t,x,\mu)-\partial_xG(x,\mu)|\leq \big(1+|x|+M_1(\mu)\big)\omega(T-t).
\end{equation}
We further require the function $W$ to be a viscosity supersolution, in the sense of Definition~\ref{def viscosity supersolution}, of \begin{equation}\label{master equation W}
\begin{aligned}
&\partial_tW(t,\xi_1,\xi_2)+\frac{\sigma_0^2}{2}\sum_{j=1}^{d}D_{(\xi_1,\xi_2)(\xi_1,\xi_2)}W(t,\xi_1,\xi_2)\left((\boldsymbol{e}_j,\boldsymbol{e}_j),(\boldsymbol{e}_j,\boldsymbol{e}_j)\right)\\
&+\mathbb{E}\big[\big(D_{\xi_1}W(t,\xi_1,\xi_2)+
D_{\xi_2}W(t,\xi_1,\xi_2)\big)\cdot\partial_pH\left(\xi_1,-D_{\xi_2}W(t,\xi_1,\xi_2)-\lambda(\xi_1-\xi_2),\rho\right)\big]\\
&+\mathbb{E}\big[\partial_xH\left(\xi_1,-D_{\xi_2}W(t,\xi_1,\xi_2)-\lambda(\xi_1-\xi_2),\rho\right)\cdot(\xi_1-\xi_2)\big]=0,
\end{aligned}
\end{equation}
where $\rho=\Phi\Big(\mathcal{L}\big(\xi_1,-D_{\xi_2}W(t,\xi_1,\xi_2)-\lambda(\xi_1-\xi_2)\big)\Big),\,(t,\xi_1,\xi_2)\in[0,T]\times \mathbb{H}$.
\end{definition}
\begin{definition}\label{Larsy-Lion monotone solution definition}
A function $V \in C_{Lip}^{0,1}(\Theta)$ with $V(T,\cdot)=G$ is called a Lasry-Lions monotone solution to the master equation \eqref{master equation} if there exists a modulus of continuity $\omega:[0,\infty)\to[0,\infty)$ such that
\begin{equation}\label{V terminal uniform continuity}
    |\partial_xV(t,x,\mu)-\partial_xG(x,\mu)|\leq \big(1+|x|+M_1(\mu)\big)\omega(T-t).
\end{equation}
We further require the associated function \(Q\) to be a viscosity supersolution, in the sense of Definition~\ref{def viscosity supersolution}, of
\begin{equation}
\begin{aligned}
&\partial_tQ(t,\xi_1,\xi_2)+\frac{\sigma_0^2}{2}\sum_{j=1}^dD_{(\xi_1,\xi_2)(\xi_1,\xi_2)}Q(t,\xi_1,\xi_2)\left((\boldsymbol{e}_j,\boldsymbol{e}_j),(\boldsymbol{e}_j,\boldsymbol{e}_j)\right)\\
&+\mathbb{E}\big[\big(
D_{\xi_1}Q(t,\xi_1,\xi_2)-\partial_{x}V(t,\xi_1,\mu_1)\big)\cdot\partial_pH\left(\xi_1,\partial_xV(t,\xi_1,\mu_1),\rho\right)\big]\\
&+\mathbb{E}\big[H(\xi_1,\partial_xV(t,\xi_1,\mu_1),\rho)-H(\xi_2,-D_{\xi_2}Q(t,\xi_1,\xi_2),\rho)\big]
=0,
\end{aligned}
\end{equation}
where $\rho=\Phi\Big(\mathcal{L}\big(\xi_1,\partial_xV(t,\xi_1,\mu_1)\big)\Big),\,(t,\xi_1,\xi_2)\in[0,T]\times \mathbb{H}$.
\end{definition}
\begin{remark}
The above definitions differ from the formulation in \cite[Section 2.2.2]{meynard2025monotone} in the treatment of the second-order terms due to common noise. There, an auxiliary variable is introduced so that these terms involve second-order derivatives only with respect to this finite-dimensional variable. This allows the second-order comparison argument to be carried out in finite dimensions.
\end{remark}
\begin{remark}
A further difference from \cite{meynard2025monotone} concerns the test functionals. Our formulation uses $\mathbb C^{1,2}$ test functionals on $[0,T]\times\mathbb L^2$, whereas those in \cite{meynard2025monotone} contain a fourth-moment penalty and are finite only for random variables in $\mathbb L^4$. This restriction serves the stability argument: uniform fourth-moment bounds at minimum points yield relative compactness of their laws in $(\mathcal P_2,\mathrm W_2)$, allowing stability to be proved for coefficients that are only continuous in the measure variable. This stability result is then used to construct monotone solutions by approximation; see \cite[Remark 2.8 and Theorem 2.12]{meynard2025monotone}. For uniqueness, \cite{meynard2025monotone} also uses the above compactness to exclude minimum points at the time boundary. 
\end{remark}
The following lemma shows that the monotonicity condition implies uniqueness, thereby laying the foundation for the subsequent proof.
\begin{lemma}\label{lemma:uniqueness_monotone}
The following statements hold.

\noindent
(i) Let $U_i\in C^0_{Lip}(\Theta)$, $i=1,2$. If, for any
$(t,\xi_1,\xi_2)\in[0,T]\times\mathbb H$,
\begin{equation*}\label{cond:disp}
\mathbb{E}\big[
\big(U_1(t,\xi_1,\mathcal{L}(\xi_1))
     - U_2(t,\xi_2,\mathcal{L}(\xi_2))\big)
     \cdot (\xi_1-\xi_2)
+ \lambda \lvert \xi_1-\xi_2\rvert^2
\big] \ge 0,
\end{equation*}
then $U_1(t,x,\mu)=U_2(t,x,\mu),$ $(t,x,\mu)\in\Theta.$
\medskip

\noindent
(ii) Let $V_i\in C^{0,1}_{Lip}(\Theta)$, $i=1,2$. If, for any
$(t,\xi_1,\xi_2)\in[0,T]\times\mathbb H$,
\begin{equation*}\label{cond:LL}
\mathbb{E}\big[
V_1(t,\xi_1,\mathcal{L}(\xi_1))
- V_1(t,\xi_2,\mathcal{L}(\xi_1))
- V_2(t,\xi_1,\mathcal{L}(\xi_2))
+ V_2(t,\xi_2,\mathcal{L}(\xi_2))
\big] \ge 0.
\end{equation*}
then
$\partial_xV_1(t,x,\mu)=\partial_xV_2(t,x,\mu)$, for any $(t,x,\mu)\in \Theta$.
\end{lemma}
\begin{proof}
The proof of (i) follows the same idea as \cite[Lemma 2.6]{meynard2025monotone}, and we provide the details below. 
\par Let $(t, \mu) \in [0, T] \times \mathcal{P}_2(\mathbb{R}^d)$ be fixed, and let $\xi_0\in \mathbb{L}^2(\mathcal{F}_0^1;\mu)$. For any $Y_0 \in \mathbb{L}^2(\mathcal{F}_0^1;\mathbb{R}^d)$ and $h > 0$, we take
$\xi_1 = \xi_0 + h Y_0, \ \xi_2 = \xi_0.$
Substituting these into (i), we obtain
$$\mathbb{E} \Big[ \big( U_1(t, \xi_0 + h Y_0, \mathcal{L}(\xi_0 + h Y_0)) - U_2(t, \xi_0, \mu) \big) \cdot (h Y_0) + \lambda h^2 |Y_0|^2 \Big] \geq 0.$$
By the Lipschitz continuity of $U_1$ and $U_2$, dividing by $h$ and passing to the limit $h \to 0^+$ gives:
\begin{equation*}\mathbb{E} \Big[ \big( U_1(t, \xi_0, \mu) - U_2(t, \xi_0, \mu) \big) \cdot Y_0 \Big] \geq 0. 
\end{equation*}
Since $Y_0$ is arbitrary, we have $U_1(t, \xi_0, \mu) = U_2(t, \xi_0, \mu)$, $\mathbb{P}$-a.s.
Consequently, $U_1(t, x, \mu) = U_2(t, x, \mu)$ for any $x \in \text{supp}(\mu)$. Since $\mu$ is arbitrary and $U_1, U_2$ are continuous, it follows that $U_1(t, x, \mu) = U_2(t, x, \mu)$, for any $(t, x, \mu) \in \Theta.$ 
\par For (ii), by \cite[Lemma 1.1]{cardaliaguet2022monotone} yields $\partial_xV_1=\partial_xV_2$, for a.e. $x$, Since $\partial_xV_1$ and $\partial_xV_2$ are continuous in $x$, it follows that $\partial_xV_1=\partial_xV_2$, for any $(t,x,\mu)$.
\end{proof}

\begin{proposition}\label{Prop1}
Assume that Assumption \ref{regularity assumption}, \ref{convexity assumption} (i) holds. Let $\rho_t:[t_0,T]\times\Omega\rightarrow\mathcal{P}_2(\mathbb{R}^{2d})$ be $\mathbb{F}^{B_0}$-progressively measurable with $\sup_{t_0\leq t\leq T}\mathbb{E}\big[M^2_2(\rho_t)\big]<\infty$, $\mu_t=\pi_1\# \rho_t$. Set
$\mathcal G_t:=\mathcal F_t^0\vee\mathcal F_0^1$, $\mathbb G:=\{\mathcal G_t\}_{0\le t\le T}.$
Let $X$ be a $\mathbb G$-progressively measurable process satisfying $\sup_{t_0\leq t\leq T}\mathbb{E}[\vert X_t\vert^2]< \infty$. The pair $(\rho,X)$ is not necessarily part of a solution to \eqref{FBSDE xi}-\eqref{FBSDE xi,x}. Then, the following BSDE on $[t_0,T]$ has a unique solution,
\begin{equation}\label{BSDE x}
\begin{aligned}
\nabla Y_t=&\partial_xG(X_T,\mu_T)+\int_t^T\partial_xH(X_s,\nabla Y_s,\rho_s
)\mathrm{d}s-\int_t^T\nabla Z_s^{0}\,\mathrm{d}B_s^0.
\end{aligned}
\end{equation}
Moreover, 
\begin{equation*}
\sup_{t_0\leq t\leq T}\mathbb {E}[\vert\nabla Y_t\vert^2]
+\mathbb{E}[
\int_{t_0}^T
\vert \nabla Z_t^{0}\vert ^2\mathrm {d}t]
\leq C\big(\vert\partial_xG(0,\delta_0)\vert^2+\vert\partial_xH(0,0,\delta_{(0,0)}) \vert^2+\sup_{t_0\leq t \leq T}\mathbb{E}[\vert X_t\vert^2+M_2^2(\rho_t)]\big), 
\end{equation*}
where $C$ depends only on $d,T,c_0$, and $C_{Lip}$. 
\end{proposition}
\begin{proof} Under Assumption \ref{regularity assumption}, 
$\partial_xH$ is Lipschitz continuous in 
$p$. Moreover, by the Lipschitz continuity of $\partial_xH$ and $\partial_xG$,
\begin{equation*}
\begin{aligned}
&\mathbb{E}[\vert \partial_xG(X_T,\mu_T)\vert^2+\int_{t_0}^T\vert\partial_xH(X_t,0,\rho_t)\vert^2\mathrm{d}t]\\
        \leq &C\Big(\vert\partial_xG(0,\delta_0)\vert^2+\mathbb{E}[\vert X_T \vert^2+W^2_1(\mu_T,\delta_0)]+\vert\partial_xH(0,0,\delta_{(0,0)})\vert^2+\int_{t_0}^T\mathbb{E}[\vert X_t \vert^2+W^2_1(\rho_t,\delta_{(0,0)})]\mathrm{d}t\Big).   \end{aligned} 
\end{equation*}
Since 
\begin{equation*}
\mathrm{W}_1(\mu_t,\delta_0)=\sup_{ Lip(\phi)\leq 1}\int_{\mathbb{R}^{d}}\big(\phi(x)-\phi(0)\big)\mu_t(\mathrm{d}x)=M_1(\mu_t)\leq M_1(\rho_t)=\mathrm{W}_1(\rho_t,\delta_{(0,0)})\leq M_2(\rho_t),
\end{equation*}
we have $ \mathbb{E}[\vert \partial_xG(X_T,\mu_T)\vert^2+\int_{t_0}^T\vert\partial_xH(X_t,0,\rho_t)\vert^2\mathrm{d}t]<\infty$. Therefore, by the standard theory of BSDE, the above BSDE \eqref{BSDE x} admits a unique solution 
$(\nabla Y,\nabla Z^{0})$. Moreover, the standard a priori estimate for BSDEs yields
\begin{equation*}
    \begin{aligned}
        &\sup_{t_0\leq t\leq T}\mathbb {E}[\vert\nabla Y_t\vert^2]
+\mathbb{E}[
\int_{t_0}^T
\vert \nabla Z_t^{0}\vert ^2\mathrm {d}t]\\
        \leq
        & C\mathbb{E}[\vert \partial_xG(X_T,\mu_T)\vert^2+\int_{t_0}^T\vert\partial_xH(X_t,0,\rho_t)\vert^2\mathrm{d}t]\\
        \leq &C\big(\vert\partial_xG(0,\delta_0)\vert^2+\vert\partial_xH(0,0,\delta_{(0,0)}) \vert^2+\sup_{t_0\leq t \leq T}\mathbb{E}[\vert X_t\vert^2+M^2_2(\rho_t)]\big).
    \end{aligned}
\end{equation*}
\end{proof}
\par We now present the relation between the weak solution and the monotone solution considered in this paper. 
Since monotone solutions are defined without individual noise, we restrict ourselves to the case $\sigma=0$ and omit this assumption below.
\begin{theorem}\label{equivalence}
Assume that Assumptions \ref{regularity assumption}, \ref{quadratic growth assumption},  
\ref{convexity assumption} hold. Then the following statements hold.

(i) Suppose that $G$ and $L$ satisfy the displacement
$\lambda$-monotonicity conditions
\eqref{G displacement}\eqref{L displacement}, and that Assumption
\ref{locally bounded assumption} holds. Let
$U\in C_{Lip}^{0}(\Theta)$ satisfy
$U(T,\cdot,\cdot)=\partial_xG.$
Then $U$ is a vectorial weak solution to the master equation
\eqref{master equation} if and only if it is a displacement
$\lambda$-monotone solution to the vectorial master equation
\eqref{vectorial master equation}.

(ii) Suppose that $G$ and $L$ satisfy the Lasry-Lions
monotonicity conditions
\eqref{G Lasry-Lions}\eqref{L Lasry-Lions}, and that Assumptions
\ref{G L convexity assumption} hold.
Let $V\in C_{Lip}^{0,1}(\Theta)$ satisfy
$
V(T,\cdot,\cdot)=G.
$
If $V$ is a weak solution to the master equation
\eqref{master equation}, then it is a Lasry-Lions monotone solution
to \eqref{master equation}. Conversely, if $V$ is a Lasry-Lions
monotone solution to \eqref{master equation}, then its spatial
derivative $\partial_xV$ is a vectorial weak solution to the master
equation \eqref{master equation}.
\end{theorem}
\begin{proof}
(i) 
\textbf{Step 1.1.} We first show that any vectorial weak solution $U$ is a displacement $\lambda$-monotone solution. Let $(X,\mu)$ be given as in \eqref{SDE with law mu}. Following the same argument used for the uniqueness proof of vectorial weak solutions in Theorem \ref{weak solution theorem}, it follows that $\nabla Y_t := U(t, X_t,\mu_t)$ is the unique solution to the BSDE \eqref{uniqueness of the weak solution: BSDE}. This implies that $U$ serves as the decoupling field of the FBSDEs \eqref{FBSDE xi}-\eqref{FBSDE xi,x}. By Theorem~\ref{FBSDE well-posed}, we have
$U$ satisfies \eqref{U terminal uniform continuity}.
\par For any $\varphi\in \mathbb{C}^{1,2}([0,T]\times\mathbb{H})$, suppose $W(t,\xi_1,\xi_2)-\varphi(t,\xi_1,\xi_2)$ attains a minimum at $\left(t^*,\xi_1^*,\xi^*_2\right)\in[0,T)\times \mathbb{H}$.
We introduce the following McKean-Vlasov FBSDE on $[t^*,T]$:
\begin{equation}\label{t40}
    \left\{
    \begin{aligned}
    X_t^{\xi^*_1}=&\xi^*_1+\int_{t^*}^t \partial_p H(X_s^{\xi^*_1},  \nabla Y_s^{\xi^*_1},\rho^{\xi^*_1}_s) \mathrm{d} s+\sigma_0 B_t^{0, t^*},\\
\nabla Y_t^{\xi^*_1}=&\partial_x G(X_T^{\xi^*_1}, \mu_T^{\xi^*_1})+\int_t^T \partial_x H(X_s^{\xi^*_1},  \nabla  Y_s^{\xi^*_1},\rho^{\xi^*_1}_s) \mathrm{d} s-\int_t^T  \nabla  Z_s^{0, \xi^*_1} \mathrm{d} B_s^0,\\
\rho^{\xi^*_1}_t=&\mathcal{L}^0\big(X_t^{\xi^*_1},\partial_p H(X_t^{\xi^*_1},  \nabla Y_t^{\xi^*_1},\rho^{\xi^*_1}_t) \big),\ \mu_t^{\xi^*_1}=\mathcal{L}^0(X_t^{\xi^*_1}),
    \end{aligned}
    \right.
\end{equation}
where $\rho_{t}^{\xi_1^*}=\Phi\big(\mathcal{L}^0(X_t^{\xi_1^*},\nabla Y_t^{\xi_1^*})\big)$.
We also introduce the SDE:
\begin{equation}\label{t41}
    \begin{aligned}
X_t^{\xi^*_2}=&\xi^*_2+\sigma_0 B_t^{0, t^*},\  t\in [t^*,T].
    \end{aligned}
\end{equation}
For any $\epsilon_0>0$, 
define 
\begin{equation*}
    \begin{aligned}
        \varphi_{\epsilon_0}(t,\xi_1,\xi_2)=~&\varphi(t^*,\xi_1^*,\xi^*_2)+(\partial_t\varphi(t^*,\xi_1^*,\xi^*_2)-\epsilon_0)(t-t^*)+\left\langle (p_1^*,p_2^*),(\xi_1,\xi_2)-(\xi_1^*,\xi^*_2)\right\rangle_{\mathbb{H}}\\
        &+\frac{1}{2}\left(\chi-\epsilon_0I\right)\big((\xi_1,\xi_2)-(\xi_1^*,\xi^*_2),(\xi_1,\xi_2)-(\xi_1^*,\xi^*_2)\big),
    \end{aligned}
\end{equation*}
where
$$
\begin{aligned}
p^*_1=& 
D_{\xi_1}\varphi(t^*,\xi^*_1,\xi^*_2),\ p^*_2= D_{\xi_2}\varphi(t^*,\xi^*_1,\xi^*_2),\ 
\chi=D_{(\xi_1,\xi_2)(\xi_1,\xi_2)}\varphi(t^*,\xi^*_1,\xi^*_2).
\end{aligned}
$$
Since $\varphi\in \mathbb{C}^{1,2}([0,T]\times\mathbb{H})$, we have
\begin{equation*}
    \begin{aligned}
        \varphi(t,\xi_1,\xi_2)=~&\varphi(t^*,\xi_1^*,\xi^*_2)+\partial_t\varphi(t^*,\xi_1^*,\xi^*_2)(t-t^*)+\left\langle (p_1^*,p_2^*),(\xi_1,\xi_2)-(\xi_1^*,\xi^*_2)\right\rangle_{\mathbb{H}}\\
        &+\frac{1}{2}\chi\big(\xi_1,\xi_2)-(\xi_1^*,\xi^*_2),(\xi_1,\xi_2)-(\xi_1^*,\xi^*_2)\big)+r(t,\xi_1,\xi_2),
    \end{aligned}
\end{equation*}
where $r(t,\xi_1,\xi_2):=o\left(\vert t-t^*\vert+\|(\xi_1,\xi_2)-(\xi_1^*,\xi^*_2)\|^2_2\right)$. Hence, for $t\geq t^*$, there exists $R_{\epsilon_0}>0$ such that 
\begin{equation*}
\varphi(t,\xi_1,\xi_2)-\varphi_{\epsilon_0}(t,\xi_1,\xi_2)=\epsilon_0(t-t^*)+\frac{\epsilon_0}{2}\|(\xi_1,\xi_2)-(\xi_1^*,\xi^*_2)\|^2_2+r(t,\xi_1,\xi_2)\geq0,
\end{equation*}
whenever $d^2(t,\xi_1,\xi_2):=\vert t-t^*\vert^2+\|(\xi_1,\xi_2)-(\xi_1^*,\xi^*_2)\|^2_2\leq 4R^2_{\epsilon_0}$.
It follows that
\begin{equation*}
\begin{aligned}
&W(t^*,\xi^*_1,\xi^*_2)-\varphi_{\epsilon_0}(t^*,\xi^*_1,\xi^*_2)\leq W(t,\xi_1,\xi_2)-\varphi_{\epsilon_0}(t,\xi_1,\xi_2),\ 
\text{for} \ d^2(t,\xi_1,\xi_2)\leq 4R^2_{\epsilon_0},
    \end{aligned}
\end{equation*}
Define $$\tilde{\varphi}_{\epsilon_0}(t,\xi_1,\xi_2)=\varphi_{\epsilon_0}(t,\xi_1,\xi_2)-M_{\epsilon_0}\eta\left(\frac{d^2(t,\xi_1,\xi_2)}{R^2_{\epsilon_0}}\right)(1+\|\xi_1\|_2^2+\|\xi_2\|_2^2),$$
where $\eta\in C^{\infty}(\mathbb{R}_+;[0,1])$ satisfies $\eta(x)=0, \ x\leq 1$, $\eta(x)=1,\ x\geq 4$. 
Since $U$ is Lipschitz, $W$ and $\varphi_{\epsilon_0}$ have at most quadratic growth. Thus, there exists $C_{\epsilon_0}>0$ such that 
\begin{equation*}
W(t,\xi_1,\xi_2)-\varphi_{\epsilon_0}(t,\xi_1,\xi_2)-\big(W(t^*,\xi^*_1,\xi^*_2)-\varphi_{\epsilon_0}(t^*,\xi^*_1,\xi^*_2)\big)\geq-C_{\epsilon_0}(1+\|\xi_1\|_2^2+\|\xi_2\|_2^2),
\end{equation*}
Choose $M_{\epsilon_0}\geq C_{\epsilon_0}$. If $d^2(t,\xi_1,\xi_2)\geq 4R_{\epsilon_0}^2$, then
\begin{equation*}
W(t,\xi_1,\xi_2)-\tilde{\varphi}_{\epsilon_0}(t,\xi_1,\xi_2)-\big(W(t^*,\xi^*_1,\xi^*_2)-\tilde{\varphi}_{\epsilon_0}(t^*,\xi^*_1,\xi^*_2)\big)\geq(M_{\epsilon_0}-C_{\epsilon_0})(1+\|\xi_1\|_2^2+\|\xi_2\|_2^2)\geq0.
\end{equation*}
Consequently,
\begin{equation*}\label{minimum inequality}
W(t^*,\xi^*_1,\xi^*_2)-\tilde{\varphi}_{\epsilon_0}(t^*,\xi^*_1,\xi^*_2)\leq W(t,\xi_1,\xi_2)-\tilde{\varphi}_{\epsilon_0}(t,\xi_1,\xi_2),     
\end{equation*}
for any $(t,\xi_1,\xi_2)\in [t^*,T]\times \mathbb{H}$.
Moreover,
\begin{equation}\label{tildevarphi derivatives}
\begin{aligned}
    &\partial_t\tilde{\varphi}_{\epsilon_0}(t^*,\xi^*_1,\xi^*_2)=\partial_t\varphi(t^*,\xi^*_1,\xi^*_2)-\epsilon_0,\ D_{(\xi_1,\xi_2)}\tilde{\varphi}_{\epsilon_0}(t^*,\xi^*_1,\xi^*_2)=(p_1^*,p_2^*),\\ &D_{(\xi_1,\xi_2)(\xi_1,\xi_2)}\tilde{\varphi}_{\epsilon_0}(t^*,\xi^*_1,\xi^*_2)=\chi-\epsilon_0I.
    \end{aligned}
\end{equation}
Since $U$ is the decoupling field of the FBSDEs \eqref{FBSDE xi}-\eqref{FBSDE xi,x}, for any $h>0$ small and $t=t^*+h$, we obtain
\begin{equation}\label{eq1}
\begin{aligned}
&\mathbb{E}[\nabla  Y_t^{\xi^*_1}\cdot(X_t^{\xi^*_1}-X_t^{\xi^*_2})+\frac{\lambda}{2}\vert X_t^{\xi^*_1}-X_t^{\xi^*_2}\vert^2-\tilde{\varphi}_{\varepsilon_0}(t,X_t^{\xi^*_1},X_t^{\xi^*_2})]\\
\geq&\mathbb{E}[\nabla Y_{t^*}^{\xi^*_1}\cdot(\xi^*_1-\xi^*_2)+\frac{\lambda}{2}\vert \xi^*_1-\xi^*_2\vert^2-\tilde{\varphi}_{\varepsilon_0}(t^*,\xi^*_1,\xi^*_2)].
\end{aligned}
\end{equation}
By applying Itô's formula to $\nabla Y_t^{\xi^*_1}\cdot(X_t^{\xi^*_1}-X_t^{\xi^*_2})+\frac{\lambda}{2}\vert X_t^{\xi^*_1}-X_t^{\xi^*_2}\vert^2$, we have
\begin{equation}\label{eq2}
\begin{aligned}
&\mathbb{E}[\nabla  Y_t^{\xi^*_1}\cdot(X_t^{\xi^*_1}-X_t^{\xi^*_2})+\frac{\lambda}{2}\vert X_t^{\xi^*_1}-X_t^{\xi^*_2}\vert^2-\nabla Y_{t^*}^{\xi^*_1}\cdot(\xi^*_1-\xi^*_2)-\frac{\lambda}{2}\vert \xi^*_1-\xi^*_2\vert^2]\\
=&h\mathbb{E}[\nabla Y_{t^*}^{\xi^*_1}\cdot\partial_pH(\xi^*_1,\nabla Y_{t^*}^{\xi^*_1},\rho_{t^*}^{\xi^*_1})-(\xi^*_1-\xi^*_2)\cdot\partial_xH(\xi^*_1,\nabla Y_{t^*}^{\xi^*_1},\rho_{t^*}^{\xi^*_1})\\
&+\lambda(\xi^*_1-\xi^*_2)\cdot \partial_pH(\xi^*_1,\nabla Y_{t^*}^{\xi^*_1},\rho_{t^*}^{\xi^*_1})]+\mathcal{I}_1,
\end{aligned}
\end{equation}
where 
$$
\begin{aligned}
\mathcal{I}_1:=
&\mathbb{E}\Big\{\int_{t^*}^t\Big[\nabla Y_s^{\xi^*_1}\cdot
\partial_pH(X_s^{\xi^*_1},\nabla Y_s^{\xi^*_1},\rho_{s}^{\xi^*_1})-\nabla Y_{t^*}^{\xi^*_1}\cdot\partial_pH(\xi^*_1,\nabla Y_{t^*}^{\xi^*_1},\rho_{t^*}^{\xi^*_1})\\
&-(X_s^{\xi^*_1}-X_s^{\xi^*_2})\cdot\partial_xH(X_s^{\xi^*_1},\nabla Y_s^{\xi^*_1},\rho_{s}^{\xi^*_1})+(\xi^*_1-\xi^*_2)\cdot\partial_xH(\xi^*_1,\nabla Y_{t^*}^{\xi^*_1},\rho_{t^*}^{\xi^*_1})\\
&+\lambda(X_s^{\xi^*_1}-X_s^{\xi^*_2})\cdot \partial_pH(X_s^{\xi^*_1},\nabla Y_s^{\xi^*_1},\rho_s^{\xi^*_1})-\lambda(\xi^*_1-\xi^*_2)\cdot \partial_pH(\xi^*_1,\nabla Y_{t^*}^{\xi^*_1},\rho_{t^*}^{\xi^*_1})\Big]\mathrm{d}s\Big\}.
\end{aligned}
$$
Next, we show that $\mathcal{I}_1=o(h)$. 
Combining the Lipschitz continuity of $\partial_xH$, $\partial_pH$, and $\Phi$, we obtain
\begin{equation*}
\begin{aligned}
\vert \mathcal{I}_1 \vert
\leq & C\mathbb{E}\Big\{\int_{t^*}^t \big[\big(|X_s^{\xi_1^*}-X_s^{\xi_2^*}|+\vert\nabla Y_s^{\xi_1^*}\vert\big)\cdot\big(\vert X_s^{\xi^*_1}-\xi^*_1\vert+\vert \nabla  Y_{s}^{\xi^*_1}-\nabla Y_{t^*}^{\xi^*_1}\vert\\
&+\mathrm{W}_1\big(\mathcal{L}^0(X_s^{\xi_1^*},\nabla  Y_{s}^{\xi^*_1}),\mathcal{L}^0(\xi^*_1,\nabla Y_{t^*}^{\xi^*_1})\big)+\big(\vert X_s^{\xi^*_1}-\xi^*_1\vert+\vert X_s^{\xi^*_2}-\xi^*_2\vert+\vert \nabla  Y_{s}^{\xi^*_1}-\nabla Y_{t^*}^{\xi^*_1}\vert\big)\\
&\cdot\big(1+\vert \xi_1^*\vert+\vert\nabla Y_{t^*}^{\xi_1^*}\vert+\mathrm{W}_1(\mathcal{L}^0(\xi^*_1,\nabla Y_{t^*}^{\xi^*_1}),\delta_{(0,0)})\big)\big]\mathrm{d}s\Big\}\\
\leq& Ch\Big[1+\big(\sup_{t^*\le s\le t}\mathbb{E}[
|X_s^{\xi_1^*}-X_s^{\xi_2^*}|^2]\big)^{\frac{1}{2}}+\big(\sup_{t^*\leq s\leq t}\mathbb {E}[\vert\nabla Y_s^{\xi_1^*}\vert^2]\big)^{\frac{1}{2}}+(\mathbb{E}[\vert \xi_1^*\vert^2])^{\frac{1}{2}}\Big]\\
&\cdot\Big(\big(\sup_{t^*\leq s\leq t}\mathbb{E}[\big\vert \nabla  Y_{s}^{\xi^*_1}-\nabla Y_{t^*}^{\xi^*_1}\big\vert^2]\big)^{\frac{1}{2}}+\big(\sup_{t^*\leq s \leq t}\mathbb{E}[\vert X_s^{\xi^*_1}-\xi^*_1\vert^2+\vert X_s^{\xi^*_2}-\xi^*_2\vert^2]\big)^{\frac{1}{2}}\Big)\\
\leq & Ch \Big[\big(\sup_{t^*\leq s\leq t}\mathbb{E}[\big\vert \nabla  Y_{s}^{\xi^*_1}-\nabla Y_{t^*}^{\xi^*_1}\big\vert^2]\big)^{\frac{1}{2}}+\big(\sup_{t^*\leq s \leq t}\mathbb{E}[\vert X_s^{\xi^*_1}-\xi^*_1\vert^2+\vert X_s^{\xi^*_2}-\xi^*_2\vert^2]\big)^{\frac{1}{2}}\Big].
\end{aligned}
\end{equation*}
Applying It\^o's formula to $\vert \nabla  Y_{t}^{\xi^*_1}-\nabla Y_{t^*}^{\xi^*_1}\vert^2$, we have
\begin{equation*}
    \begin{aligned}
        \mathbb{E}[\big\vert \nabla  Y_{t}^{\xi^*_1}-\nabla Y_{t^*}^{\xi^*_1}\big\vert^2]=&\mathbb{E}\Big[\int_{t^*}^t\big(\vert\nabla Z_s^{0,\xi_1^*}\vert^2-2(\nabla  Y_{s}^{\xi^*_1}-\nabla Y_{t^*}^{\xi^*_1})\cdot\partial_xH(X_s^{\xi^*_1},\nabla Y_s^{\xi^*_1},\rho^{\xi^*_1}_s)\big)\mathrm{d}s\Big].
    \end{aligned}
\end{equation*}
By the well-posedness of the FBSDE \eqref{t40}, we have
\begin{equation*}
\sup_{t^*\leq t\leq T}\mathbb {E}[\vert X^{\xi_1^*}_t\vert^2+\vert\nabla Y^{\xi_1^*}_t\vert^2]
+\mathbb{E}[
\int_{t^*}^T
\vert \nabla Z_t^{0,\xi_1^*}\vert ^2\mathrm {d}t]< \infty.   
\end{equation*}
By Gronwall's inequality, 
\begin{equation*}
    \begin{aligned}
        \mathbb{E}[\big\vert \nabla  Y_{t}^{\xi^*_1}-\nabla Y_{t^*}^{\xi^*_1}\big\vert^2]\leq&C\Big(\mathbb{E}[\int_{t^*}^t\vert\nabla Z_s^{0,\xi_1^*}\vert^2\mathrm{d}s]+h\big(1+\sup_{t^*\leq t\leq T}\mathbb {E}[\vert X_t^{\xi_1^*}\vert^2+\vert\nabla Y_t^{\xi_1^*}\vert^2]\big)\Big).
    \end{aligned}
\end{equation*}
Hence, $\sup_{t^*\leq s\leq t}\mathbb{E}[\big\vert \nabla  Y_{s}^{\xi^*_1}-\nabla Y_{t^*}^{\xi^*_1}\big\vert^2]\rightarrow 0$ as $h\rightarrow 0$.
From \eqref{t40}-\eqref{t41}, we obtain
\begin{equation*}
    \begin{aligned}
\sup_{t^*\leq s \leq t}\mathbb{E}\big[\vert X_s^{\xi^*_1}-\xi^*_1\vert^2+\vert X_s^{\xi^*_2}-\xi^*_2\vert^2\big]
\leq &C\big(h\mathbb{E}\big[\int_{t^*}^t\vert\partial_pH(X_s^{\xi^*_1},\nabla Y_s^{\xi^*_1},\rho^{\xi^*_1}_s)\vert^2\mathrm{d}s\big]+h\big)\\
\leq&Ch(1+\sup_{t^*\leq t\leq T}\mathbb {E}[\vert X_t^{\xi_1^*}\vert^2+\vert\nabla Y_t^{\xi_1^*}\vert^2]).
 \end{aligned}
\end{equation*}
Consequently,
\begin{equation*}
\sup_{t^*\le s\le t}\mathbb{E}\big[
\big(|X_s^{\xi_1^*}-\xi_1^*|^2
+|X_s^{\xi_2^*}-\xi_2^*|^2\big)\big]
\rightarrow 0 ~\text{as}~h\rightarrow 0.
\end{equation*}
Combining the above estimates yields $\mathcal{I}_1=o(h)$. 
Since, for fixed $t^*$ and $\xi_1^*$, the map $$\xi_2\mapsto \mathbb{E}\big[U(t^*,\xi^*_1,\mathcal{L}(\xi^*_1))\cdot(\xi^*_1-\xi_2)+\frac{\lambda}{2}\vert\xi^*_1-\xi_2 \vert^2\big]-\varphi(t^*,\xi^*_1,\xi_2)$$ attains a minimum at $\xi^*_2$, we obtain
\begin{equation}\label{eq4}
\nabla Y_{t^*}^{\xi^*_1}=U(t^*,\xi_1^*,\mathcal{L}(\xi^*_1))=-p_2^*-\lambda(\xi^*_1-\xi^*_2 ).
\end{equation}
Since $\vert \partial_t\tilde{\varphi}_{\epsilon_0}\vert +\|D^2\tilde{\varphi}_{\epsilon_0}\|\leq C_{\epsilon_0}$, $\|D\tilde{\varphi}_{\epsilon_0}\|_2\leq C_{\epsilon_0}(1+\|\xi_1\|_2+\|\xi_2\|_2)$, the conditions of \cite[Theorem 1.163]{FabbriGozziSwiech2017} are satisfied. Applying It\^o's formula to $\tilde{\varphi}_{\epsilon_0}(s,X_s^{\xi_1^*},X_s^{\xi_2^*})$ on $[t^*,t]$, 
\begin{equation}\label{varphi Ito}
\begin{aligned}
&\tilde{\varphi}_{\epsilon_0}
(t,X_t^{\xi_1^*},X_t^{\xi_2^*})
-\tilde{\varphi}_{\epsilon_0}
\big(t^*,\xi_1^*,\xi_2^*\big) \\
=&
\int_{t^*}^t
\Big(\partial_t\tilde{\varphi}_{\epsilon_0}(s,X_s^{\xi_1^*},X_s^{\xi_2^*})+\big\langle D_{(\xi_1,\xi_2)}\tilde{\varphi}_{\epsilon_0}(s,X_s^{\xi_1^*},X_s^{\xi_2^*}),\big(\partial_pH(X_s^{\xi_1^*},\nabla Y_{s}^{\xi^*_1},\rho_{s}^{\xi^*_1}),0\big)\big\rangle_{\mathbb{H}}\\
&+\frac{\sigma_0^2}{2}\sum_{j=1}^d D_{(\xi_1,\xi_2)(\xi_1,\xi_2)}\tilde{\varphi}_{\epsilon_0}(s,X_s^{\xi_1^*},X_s^{\xi_2^*})\left((\boldsymbol{e}_j,\boldsymbol{e}_j),(\boldsymbol{e}_j,\boldsymbol{e}_j)\right)\Big)\mathrm ds\\
&+
\sigma_0\sum_{j=1}^d
\int_{t^*}^t
\langle D_{(\xi_1,\xi_2)}\tilde{\varphi}_{\epsilon_0}
(s,X_s^{\xi_1^*},X_s^{\xi_2^*}), (\boldsymbol{e}_j,\boldsymbol{e}_j)\rangle_{\mathbb H}
\mathrm dB_s^{0,j}.
\end{aligned}
\end{equation}
By the Lipschitz continuity of $\partial_pH$ and $\Phi$, we have 
\begin{equation*}
    \begin{aligned}
&\mathbb{E}\big[\int_{t^*}^t\big\vert D_{\xi_1}\tilde{\varphi}_{\epsilon_0}(s,X_s^{\xi_1^*},X_s^{\xi_2^*})\cdot\partial_pH(X_s^{\xi_1^*},\nabla Y_{s}^{\xi^*_1},\rho_{s}^{\xi^*_1})-D_{\xi_1}\tilde{\varphi}_{\epsilon_0}(t^*,\xi^*_1,\xi^*_2)\cdot\partial_pH(\xi^*_1,\nabla Y_{t^*}^{\xi^*_1},\rho_{t^*}^{\xi^*_1})\big\vert\mathrm{d}s\big]\\
\leq& h \|D_{\xi_1}\tilde{\varphi}_{\epsilon_0}(t^*,\xi^*_1,\xi^*_2)\|_2\big(\sup_{t^*\leq s\leq t}\mathbb{E}[\vert X_s^{\xi^*_1}-\xi^*_1\vert^2+\vert \nabla Y_s^{\xi^*_1}-\nabla Y_{t^*}^{\xi^*_1}\vert^2]\big)^{\frac{1}{2}}\\
&+\int_{t^*}^t\|D_{\xi_1}\tilde{\varphi}_{\epsilon_0}(s,X_s^{\xi_1^*},X_s^{\xi_2^*})-D_{\xi_1}\tilde{\varphi}_{\epsilon_0}(t^*,\xi^*_1,\xi^*_2)\|_2\big(\mathbb{E}[\vert \partial_pH(X_s^{\xi_1^*},\nabla Y_{s}^{\xi^*_1},\rho_{s}^{\xi^*_1})\vert^2]\big)^{\frac{1}{2}}\mathrm{d}s\\
=&o(h).
    \end{aligned}
\end{equation*}
where the last equality follows from the continuity of $D_{\xi_1}\tilde{\varphi}_{\epsilon_0}$ and the preceding estimates. Then taking expectations in \eqref{varphi Ito} and using $\vert \partial_t\tilde{\varphi}_{\epsilon_0}\vert +\|D^2\tilde{\varphi}_{\epsilon_0}\|\leq C_{\epsilon_0}$, we obtain
\begin{equation}\label{t34}
    \begin{aligned}
&\mathbb{E}\big[\tilde{\varphi}_{\epsilon_0}(t,X_t^{\xi^*_1},X_t^{\xi^*_2})\big]\\
=&\tilde{\varphi}_{\epsilon_0}(t^*,\xi^*_1,\xi^*_2)+h\mathbb{E}\big[    \partial_t\tilde{\varphi}_{\epsilon_0}(t^*,\xi^*_1,\xi^*_2)+D_{\xi_1}\tilde{\varphi}_{\epsilon_0}(t^*,\xi^*_1,\xi^*_2)\cdot\partial_pH(\xi^*_1,\nabla Y_{t^*}^{\xi^*_1},\rho_{t^*}^{\xi^*_1})\\
&+\frac{\sigma_0^2}{2}\sum_{j=1}^d D_{(\xi_1,\xi_2)(\xi_1,\xi_2)}\tilde{\varphi}_{\epsilon_0}(t^*,\xi^*_1,\xi^*_2)\left((\boldsymbol{e}_j,\boldsymbol{e}_j),(\boldsymbol{e}_j,\boldsymbol{e}_j)\right)\big] 
+o(h).
\end{aligned}
\end{equation}
Combining \eqref{tildevarphi derivatives}-\eqref{t34}, we derive
\begin{equation*}
    \begin{aligned}
    &-h\mathbb{E}\big[\partial_xH(\xi^*_1,-p^*_2-\lambda(\xi^*_1-\xi^*_2),\rho^*)\cdot(\xi^*_1-\xi^*_2)+p^*_2\cdot\partial_pH(\xi^*_1,-p^*_2-\lambda(\xi^*_1-\xi^*_2),\rho^*)\big] \\
\geq&h\mathbb{E}\big[\partial_t\varphi(t^*,\xi^*_1,\xi^*_2)-\epsilon_0+
p^*_{1}\cdot \partial_pH(\xi^*_1,-p^*_2-\lambda(\xi^*_1-\xi^*_2),\rho^*)\\
&\qquad+\frac{\sigma_0^2}{2}\sum_{j=1}^d(\chi-\epsilon_0I)\left((\boldsymbol{e}_j,\boldsymbol{e}_j),(\boldsymbol{e}_j,\boldsymbol{e}_j)\right)\big]
+o(h),
\end{aligned}
\end{equation*}
where
$\rho^*=\rho^{\xi_1^*}_{t^*}=\Phi\Big(\mathcal{L}\big(\xi^*_1,-p^*_2-\lambda(\xi^*_1-\xi^*_2)\big)\Big).$
Dividing both sides by $h$ and letting $h\to0$, and then letting $\epsilon_0\rightarrow 0$, we have
\begin{equation*}
    \begin{aligned}
    &-\mathbb{E}\big[\partial_xH(\xi^*_1,-p^*_2-\lambda(\xi^*_1-\xi^*_2),\rho^*)\cdot(\xi^*_1-\xi^*_2)+p^*_2\cdot\partial_pH(\xi^*_1,-p^*_2-\lambda(\xi^*_1-\xi^*_2),\rho^*)\big] \\
\geq&\mathbb{E}\big[\partial_t\varphi(t^*,\xi^*_1,\xi^*_2)+
p^*_{1}\cdot \partial_pH(\xi^*_1,-p^*_2-\lambda(\xi^*_1-\xi^*_2),\rho^*)+\frac{\sigma_0^2}{2}\sum_{j=1}^d\chi\left((\boldsymbol{e}_j,\boldsymbol{e}_j),(\boldsymbol{e}_j,\boldsymbol{e}_j)\right)\big].
\end{aligned}
\end{equation*}
Therefore, 
$U$ satisfies the definition of displacement
$\lambda$-monotone solution.

\noindent\textbf{Step 1.2.} Conversely, we show that any displacement $\lambda$-monotone solution $U$ of the vectorial master equation \eqref{vectorial master equation} is indeed a vectorial weak solution to the master equation \eqref{master equation}. For any $(t_0, x, \mu_{t_0})\in [0,T]\times \mathbb{R}^d\times\mathcal{P}_2(\mathbb{R}^d)$,  
define $\widetilde{U}(t_0, x, \mu_{t_0}) = \nabla Y_{t_0}^{\xi, x}$, $\mathcal L(\xi)=\mu_{t_0},$ where $\nabla Y^{\xi, x}$ is the solution to the FBSDEs \eqref{FBSDE xi}-\eqref{FBSDE xi,x}. By the existence result for vectorial weak solutions established in Theorem \ref{weak solution theorem}, $\widetilde U$ is a vectorial weak solution to the master equation \eqref{master equation}, with
$\mu_t=\mathcal L^0(X_t^\xi)$ being a weak solution to the SPDE \eqref{vectorial SPDE}. Hence, by Step 1.1, $\widetilde U$ is also a displacement $\lambda$-monotone solution to the vectorial master equation \eqref{vectorial master equation}. 
\par For  $(t,\xi_1,\xi_2)\in[0,T]\times\mathbb{H}$, define
\[
\begin{aligned}
\widetilde{W}(t,\xi_1,\xi_2)
:=&\mathbb{E}\big[\widetilde{U}(t,\xi_1,\mu_1)\cdot(\xi_1-\xi_2)
+\frac{\lambda}{2}|\xi_1-\xi_2|^2\big],\\
\end{aligned}
\]
where $\mu_1=\mathcal{L}(\xi_1)$. 
Our goal is to prove that $W(t,\xi_1,\xi_2)+\widetilde{W}(t,\xi_2,\xi_1)\ge0$, for any $(t,\xi_1,\xi_2)$. 
By Lemma~\ref{lemma:uniqueness_monotone}, this yields $U=\widetilde{U}$. Since 
$\widetilde{U}$ is a vectorial weak solution to the master equation \eqref{master equation}, it follows that $U$ is also a vectorial weak solution to \eqref{master equation}. 

We argue by contradiction. Suppose that there exists
$(\hat t,\hat \xi_1,\hat \xi_2)\in[0,T]\times \mathbb{H}$ such that
\[
W(\hat t,\hat \xi_1,\hat \xi_2)+\tilde W(\hat t,\hat \xi_2,\hat \xi_1)<0.
\]
At \(t=T\), the terminal displacement
\(\lambda\)-monotonicity \eqref{G displacement} gives
\begin{equation}\label{terminal displacement monotonicity}
\begin{aligned}
W(T,\xi_1,\xi_2)+\widetilde W(T,\xi_2,\xi_1)
=
 \mathbb{E}\left[
   \left(
     \partial_xG(\xi_1,\mathcal{L}(\xi_1))
     -\partial_xG(\xi_2,\mathcal{L}(\xi_2))
   \right)\cdot(\xi_1-\xi_2)
   +\lambda|\xi_1-\xi_2|^2
 \right]
 \geq0.
\end{aligned}
\end{equation}
Therefore, \(\hat t<T\). 
Set
\[
    F(t):=
    W(t,\widehat\xi_1,\widehat\xi_2)
    +\widetilde W(t,\widehat\xi_2,\widehat\xi_1).
\]
We have \(F(\hat t)<0\). If \(\hat t>0\), set
\(t_*= \hat t\). If \(\hat t=0\), then, by the
continuity of \(F\), there exists \(t_*\in(0,T)\), sufficiently
close to \(0\), such that
   $ F(t_*)<0.$
For $\gamma>0$, define the penalized functional
\[
W^\gamma(t,\xi_1,\xi_2)
:=W(t,\xi_1,\xi_2)+\gamma(T-t)+\gamma(\frac1t-\frac1T).
\]
Notice that $W^{\gamma}(T,\xi_1,\xi_2)=W(T,\xi_1,\xi_2)$, $W^{\gamma}\geq W$ and $\lim_{t\rightarrow 0^{+}}W^{\gamma}(t,\xi_1,\xi_2)=+\infty$. Since \(t_*\in(0,T)\) and \(F(t_*)<0\), we may choose
\(\gamma>0\) sufficiently small so that
\[
M^{\gamma}:=\inf_{(t,\xi_1,\xi_2)\in(0,T]\times\mathbb{H}} W^\gamma(t,\xi_1,\xi_2)+\widetilde W(t,\xi_2,\xi_1)\leq W^\gamma(t_*,\hat{\xi}_1,\hat{\xi}_2)+\tilde W(t_*,\hat{\xi}_2,\hat{\xi}_1)<0.
\]
A universal constant \(K_0>0\) will be specified below. Fix
\(K>K_0\), and, for \(\varepsilon>0\), define
\[
    P_\varepsilon(t,\xi_1,\xi_2)
    :=
    \varepsilon e^{K(T-t)}
    \left(1+\|\xi_1\|_2^2+\|\xi_2\|_2^2\right)^2.
\]
Writing $R:=1+\|\xi_1\|_2^2+\|\xi_2\|_2^2$,
we have
\[
    \partial_tP_\varepsilon(t,\xi_1,\xi_2)=-KP_\varepsilon(t,\xi_1,\xi_2),\ 
    D_{(\xi_1,\xi_2)}P_\varepsilon(t,\xi_1,\xi_2)
    =4\varepsilon e^{K(T-t)}R(\xi_1,\xi_2),
\]
and for any $h,k\in \mathbb{H}$,
\begin{equation}\label{P seceond derivative}
D_{(\xi_1,\xi_2)(\xi_1,\xi_2)}P_\varepsilon(t,\xi_1,\xi_2)(h,k)
    =
    4\varepsilon e^{K(T-t)}
    \bigl(R\langle h,k\rangle_{\mathbb{H}}+2\langle (\xi_1,\xi_2),h\rangle_{\mathbb{H}}\langle (\xi_1,\xi_2),k\rangle_{\mathbb{H}}\bigr).    
\end{equation}
Consequently,
\begin{equation}\label{eq:step12-P-derivatives}
\begin{aligned}
    \|D_{(\xi_1,\xi_2)}P_\varepsilon\|_2
        \leq C\varepsilon e^{K(T-t)}R^{3/2},\
    \|D_{(\xi_1,\xi_2)(\xi_1,\xi_2)}P_\varepsilon\|
        \leq C\varepsilon e^{K(T-t)}R.
\end{aligned}
\end{equation}
After \(K\) has been fixed, choose \(\varepsilon>0\) sufficiently
small so that
\[
\begin{aligned}
    &M^{\gamma}_{\varepsilon}:=\inf_{(t,\xi_1,\xi_2)} W^\gamma(t,\xi_1,\xi_2)+\widetilde W(t,\xi_2,\xi_1)+P_{\varepsilon}(t,\xi_1,\xi_2)+P_{\varepsilon}(t,\xi_2,\xi_1)
      <0.
\end{aligned}
\]
For $\alpha>0$, define
\[
\begin{aligned}
\Psi_{\alpha,\varepsilon}^{\gamma}(t,t',\xi_1,\xi_2,\xi'_1,\xi'_2)
:=&\, W^{\gamma}(t,\xi_1,\xi_2)+\widetilde{W}(t',\xi'_2,\xi'_1) +P_{\varepsilon}(t,\xi_1,\xi_2)+P_{\varepsilon}(t',\xi'_2,\xi_1') \\
&+\frac{1}{2\alpha}\mathbb{E}\big[|\xi_1-\xi'_1|^2+|\xi_2-\xi'_2|^2\big]
+\frac{1}{2\alpha}(t-t')^2 .
\end{aligned}
\]
Let $M_{\varepsilon,\alpha}^{\gamma} = \inf_{t,t',\xi_1,\xi_2,\xi'_1,\xi'_2} \Psi_{\alpha}^{\gamma}(t,t',\xi_1,\xi_2,\xi'_1,\xi'_2)$. The functions \(W^\gamma\) and \(\widetilde W\) have at most
quadratic growth, whereas \(P_\varepsilon\) has quartic growth.
Thus $\Psi_{\alpha}^{\gamma}$ is lower semicontinuous and coercive. By Stegall's Lemma (see, e.g., \cite{fabian2007stegall}), for any $\delta > 0$, there exist perturbation elements 
$(\zeta_1,\zeta_2),(\zeta'_1,\zeta'_2) \in \mathbb{H}$
and $s,s'\in\mathbb{R}$ such that
\begin{equation}\label{Stegall perturbation elements}
\|(\zeta_1,\zeta_2)\|_2+\|(\zeta'_1,\zeta'_2)\|_2+|s|+|s'|\le\delta, 
\end{equation}
and the map
$$(t, t', \xi_1, \xi_2, \xi'_1, \xi'_2) \mapsto \Psi_{\alpha}^{\gamma}- \mathbb{E}[\zeta_1 \cdot \xi_1 + \zeta_2 \cdot \xi_2 + \zeta'_1 \cdot \xi'_1 + \zeta'_2 \cdot \xi'_2] - t s - t' s'$$
attains its strict minimum $M^{\gamma}_{\varepsilon,\alpha,\delta}$ at
$\Xi_\delta:=(t_\delta,t'_\delta,\xi_{1,\delta},\xi_{2,\delta},\xi'_{1,\delta},\xi'_{2,\delta})$. 
Define
$$
\begin{aligned}  
&R_\delta := 1+\|\xi_{1,\delta}\|_2^2+\|\xi_{2,\delta}\|_2^2, \ R'_\delta := 1+\|\xi'_{1,\delta}\|_2^2+\|\xi'_{2,\delta}\|_2^2,\\
&P_{\varepsilon,\delta}:= P_{\varepsilon}(t_{\delta},\xi_{1,\delta},\xi_{2,\delta})=\varepsilon e^{K(T-t_{\delta})}R_{\delta}^2,\ P'_{\varepsilon,\delta}:=P_{\varepsilon}(t'_{\delta},\xi'_{2,\delta},\xi'_{1,\delta})=\varepsilon e^{K(T-t'_{\delta})}(R_{\delta}')^2.
\end{aligned}$$
Comparing the minimum with the fixed point $(t_*,t_*,
      \widehat\xi_1,\widehat\xi_2,
      \widehat\xi_1,\widehat\xi_2)$
and using the quadratic growth of \(W^\gamma\) and
\(\widetilde W\), we obtain
\begin{equation}\label{eq:step12-localization}
    P_{\varepsilon,\delta}+P'_{\varepsilon,\delta}
    \leq C(1+R_\delta+R'_\delta),
    \
    \varepsilon \left( R_\delta^2+(R'_\delta)^2 \right) \leq C(1+R_\delta+R'_\delta).  
\end{equation}
Here and below, \(C\) is independent of \(\alpha\) and
\(\delta\). Hence, for fixed \(K\) and \(\varepsilon\),
\begin{equation}
    R_{\delta}+R'_{\delta}\leq C_{\varepsilon(K)}.
    \label{eq:step12-boundedness}
\end{equation}
Note that, for fixed $\alpha>0$, $M^{\gamma}_{\varepsilon,\alpha,\delta} \to M^{\gamma}_{\varepsilon,\alpha}$ as $\delta \to 0$. Since $M^{\gamma}_{\varepsilon,\alpha}\rightarrow M^{\gamma}_{\varepsilon}$ as $\alpha \to 0$,
\begin{equation}\label{t14}
\begin{aligned}
&\lim_{\alpha\rightarrow0}\limsup_{\delta\rightarrow  0}\Big(\frac{1}{\alpha}\mathbb{E}\left[\vert\xi_{1,\delta}-\xi'_{1,\delta}\vert^2+\vert\xi_{2,\delta}-\xi'_{2,\delta}\vert^2\right]+\frac{1}{\alpha}\vert t_{\delta}-t_{\delta}'\vert^2\Big)=0.
\end{aligned}
\end{equation}
The singular term in $W^\gamma$ keeps $t_{\delta}$ away from \(0\), and \eqref{t14} then ensures the same for \(t'_{\delta}\). Moreover, the terminal displacement monotonicity \eqref{terminal displacement monotonicity}, \eqref{U terminal uniform continuity}, and \eqref{t14} imply that both \(t_{\delta}\) and \(t'_{\delta}\) stay away from \(T\).
Set
$$ 
\begin{aligned}  
&(r_{1,\delta},r_{2,\delta}):=D_{(\xi_1,\xi_2)}P_\varepsilon(t_\delta,\xi_{1,\delta},\xi_{2,\delta}), \
(r'_{2,\delta},r'_{1,\delta}):=D_{(\xi'_2,\xi'_1)}P_\varepsilon(t'_\delta,\xi'_{2,\delta},\xi'_{1,\delta}) ,\\
&S_{\varepsilon,\delta}:=D_{(\xi_1,\xi_2)(\xi_1,\xi_2)}P_\varepsilon(t_\delta,\xi_{1,\delta},\xi_{2,\delta}),\  S'_{\varepsilon,\delta}:=D_{(\xi'_2,\xi'_1)(\xi'_2,\xi'_1)}P_\varepsilon(t'_\delta,\xi'_{2,\delta},\xi'_{1,\delta}),
\end{aligned}
$$
and let
$$r_{\delta}:=\|r_{1,\delta}\|_2+\|r_{2,\delta}\|_2+\|r'_{1,\delta}\|_2+\|r'_{2,\delta}\|_2.$$
By \eqref{eq:step12-P-derivatives},
\begin{equation}\label{r estimate}
    \begin{aligned}
r_{\delta}\leq C\varepsilon\big(e^{K(T-t_{\delta})}R_{\delta}^{3/2}+e^{K(T-t'_{\delta})}(R'_{\delta})^{3/2}\big).
    \end{aligned}
\end{equation}
Moreover, by \eqref{P seceond derivative} and $R_{\delta},R_{\delta}\geq 1$,
\begin{equation}\label{S estimate}
    \begin{aligned}
\big(S_{\varepsilon,\delta}+S'_{\varepsilon,\delta}\big)\left((\boldsymbol{e}_j,\boldsymbol{e}_j),(\boldsymbol{e}_j,\boldsymbol{e}_j)\right)
\leq&24\varepsilon\big(e^{K(T-t_{\delta})}R_{\delta}+e^{K(T-t'_{\delta})}R'_{\delta}\big)\\
\leq&24\varepsilon\big(e^{K(T-t_{\delta})}R_{\delta}^{2}+e^{K(T-t'_{\delta})}(R'_{\delta})^{2}\big)\\
\leq& C(P_{\varepsilon,\delta}+P'_{\varepsilon,\delta}),
    \end{aligned}
\end{equation}
where in the first inequality, we used $$\big\vert\langle(\xi_{1,\delta},\xi_{2,\delta}),(\boldsymbol{e}_j,\boldsymbol{e}_j)\rangle_{\mathbb H}\big\vert^2\leq\|(\xi_{1,\delta},\xi_{2,\delta})\|^2_2\|(\boldsymbol{e}_j,\boldsymbol{e}_j)\|^2_2\leq 2(R_{\delta}-1)\leq 2 R_{\delta}.$$
By \eqref{t14}, after taking \(\alpha,\delta\)
sufficiently small, we may assume
\begin{equation}\label{max a bouned}
\max\left\{e^{K(t'_{\delta}-t_{\delta})},e^{K(t_{\delta}-t'_{\delta})}\right\}
    =e^{K|t_{\delta}-t'_{\delta}|}\leq2.    
\end{equation}
Using \eqref{r estimate}, $R_{\delta},R'_{\delta}\geq 1$, and Young's inequality,
\begin{equation}\label{rdelta Rdelta estimate}
\begin{aligned}
r_{\delta}\big(\sqrt{R_{\delta}}+\sqrt{R'_{\delta}}\big)\leq &C\varepsilon\big(e^{K(T-t_{\delta})}R_{\delta}^{3/2}+e^{K(T-t'_{\delta})}(R'_{\delta})^{3/2}\big)\left(\sqrt{R_{\delta}}+\sqrt{R'_{\delta}}\right)\\
    \leq& C(P_{\varepsilon,\delta}+P'_{\varepsilon,\delta}).
    \end{aligned}
\end{equation}
Indeed, the mixed terms are controlled by
\begin{equation*}
    \begin{aligned}
    e^{K(T-t_{\delta})}R^{3/2}_{\delta}(R'_{\delta})^{1/2}
       \leq C\big(e^{K(T-t_{\delta})}R^2_{\delta}+e^{K(T-t'_{\delta})}(R'_{\delta})^2\big),
    \\
    e^{K(T-t'_{\delta})}(R'_{\delta})^{3/2}R_{\delta}^{1/2}
       \leq C\big(e^{K(T-t'_{\delta})}(R'_{\delta})^2+e^{K(T-t_{\delta})}R_{\delta}^2\big).    
    \end{aligned}
\end{equation*}
We also have
\[
    r^2_{\delta}
    \leq
    C\varepsilon^2
    \left(e^{2K(T-t_{\delta})}R_{\delta}^{3}+e^{2K(T-t'_{\delta})}(R'_{\delta})^{3}\right).
\]
Let \(a_*:=\min\{e^{K(T-t_{\delta})},e^{K(T-t'_{\delta})}\}\). From
\eqref{eq:step12-localization},
\[
    \frac{\varepsilon a_*}{2}(R_{\delta}+R'_{\delta})^2
    \leq P_{\varepsilon,\delta}+P'_{\varepsilon,\delta}
    \leq C(1+R_{\delta}+R_{\delta}').
\]
Since \(R_{\delta}+R_{\delta}'\geq2\) then $1+R_{\delta}+R_{\delta}'\leq\frac{3}{2}(R_{\delta}+R_{\delta}')$, it follows
that
\[
    \varepsilon a_*(R_{\delta}+R'_{\delta})\leq C.
\]
By \eqref{max a bouned}, we have
\[
    \varepsilon e^{K(T-t_{\delta})}R_{\delta}+\varepsilon e^{K(T-t'_{\delta})}R'_{\delta}\leq C.
\]
Therefore,
\begin{equation}\label{rdelta2 estimate}
r^2_{\delta}
    \leq
    C\bigl[
       (\varepsilon e^{K(T-t_{\delta})}R_{\delta})P_{\varepsilon,\delta}
       +(\varepsilon e^{K(T-t'_{\delta})}R'_{\delta})P'_{\varepsilon,\delta}
     \bigr]
    \leq C(P_{\varepsilon,\delta}+P'_{\varepsilon,\delta}).
\end{equation}
Since we eventually let \(\delta\rightarrow0\), we may assume that
\(0<\delta\leq1\). By the estimate \eqref{r estimate}, and using
\(R_\delta,R'_\delta\geq1\), we obtain
\begin{equation}\label{delta rdelta estimate}
\begin{aligned}
\delta r_\delta
&\leq
C\varepsilon\delta
\left(
e^{K(T-t_\delta)}R_\delta^{3/2}
+
e^{K(T-t'_\delta)}(R'_\delta)^{3/2}
\right)\\
&\leq
C\varepsilon
\left(
e^{K(T-t_\delta)}R_\delta^2
+
e^{K(T-t'_\delta)}(R'_\delta)^2
\right)\\
&=
C\bigl(
P_{\varepsilon,\delta}
+
P'_{\varepsilon,\delta}
\bigr).
\end{aligned}
\end{equation}
Similarly, since \(0<\alpha\leq1\), we have
\begin{equation}\label{alpha rdelta}
\begin{aligned}
\alpha r_\delta
&\leq
C\bigl(
P_{\varepsilon,\delta}
+
P'_{\varepsilon,\delta}
\bigr).
\end{aligned}
\end{equation}
According to Lemma \ref{lem1}, for all $N \geq 1$, there exist operators $\mathcal{X}_N$ and $\mathcal{Y}_N$ such that $\mathcal{X}_N = P_N \mathcal{X}_N P_N$ and $\mathcal{Y}_N = P_N \mathcal{Y}_N P_N$. Furthermore, these operators satisfy the matrix inequality \eqref{matrixinequality} along with the following relations
\begin{equation*}\label{second1}
\begin{aligned}
\Big( s - \frac{1}{\alpha}(t_{\delta}-t'_{\delta})+KP_{\varepsilon,\delta}, \
& -\frac{1}{\alpha}\big( (\xi_{1,\delta},\xi_{2,\delta}) - (\xi'_{1,\delta},\xi'_{2,\delta}) \big) - (r_{1,\delta},r_{2,\delta}) + (\zeta_1,\zeta_2), \\
& \mathcal{X}_N - \frac{2}{\alpha}Q_N - S_{\varepsilon,\delta} \Big) \in \overline{D}^{1,2,-}W^{\gamma}(t_\delta,\xi_{1,\delta},\xi_{2,\delta}),
\end{aligned}
\end{equation*}
and
\begin{equation*}\label{second2}
\begin{aligned}
\Big( s' + \frac{1}{\alpha}(t_{\delta}-t'_{\delta})+KP'_{\varepsilon,\delta}, \
& \frac{1}{\alpha}\big( (\xi_{2,\delta},\xi_{1,\delta}) - (\xi'_{2,\delta},\xi'_{1,\delta}) \big)-(r'_{2,\delta},r'_{1,\delta})+ (\zeta'_2,\zeta'_1), \\
& \mathcal{Y}_N - \frac{2}{\alpha}Q_N - S'_{\varepsilon,\delta}\Big) \in \overline{D}^{1,2,-}\widetilde{W}(t'_{\delta},\xi'_{2,\delta},\xi'_{1,\delta}).
\end{aligned}
\end{equation*}
Moreover, there exist $(t_k,\xi_{1,k},\xi_{2,k},q_k,p_k,\chi_k)$ and $(t'_k,\xi'_{1,k},\xi'_{2,k},q'_k,p'_k,\chi'_k)$  satisfying
\[
(q_k,p_k,\chi_k)\in D^{1,2,-}W^{\gamma}
    (t_k,\xi_{1,k},\xi_{2,k}),
\qquad
(q'_k,p'_k,\chi'_k)\in D^{1,2,-}\widetilde{W}
    (t'_k,\xi'_{2,k},\xi'_{1,k}),
\]
respectively, such that, as $k\to\infty$,
\begin{equation}\label{subdifferential limit}
\begin{aligned}
&(t_k,t'_k,\xi_{1,k},\xi_{2,k},\xi'_{1,k},\xi'_{2,k})\rightarrow \Xi_{\delta},\\ 
&W^{\gamma}(t_k,\xi_{1,k},\xi_{2,k})\to W^{\gamma}(t_{\delta},\xi_{1,\delta},\xi_{2,\delta}),\
\widetilde{W}(t'_k,\xi'_{2,k},\xi'_{1,k})\to \tilde{W}(t'_{\delta},\xi'_{2,\delta},\xi'_{1,\delta}),\\
&q_k\to  s-\frac{1}{\alpha}(t_{\delta}-t'_{\delta})+KP_{\varepsilon,\delta},\
q'_k\to s'+ \frac{1}{\alpha}(t_{\delta}-t'_{\delta})+KP'_{\varepsilon,\delta},\\
& p_k:=(p_{1,k},p_{2,k})\to-\frac{\left(\xi_{1,\delta},\xi_{2,\delta}\right)-(\xi'_{1,\delta},\xi'_{2,\delta})}{\alpha}-(r_{1,\delta},r_{2,\delta})+(\zeta_1,\zeta_2),\\
&p'_k:=(p'_{2,k},p'_{1,k})\to \frac{\left(\xi_{2,\delta},\xi_{1,\delta}\right)-(\xi'_{2,\delta},\xi'_{1,\delta})}{\alpha}-(r'_{2,\delta},r'_{1,\delta})+(\zeta_2^{\prime},\zeta_1^{\prime}),\\
&\chi_k\to \mathcal{X}_N- \frac{2}{\alpha}Q_N - S_{\varepsilon,\delta},\ \chi'_k\to \mathcal{Y}_N- \frac{2}{\alpha}Q_N - S'_{\varepsilon,\delta}. 
\end{aligned} 
\end{equation}
By Lemma \ref{relationship} and the definition of the displacement $\lambda$-monotone solution, we obtain
\begin{equation}\label{t11}
\begin{aligned} 
    \gamma\leq\gamma+\frac{\gamma}{t_k^2} \leq &-q_k-\frac{\sigma_0^2}{2}\sum_{j=1}^{d}\chi_{k}\left((\boldsymbol{e}_j,\boldsymbol{e}_j),(\boldsymbol{e}_j,\boldsymbol{e}_j)\right)\\
&-\mathbb{E}\Big[\partial_xH\left(\xi_{1,k},-p_{2,k}-\lambda(\xi_{1,k}-\xi_{2,k}),\rho_{k}\right)\cdot(\xi_{1,k}-\xi_{2,k})\Big]\\
& -\mathbb{E}\Big[\big(p_{1,k}+p_{2,k}\big)\cdot\partial_pH\left(\xi_{1,k},-p_{2,k}-\lambda(\xi_{1,k}-\xi_{2,k}),\rho_{k}\right)\Big],
\end{aligned}
\end{equation}
and
\begin{equation}\label{t12}
\begin{aligned} 
    0 \leq &-q'_k
    -\frac{\sigma_0^2}{2}\sum_{j=1}^{d}\chi'_{k}\left((\boldsymbol{e}_j,\boldsymbol{e}_j),(\boldsymbol{e}_j,\boldsymbol{e}_j)\right)\\
&-\mathbb{E}\Big[\partial_xH\left(\xi'_{2,k},-p'_{1,k}-\lambda(\xi'_{2,k}-\xi'_{1,k}),\rho_k'\right)\cdot(\xi'_{2,k}-\xi'_{1,k})\Big]\\
& -\mathbb{E}\Big[\big(p'_{1,k}+p'_{2,k}\big)\cdot\partial_pH\left(\xi'_{2,k},-p'_{1,k}-\lambda(\xi'_{2,k}-\xi'_{1,k}),\rho_k'\right)\Big],
\end{aligned}
\end{equation}
where 
$\rho_{k}=\Phi\Big(\mathcal{L}\big(\xi_{1,k},-p_{2,k}-\lambda(\xi_{1,k}-\xi_{2,k})\big)\Big),\ 
\rho'_{k}=\Phi\Big(\mathcal{L}\big(\xi'_{2,k},-p'_{1,k}-\lambda(\xi'_{2,k}-\xi'_{1,k})\big)\Big)$.
Adding \eqref{t11} and \eqref{t12}, then letting $k\to\infty$, we obtain, by the Lipschitz continuity of $\partial_x H$, $\partial_p H$ and $\Phi$, together with the convergence in
\eqref{subdifferential limit},
\begin{equation*}
\begin{aligned}
\gamma \leq &
-s-s'-K(P_{\varepsilon,\delta}+P'_{\varepsilon,\delta})\\
&-\frac{\sigma_0^2}{2}\sum_{j=1}^{d}\big(\mathcal{X}_N+\mathcal{Y}_N-\frac{4}{\alpha}Q_N-S_{\varepsilon,\delta}-S'_{\varepsilon,\delta}\big)\left((\boldsymbol{e}_j,\boldsymbol{e}_j),(\boldsymbol{e}_j,\boldsymbol{e}_j)\right) \\
&-\mathbb{E}\big[
    \partial_x H(\xi_{1,\delta},\vartheta_\delta,\rho_\delta)
    \cdot(\xi_{1,\delta}-\xi_{2,\delta})
\big] -\mathbb{E}\big[
    \partial_x H(\xi'_{2,\delta},\vartheta'_\delta,\rho'_\delta)
    \cdot(\xi'_{2,\delta}-\xi'_{1,\delta})
\big] \\
&-\mathbb{E}\Bigg[
    \left(
        -\frac{1}{\alpha}(\xi_{1,\delta}-\xi'_{1,\delta})
        -\frac{1}{\alpha}(\xi_{2,\delta}-\xi'_{2,\delta})
        -r_{1,\delta}-r_{2,\delta}
        +\zeta_1+\zeta_2
    \right)
    \cdot
    \partial_p H(\xi_{1,\delta},\vartheta_\delta,\rho_\delta)
\Bigg] \\
&-\mathbb{E}\Bigg[
    \bigg(
        \frac{1}{\alpha}(\xi_{1,\delta}-\xi'_{1,\delta})
        +\frac{1}{\alpha}(\xi_{2,\delta}-\xi'_{2,\delta})
        -r'_{2,\delta}-r'_{1,\delta}
        +\zeta'_1+\zeta'_2
    \bigg)
    \cdot
    \partial_p H(\xi'_{2,\delta},\vartheta'_\delta,\rho'_\delta)
\Bigg],
\end{aligned}
\end{equation*}
where
\begin{equation*}
\begin{aligned}
\vartheta_\delta
&:=
\frac{1}{\alpha}(\xi_{2,\delta}-\xi'_{2,\delta})+r_{2,\delta}
-\zeta_2
-\lambda(\xi_{1,\delta}-\xi_{2,\delta}), \ 
\vartheta'_\delta
:=
-\frac{1}{\alpha}(\xi_{1,\delta}-\xi'_{1,\delta})+r'_{1,\delta}
-\zeta'_1
-\lambda(\xi'_{2,\delta}-\xi'_{1,\delta}),
\end{aligned}
\end{equation*}
and
\begin{equation*}
\begin{aligned}
\rho_\delta
:=
\Phi\big(
    \mathcal{L}(\xi_{1,\delta},\vartheta_\delta)
\big), \ 
\rho'_\delta
:=
\Phi\big(
    \mathcal{L}(\xi'_{2,\delta},\vartheta'_\delta)
\big).
\end{aligned}
\end{equation*}
By \eqref{matrixinequality}, we have
\begin{equation*}
\sum_{j=1}^{d}\big(\mathcal{X}_N+\mathcal{Y}_N\big)\left((\boldsymbol{e}_j,\boldsymbol{e}_j),(\boldsymbol{e}_j,\boldsymbol{e}_j)\right) \geq 0,  
\end{equation*}
together with \eqref{Q_N equation}, \eqref{S estimate}, we have
\begin{equation}\label{inequality}
\begin{aligned}
\gamma \leq &
-s-s'-(K-C)(P_{\varepsilon,\delta}+P'_{\varepsilon,\delta})\\
&-\mathbb{E}\big[
    \partial_x H(\xi_{1,\delta},\vartheta_\delta,\rho_\delta)
    \cdot(\xi_{1,\delta}-\xi_{2,\delta})
\big] -\mathbb{E}\big[
    \partial_x H(\xi'_{2,\delta},\vartheta'_\delta,\rho'_\delta)
    \cdot(\xi'_{2,\delta}-\xi'_{1,\delta})
\big] \\
&-\mathbb{E}\Bigg[
    \left(
        -\frac{1}{\alpha}(\xi_{1,\delta}-\xi'_{1,\delta})
        -\frac{1}{\alpha}(\xi_{2,\delta}-\xi'_{2,\delta})
        -r_{1,\delta}-r_{2,\delta}
        +\zeta_1+\zeta_2
    \right)
    \cdot
    \partial_p H(\xi_{1,\delta},\vartheta_\delta,\rho_\delta)
\Bigg] \\
&-\mathbb{E}\Bigg[
    \bigg(
        \frac{1}{\alpha}(\xi_{1,\delta}-\xi'_{1,\delta})
        +\frac{1}{\alpha}(\xi_{2,\delta}-\xi'_{2,\delta})
         -r'_{2,\delta}-r'_{1,\delta}
        +\zeta'_1+\zeta'_2
    \bigg)
    \cdot
    \partial_p H(\xi'_{2,\delta},\vartheta'_\delta,\rho'_\delta)
\Bigg].
\end{aligned}
\end{equation}
Combining the above terms, we obtain
\begin{equation*}
\begin{aligned}
\gamma \leq &
-s-s'-(K-C)(P_{\varepsilon,\delta}+P'_{\varepsilon,\delta})\\
&-\mathbb{E}\big[
    \big(\partial_x H(\xi_{1,\delta},\vartheta_\delta,\rho_\delta)-\partial_xH(\xi'_{2,\delta},\vartheta'_\delta,\rho'_\delta)
    \big)\cdot(\xi_{1,\delta}-\xi'_{2,\delta})
\big]\\
&+\mathbb{E}\big[\big(\partial_p H(\xi_{1,\delta},\vartheta_\delta,\rho_\delta)
    -\partial_p H(\xi'_{2,\delta},\vartheta'_\delta,\rho'_\delta)\big)\cdot\big(\vartheta_{\delta}-\vartheta'_{\delta}\big)\big]\\
    &+2\lambda\E\big[\big(\partial_p H(\xi_{1,\delta},\vartheta_\delta,\rho_\delta)
    -\partial_p H(\xi'_{2,\delta},\vartheta'_\delta,\rho'_\delta)\big)\cdot\big(\xi_{1,\delta}-\xi'_{2,\delta}\big)\big]\\
&+\mathbb{E}\left[
\partial_xH(\xi_{1,\delta},\vartheta_\delta,\rho_{\delta})\cdot(\xi_{2,\delta}-\xi'_{2,\delta})\right]
+\mathbb{E}\left[
\partial_xH(\xi'_{2,\delta},\vartheta'_\delta,\rho'_{\delta})\cdot(\xi'_{1,\delta}-\xi_{1,\delta})\right]\\
&-\mathbb{E}\left[\big(\zeta_1+\zeta'_1-r_{1,\delta}-r'_{1,\delta}+\lambda(\xi_{2,\delta}-\xi'_{2,\delta})+\lambda(\xi_{1,\delta}-\xi'_{1,\delta})\big)\cdot
\partial_pH(\xi_{1,\delta},\vartheta_\delta,\rho_{\delta})\right]\\
&-\mathbb{E}\left[\big(\zeta_2+\zeta'_2\-r_{2,\delta}-r'_{2,\delta}+\lambda(\xi'_{2,\delta}-\xi_{2,\delta})+\lambda(\xi'_{1,\delta}-\xi_{1,\delta})\big)\cdot
\partial_pH(\xi'_{2,\delta},\vartheta'_\delta,\rho'_{\delta})\right].
\end{aligned}
\end{equation*}
Applying the displacement $\lambda$-monotonicity \eqref{H displacement}, we get
\begin{equation*}
\begin{aligned}
\gamma\leq& -s-s'-(K-C)(P_{\varepsilon,\delta}+P'_{\varepsilon,\delta})\\
&+\mathbb{E}\left[
\partial_xH(\xi_{1,\delta},\vartheta_\delta,\rho_{\delta})\cdot(\xi_{2,\delta}-\xi'_{2,\delta})\right]
+\mathbb{E}\left[
\partial_xH(\xi'_{2,\delta},\vartheta'_\delta,\rho'_{\delta})\cdot(\xi'_{1,\delta}-\xi_{1,\delta})\right]\\
&-\mathbb{E}\left[\big(\zeta_1+\zeta'_1-r_{1,\delta}-r'_{1,\delta}+\lambda(\xi_{2,\delta}-\xi'_{2,\delta})+\lambda(\xi_{1,\delta}-\xi'_{1,\delta})\big)\cdot
\partial_pH(\xi_{1,\delta},\vartheta_\delta,\rho_{\delta})\right]\\
&-\mathbb{E}\left[\big(\zeta_2+\zeta'_2-r_{2,\delta}-r'_{2,\delta}+\lambda(\xi'_{2,\delta}-\xi_{2,\delta})+\lambda(\xi'_{1,\delta}-\xi_{1,\delta})\big)\cdot
\partial_pH(\xi'_{2,\delta},\vartheta'_\delta,\rho'_{\delta})\right].
\end{aligned}
\end{equation*}
Combining \eqref{Stegall perturbation elements} with the Lipschitz continuity of $\partial_x H$,  $\partial_p H$ and $\Phi$,
it follows that
\begin{equation*}
\begin{aligned}
&\big\vert\mathbb{E}\left[
\partial_xH(\xi_{1,\delta},\vartheta_\delta,\rho_{\delta})\cdot(\xi_{2,\delta}-\xi'_{2,\delta})\right]\big\vert+\big\vert\mathbb{E}\left[
\partial_xH(\xi'_{2,\delta},\vartheta'_\delta,\rho'_{\delta})\cdot(\xi'_{1,\delta}-\xi_{1,\delta})\right]\big\vert\\
\leq& C\big(\|\xi_{2,\delta}-\xi'_{2,\delta}\|_2+\|\xi'_{1,\delta}-\xi_{1,\delta}\|_2\big)\big(1+\|\xi_{1,\delta}\|_2+\|\xi_{2,\delta}\|_2+\|\xi'_{1,\delta}\|_2+\|\xi'_{2,\delta}\|_2\\
&+\|r_{2,\delta}\|_2+\|r'_{1,\delta}\|_2+\frac{1}{\alpha}\|\xi_{2,\delta}-\xi'_{2,\delta}\|_2+\frac{1}{\alpha}\|\xi_{1,\delta}-\xi'_{1,\delta}\|_2+\delta\big),
\end{aligned}
\end{equation*}
and
\begin{equation*}
\begin{aligned}
&\big\vert\mathbb{E}\left[\big(\zeta_1+\zeta'_1-r_{1,\delta}-r'_{1,\delta}+\lambda(\xi_{2,\delta}-\xi'_{2,\delta})+\lambda(\xi_{1,\delta}-\xi'_{1,\delta})\big)\cdot
\partial_pH(\xi_{1,\delta},\vartheta_\delta,\rho_{\delta})\right]\big\vert\\
&+\big\vert\mathbb{E}\left[\big(\zeta_2+\zeta'_2-r_{2,\delta}-r'_{2,\delta}+\lambda(\xi'_{2,\delta}-\xi_{2,\delta})+\lambda(\xi'_{1,\delta}-\xi_{1,\delta})\big)\cdot
\partial_pH(\xi'_{2,\delta},\vartheta'_\delta,\rho'_{\delta})\right]\big\vert\\
\leq& C\big(\delta+\|r_{1,\delta}\|_2+\|r'_{1,\delta}\|_2+\|r_{2,\delta}\|_2+\|r'_{2,\delta}\|_2+\|\xi_{2,\delta}-\xi'_{2,\delta}\|_2+\|\xi'_{1,\delta}-\xi_{1,\delta}\|_2\big)\big(1+\|\xi_{1,\delta}\|_2+\|\xi_{2,\delta}\|_2\\
&+\|\xi'_{1,\delta}\|_2+\|\xi'_{2,\delta}\|_2+\frac{1}{\alpha}\|\xi_{2,\delta}-\xi'_{2,\delta}\|_2+\frac{1}{\alpha}\|\xi_{1,\delta}-\xi'_{1,\delta}\|_2+\|r_{2,\delta}\|_2+\|r'_{1,\delta}\|_2+\delta\big).
\end{aligned}
\end{equation*}
Hence, using $|s|+|s'|\leq\delta$, we have
\begin{equation}\label{gamma inequality}
    \begin{aligned}
\gamma\leq&C\delta-(K-C)(P_{\varepsilon,\delta}+P'_{\varepsilon,\delta})\\ &+C\big(\delta+r_{\delta}+\|\xi_{2,\delta}-\xi'_{2,\delta}\|_2+\|\xi'_{1,\delta}-\xi_{1,\delta}\|_2\big)\\
&\cdot\big(\sqrt{R_{\delta}}+\sqrt{R'_{\delta}}+\frac{1}{\alpha}\|\xi_{2,\delta}-\xi'_{2,\delta}\|_2+\frac{1}{\alpha}\|\xi_{1,\delta}-\xi'_{1,\delta}\|_2+r_{\delta}+\delta\big).     
    \end{aligned}
\end{equation}
Since \(U,\widetilde U\in
C_{Lip}^0(\Theta)\), one has
\[
|W(t,\xi_1,\xi_2)-W(t,\xi'_1,\xi'_2)|
\leq
C\bigl(1+\|\xi_1\|_2+\|\xi'_1\|_2+\|\xi_2\|_2+\|\xi'_2\|_2\bigr)
\big(\|\xi_1-\xi_1'\|_2+\|\xi_2-\xi_2'\|_2\big),
\]
The same estimate holds for \(\widetilde W\).
Applying Lemma~\ref{slope estimate} gives
\[
\begin{aligned}
 &\left\|
   -\frac{1}{\alpha}\big( (\xi_{1,\delta},\xi_{2,\delta}) - (\xi'_{1,\delta},\xi'_{2,\delta}) \big) - (r_{1,\delta},r_{2,\delta}) + (\zeta_1,\zeta_2)
 \right\|_2
 \leq C\sqrt R_{\delta},\\
 &\left\|
   \frac{1}{\alpha}\big( (\xi_{2,\delta},\xi_{1,\delta}) - (\xi'_{2,\delta},\xi'_{1,\delta}) \big)-(r'_{2,\delta},r'_{1,\delta})+ (\zeta'_2,\zeta'_1)
 \right\|_2
 \leq C\sqrt{R'_{\delta}}.
\end{aligned}
\]
Consequently, by the triangle inequality and
\eqref{Stegall perturbation elements},
\begin{equation}\label{D estimate}
\begin{aligned}
\frac{1}{\alpha}\|\xi_{1,\delta}-\xi'_{1,\delta}\|_2
+\frac{1}{\alpha}\|\xi_{2,\delta}-\xi'_{2,\delta}\|_2
 \leq C(\sqrt{R_{\delta}}+\sqrt{R'_{\delta}})+r_{\delta}+\delta.
\end{aligned}
\end{equation}
Recalling the estimates \eqref{rdelta Rdelta estimate}-\eqref{delta rdelta estimate},
\begin{equation*}
\begin{aligned}
r_{\delta}\Big(\frac{1}{\alpha}\|\xi_{1,\delta}-\xi'_{1,\delta}\|_2
+\frac{1}{\alpha}\|\xi_{2,\delta}-\xi'_{2,\delta}\|_2\Big)
&\leq r_{\delta}\Big(C(\sqrt{R_{\delta}}+\sqrt{R'_{\delta}})+r_{\delta}+\delta\Big)\leq C\big(P_{\varepsilon,\delta}+P'_{\varepsilon,\delta}\big).
\end{aligned} 
\end{equation*}
Moreover, for $i=1,2$, by \eqref{eq:step12-boundedness}
$$\begin{aligned}  
\|\xi_{i,\delta}-\xi'_{i,\delta}\|_2\big(\sqrt{R_{\delta}}+\sqrt{R'_{\delta}}\big)
\leq&\frac{C}{\alpha}\|\xi_{i,\delta}-\xi'_{i,\delta}\|_2^2+C\alpha\big(R_{\delta}+R'_{\delta}\big)\\
\leq&\frac{C}{\alpha}\|\xi_{i,\delta}-\xi'_{i,\delta}\|_2^2+C_{\varepsilon(K)}\alpha,
\end{aligned}
$$
while
$$\frac{\delta}{\alpha}\|\xi_{i,\delta}-\xi'_{i,\delta}\|_2\leq\frac{1}{2\alpha}\|\xi_{i,\delta}-\xi'_{i,\delta}\|_2^2+\frac{\delta^2}{2\alpha},\ \delta\|\xi_{i,\delta}-\xi'_{i,\delta}\|_2\leq\frac{1}{2\alpha}\|\xi_{i,\delta}-\xi'_{i,\delta}\|_2^2+\frac{\alpha\delta^2}{2}.$$
Therefore, combining with \eqref{rdelta Rdelta estimate}-\eqref{alpha rdelta}, inequality \eqref{gamma inequality} becomes
\begin{equation*}
\begin{aligned}
\gamma\leq&-(K-C_*)\big(P_{\varepsilon,\delta}+P'_{\varepsilon,\delta}\big)\\ &+C\Big(\delta^2+\frac{\delta^2}{2\alpha}+C_{\varepsilon(K)}\delta+ C_{\varepsilon(K)}\alpha+\frac{1}{\alpha}(\|\xi_{1,\delta}-\xi'_{1,\delta}\|_2^2+\|\xi_{2,\delta}-\xi'_{2,\delta}\|_2^2)\Big).
\end{aligned}
\end{equation*}
where \(C_*>0\) depends only on the constants in the assumptions
and is independent of \(K,\varepsilon,\alpha,\delta\).
We now choose $K_0:=C_*+1$
and recall that \(K>K_0\). Since \(P_{\varepsilon,\delta}+P'_{\varepsilon,\delta}\geq0\), we obtain
\[
    \gamma
    \leq C\Big(\delta^2+\frac{\delta^2}{2\alpha}+C_{\varepsilon(K)}\delta+ C_{\varepsilon(K)}\alpha+\frac{1}{\alpha}(\|\xi_{1,\delta}-\xi'_{1,\delta}\|_2^2+\|\xi_{2,\delta}-\xi'_{2,\delta}\|_2^2)\Big).
\]
By \eqref{t14}, letting first $\delta\rightarrow 0$ for fixed $\alpha>0$, and then letting $\alpha\rightarrow0$, we obtain $\gamma\leq 0$, which contradicts $\gamma>0$.
Hence $U=\widetilde{U}$.





(ii) We first show that every weak solution is a Lasry--Lions monotone solution. Let \(V\) be a weak solution to the master equation
\eqref{master equation}. Fix
$(t_0,x,\mu_{t_0})
\in [0,T]\times\mathbb R^d\times\mathcal P_2(\mathbb R^d)$,
and choose
\(\xi\in \mathbb L^2(\mathcal F_0^1;\mathbb R^d)\) such that
\(\mathcal L(\xi)=\mu_{t_0}\). Let
\((X^\xi,\nabla Y^\xi)\) and
\((X^{\xi,x},\nabla Y^{\xi,x})\) denote the corresponding solutions
to \eqref{FBSDE xi}-\eqref{FBSDE xi,x}, and set
$\mu_t:=\mathcal L^0(X_t^\xi)$.
The similar argument used in the uniqueness proof of
Theorem \ref{weak solution theorem} yields
$\partial_xV(t,X_t^{\xi,x},\mu_t)
=
\nabla Y_t^{\xi,x}$, $t\in[t_0,T]$.
Moreover,
\begin{equation}\label{Larsy-Lions monotone solution V}
\begin{aligned}
V(t_0,x,\mu_{t_0})
=
\mathbb E\Big[
G(X_T^{\xi,x},\mu_T)
+\int_{t_0}^T
\big(
H(X_s^{\xi,x},\nabla Y_s^{\xi,x},\rho_s)
-
\nabla Y_s^{\xi,x}\cdot
\partial_pH(X_s^{\xi,x},\nabla Y_s^{\xi,x},\rho_s)
\big)\mathrm ds
\Big].
\end{aligned}
\end{equation}
Theorem \ref{FBSDE well-posed} further implies that \(V\) satisfies
\eqref{V terminal uniform continuity}. Repeating the argument in
Step 1.1 of part (i), we conclude that \(V\) is a Lasry-Lions monotone solution to \eqref{master equation}.
\par Conversely, suppose that \(V\) is a Lasry-Lions monotone solution.
For
$
(t_0,x,\mu_{t_0})
\in[0,T]\times\mathbb R^d\times\mathcal P_2(\mathbb R^d)$,
$\mathcal L(\xi)=\mu_{t_0}$,
define \(\widetilde V(t_0,x,\mu_{t_0})\) by the right-hand side of
\eqref{Larsy-Lions monotone solution V}, where
\((X^\xi,\nabla Y^\xi)\) and
\((X^{\xi,x},\nabla Y^{\xi,x})\) solve
\eqref{FBSDE xi}-\eqref{FBSDE xi,x}. By
Theorem \ref{weak solution theorem}, \(\widetilde V\) is a weak
solution to the master equation \eqref{master equation} and satisfies
$\partial_x\widetilde V(t_0,x,\mu_{t_0})
=
\nabla Y_{t_0}^{\xi,x}$.
The implication established above then shows that \(\widetilde V\)
is also a Lasry-Lions monotone solution.
For \((t,\xi_1,\xi_2)\in[0,T]\times\mathbb H\), define
\[
\widetilde Q(t,\xi_1,\xi_2)
:=
\mathbb E\big[
\widetilde V(t,\xi_1,\mathcal L(\xi_1))
-
\widetilde V(t,\xi_2,\mathcal L(\xi_1))
\big].
\]
Set \(\mu_1:=\mathcal L(\xi_1)\), $\mu_2=\mathcal L(\xi_2)$. By the fundamental theorem of
calculus,
\[
Q(t,\xi_1,\xi_2)
=
\mathbb E\left[
\int_0^1
\partial_xV\big(
t,\xi_2+\theta(\xi_1-\xi_2),\mu_1
\big)\cdot(\xi_1-\xi_2)\,\mathrm d\theta
\right].
\]
The same representation holds for \(\widetilde Q\), with \(V\)
replaced by \(\widetilde V\).
Since \(V(T,\cdot,\cdot)=G\), 
the above representation and 
\eqref{V terminal uniform continuity} implies that, for any \(R>0\),
\[
\lim_{t\rightarrow T}
\sup_{\substack{M_2(\mu_1)\leq R\\ M_2(\mu_2)\leq R}}
\Big|Q(t,\xi_1,\xi_2)-\mathbb E\big[
G(\xi_1,\mu_1)
-
G(\xi_2,\mu_1)
\big]\Big|
=0.
\]
By \eqref{U holder continuous} for \(\widetilde V\) derived in
Theorem \ref{FBSDE well-posed}, an analogous estimate holds for \(\widetilde Q\).
We next claim that
\begin{equation}\label{monotonicity Q}
Q(t,\xi_1,\xi_2)
+
\widetilde Q(t,\xi_2,\xi_1)
\ge 0,
\qquad
(t,\xi_1,\xi_2)\in[0,T]\times\mathbb H.
\end{equation}
Indeed, the proof follows from the same doubling-of-variables
argument as in Step 1.2 of part (i). 
By Lemma \ref{lemma:uniqueness_monotone},
\eqref{monotonicity Q} implies that
$\partial_xV=\partial_x\widetilde V$.
Since \(\partial_x\widetilde V\) is a vectorial weak solution to
the master equation \eqref{master equation}, by
Theorem \ref{weak solution theorem}, the same is true of
\(\partial_xV\). This completes the proof of part (ii).
\end{proof}

\begin{remark}
We briefly explain the relation between the Lasry-Lions monotone solution used here and that in \cite{cardaliaguet2022monotone}. The main arguments are similar, but there are some differences between the two settings. In \cite{cardaliaguet2022monotone}, the solution is assumed to be Lipschitz continuous
and semiconcave in $x$. Moreover, 
the Hilbert space formulation involves random variables whose laws are 
absolutely continuous with bounded densities, together with an additional 
penalization term involving the $L^\infty$-norm of the density. 
In the present setting, we impose the stronger Assumption \ref{G L convexity assumption}, which yields the boundedness of $\partial_{xx}V$ and hence stronger spatial regularity of $V$. This stronger regularity allows us to impose weaker conditions on the random variables in the definition of Lasry-Lions monotone solutions: they can be arbitrary elements of $\mathbb H$, and no absolute continuity or bounded-density assumption is required on their laws. Moreover, the construction of monotone solutions is different from that in \cite{cardaliaguet2022monotone}. There, monotone solutions are obtained through the analysis of the associated MFG system, whereas in our framework they are constructed via the well-posedness of the FBSDE system. 
\end{remark}
\begin{remark}\label{uniqueness_gradient}
It is worth noting that while a weak solution to the master equation \eqref{master equation} is necessarily a Lasry-Lions monotone solution, the converse does not generally hold. 
This asymmetry stems from the possible lack of uniqueness of Lasry-Lions monotone solutions. 
Indeed, by Lemma \ref{lemma:uniqueness_monotone}, such solutions are unique only up to an additive function 
$c:[0,T]\times\mathcal{P}_2(\mathbb{R}^d)\to\mathbb{R}$ satisfying $c(T,\cdot)=0$. 
In the case of classical solutions, one can further show that $c=0$. 
For the Lasry-Lions monotone solution $V$ introduced in Definition \ref{Larsy-Lion monotone solution definition}, our result guarantees the uniqueness of the spatial gradient $\partial_x V$, which coincides with the vectorial weak solution to the master equation \eqref{master equation}, rather than the uniqueness of the potential function $V$. 
To recover uniqueness of $V$, one may impose a stronger monotonicity condition; see Hypothesis 2 and Theorem 3.1 in \cite{bertucci2023monotone}.

In fact, we emphasize that the knowledge of $\partial_x V$ is central to applications, as the vector field $ \partial_p H(x, \partial_x V(t, x, \mu), \rho)$ provides the optimal feedback control for the MFGC problem. For a detailed discussion, see \cite[Section 3]{cardaliaguet2022monotone}, where the authors formulate all results regarding Lasry-Lions monotone solutions in MFG problems in terms of $\partial_x V$.
\end{remark}

\section*{Funding}
The work of the first author was supported by NSFC grant 12522122, NSFC/RGC JRS N\_CityU165/25. The work of the third author was supported by NSFC grant 12522122, NSFC/RGC JRS N\_CityU165/25, GRF 11311422, and GRF 11303223. The work of the fourth author was supported by the National Natural Science Foundation of China (Nos. 12521001, 62561160159), the National Key Research and Development Program of China (No. 2023YFA1009200), the Fundamental and Interdisciplinary Disciplines Breakthrough Plan of the Ministry of Education of China (No. JYB2025XDXM114), the Shandong Provincial Key Laboratory of Stochastic System Control and Scientific Computing.

\section*{Declarations}

\par \textbf{Conflict of interest} The authors declare that they do not have any conflicts of interests.
\par \noindent\textbf{Data availability} Data sharing not applicable to this article as no datasets were generated or analyzed during
the current study.

\section*{Appendix A. Proof of Lemma \ref{relationship}}
The characterization in Lemma \ref{relationship} is standard in the viscosity solution literature and is stated, in particular, in the work of \cite{ishii1993viscosity}. However, a detailed proof does not seem to be readily available in the references we have consulted. Although this characterization is frequently used without proof, for completeness, we provide the details below in the present Hilbert space setting.
\begin{proof}
Suppose first that there exists $\varphi\in \mathbb{C}^{1,2}([0,T]\times\mathbb H)$ such that $f-\varphi$ attains a local minimum at $(t_0,x_0)$ and
$\big(\partial_t\varphi(t_0,x_0),
D_x\varphi(t_0,x_0),
D_{xx}\varphi(t_0,x_0)\big)
=(q,p,\chi).$ Adding a constant to $\varphi$ if necessary, we may assume that
$\varphi(t_0,x_0)=f(t_0,x_0).$
Then there exists $\delta>0$ such that for any $(t,x)\in[0,T]\times\mathbb H$ with $|t-t_0|+\|x-x_0\|_2<\delta$, one has
$$f(t,x)-f(t_0,x_0)
\geq
\varphi(t,x)-\varphi(t_0,x_0).$$
Since $\varphi\in \mathbb{C}^{1,2}([0,T]\times\mathbb H)$, we have the expansion
\[
\begin{aligned}
\varphi(t,x)-\varphi(t_0,x_0)
=q(t-t_0)+\langle p,x-x_0\rangle_{\mathbb{H}}
+\frac{1}{2}\langle\chi(x-x_0),x-x_0\rangle_{\mathbb{H}}+o\big(|t-t_0|+\|x-x_0\|_2^2\big),
\end{aligned}
\]
as $(t,x)\to(t_0,x_0)$. 
Consequently, for any $(t,x)\in[0,T]\times\mathbb H$ with $|t-t_0|+\|x-x_0\|_2<\delta$,
\[
\begin{aligned}
f(t,x)-f(t_0,x_0)
-q(t-t_0)-\langle p,x-x_0\rangle_{\mathbb{H}}
-\frac{1}{2}\langle\chi(x-x_0),x-x_0\rangle_{\mathbb{H}}\geq o\big(|t-t_0|+\|x-x_0\|_2^2\big),
\end{aligned}
\]
After division by $|t-t_0|+\|x-x_0\|_2^2$ and taking the lower limit,
\[
\begin{aligned}
\liminf_{|t-t_0|+\|x-x_0\|_2^2\rightarrow 0} \frac{f(t,x)-f(t_0,x_0)
-q(t-t_0)-\langle p,x-x_0\rangle_{\mathbb{H}}
-\frac{1}{2}\langle\chi(x-x_0),x-x_0\rangle_{\mathbb{H}}}{|t-t_0|+\|x-x_0\|_2^2}\geq 0,
\end{aligned}
\]
Thus,
$(q,p,\chi)\in D^{1,2,-}f(t_0,x_0).$
\par Conversely, suppose that
$(q,p,\chi)\in D^{1,2,-}f(t_0,x_0).$
Set
$s=t-t_0$, $y=x-x_0,$
and define
$$R(s,y)
:=
f(t_0+s,x_0+y)-f(t_0,x_0)
-qs-\langle p,y\rangle_{\mathbb{H}}
-\frac12\langle\chi y,y\rangle_{\mathbb{H}}.$$
By the definition of $D^{1,2,-}f(t_0,x_0)$,
$$\liminf_{\substack{s\to0,\ t_0+s\in[0,T],\\ y\to0}}
\frac{R(s,y)}
{|s|+\|y\|_2^2}
\geq0.$$
Define
$d(s,y):=\left(s^2+\|y\|_2^4\right)^{\frac{1}{2}}, $
Since $d(s,y)\leq |s|+\|y\|_2^2\leq\sqrt{2}d(s,y),$ we also have
$$\liminf_{\substack{s\to0,\ t_0+s\in[0,T],\\ y\to0}}
\frac{R(s,y)}
{d(s,y)}
\geq0.$$
For $r>0$ sufficiently small, define
$$w(r)=\max\left\{0,\sup_{\substack{0<d(s,y)\leq r,\\ t_0+s\in[0,T]}}
\frac{-R(s,y)}
{d(s,y)}\right\}.$$
Then $w$ is nonnegative and nondecreasing, and
$w(r)\rightarrow 0$, as $r\rightarrow 0$. Moreover, for $(s,y)$ sufficiently close to $(0,0)$,
$$R(s,y)\geq -w\left(d(s,y)\right)d(s,y).$$
By the standard one-dimensional smoothing construction used in the proof of the corresponding finite-dimensional result in \cite[Lemma 4.1]{FlemingSoner}, there exists a nonnegative function
$g:[0,\infty)\to[0,\infty)$ such that, for $r>0$ sufficiently small,
$g(r)\geq rw(r),\ g(0)=0,\ \lim_{r\rightarrow 0}g'(r)=0,\ \lim_{r\rightarrow 0}rg''(r)= 0,$
and $\tilde{g}:=g\left((s^2+\|y\|^4_2)^{\frac{1}{2}}
\right)$ is $\mathbb{C}^{1,2}$ in a neighborhood of $(0,0)$, with
$\partial_s\tilde{g}\left(0,0
\right)=0,\ D_y\tilde{g}\left(0,0
\right)=0,\ D_{yy}\tilde{g}\left(0,0
\right)=0.$
The scalar construction of $g$ is the same as in the finite-dimensional case. The corresponding derivative computations remain valid in the present Hilbert space setting.
Define
$$\begin{aligned}
\varphi(t,x)
:=&
f(t_0,x_0)
+q(t-t_0)
+\langle p,x-x_0\rangle_{\mathbb{H}}
+\frac12\langle\chi(x-x_0),x-x_0\rangle_{\mathbb{H}}\\
&-
g\left(((t-t_0)^2+\|x-x_0\|_2^4)^{\frac{1}{2}}
\right).
\end{aligned}$$
After extending $g$ outside a sufficiently small neighborhood of $0$, if necessary, without changing its properties near $0$, we may assume that $\varphi\in \mathbb{C}^{1,2}([0,T]\times \mathbb{H})$. By the properties of $\tilde{g}$,
$$\big(
\partial_t\varphi(t_0,x_0),
D_x\varphi(t_0,x_0),
D_{xx}\varphi(t_0,x_0)
\big)
=(q,p,\chi).$$
Moreover,
$$\begin{aligned}
f(t_0+s,x_0+y)-\varphi(t_0+s,x_0+y)
=&
R(s,y)+g\left(d(s,y)
\right)\\
\geq&
-w\left(d(s,y)\right)d(s,y)
+g\left(d(s,y)\right)
\geq0
\end{aligned}$$
for $(s,y)$ sufficiently close to $(0,0)$, while
$f(t_0,x_0)-\varphi(t_0,x_0)=0.$
Thus $f-\varphi$ attains a local minimum at $(t_0,x_0)$.
This completes the proof.
\end{proof}
\bibliographystyle{plain}
\bibliography{main}

\end{document}